\documentclass[a4paper,10pt,oneside]{scrartcl}
 
\usepackage[a4paper, top=2cm, bottom=2cm, left=2cm, right=2cm, heightrounded, centering]{geometry}
 
\usepackage[T1]{fontenc}
\usepackage[utf8]{inputenc}
\usepackage{microtype}
\usepackage{lmodern}
\usepackage{authblk}
\usepackage{placeins}
\usepackage{afterpage}
\usepackage{setspace}
\usepackage{comment}
\usepackage{amsmath, amsfonts, amssymb, amsthm, mathtools}
\usepackage{stmaryrd}
 
\usepackage{graphicx}
\usepackage{tikz}
\usetikzlibrary{arrows.meta, positioning, cd}
\usepackage{adjustbox}
\usepackage{caption}
\usetikzlibrary{calc}

\usepackage{booktabs}
\usepackage{siunitx,array,xcolor,colortbl,longtable}
 \usepackage[
  colorlinks=true,
  linkcolor=purple!80!black,
  citecolor=purple!80!black,
  urlcolor=blue!40!black,
  linktocpage=true
]{hyperref}
\usepackage{cleveref}
\usetikzlibrary{arrows.meta}

\definecolor{flagblue}{HTML}{315B78}
\definecolor{contactpurple}{HTML}{715780}
\definecolor{collapseorange}{HTML}{AD6335}
 \newtheoremstyle{mystyle}
  {12pt}{12pt}
  {\itshape}
  {}
  {\bfseries}
  {.}
  {.5em}
  {}
\DeclareFontFamily{U}{mathx}{\hyphenchar\font45}
\DeclareFontShape{U}{mathx}{m}{n}{
      <5> <6> <7> <8> <9> <10>
      <10.95> <12> <14.4> <17.28> <20.74> <24.88>
      mathx10
      }{}
\DeclareSymbolFont{mathx}{U}{mathx}{m}{n}
\DeclareFontSubstitution{U}{mathx}{m}{n}
\DeclareMathAccent{\widecheck}{0}{mathx}{"71}
\DeclareMathAccent{\wideparen}{0}{mathx}{"75}

\theoremstyle{mystyle}
\newtheorem{thm}{Theorem}[section]
\newtheorem{prop}[thm]{Proposition}
\newtheorem{lemma}[thm]{Lemma}

\theoremstyle{definition}
\newtheorem{defn}[thm]{Definition}

\theoremstyle{remark}
\newtheorem{rmk}[thm]{Remark}

 \newcommand{\Span}{\mathrm{Span}}
\newcommand{\mr}[1]{\mathrm{#1}}
\newtheoremstyle{note}
  {12pt}{12pt}{}
  {}{\bfseries}{.}{.5em}{}
\theoremstyle{note}

\allowdisplaybreaks[4]
\vfill

\author[1]{Antonio Del Donno\thanks{antonio.deldonno@matfyz.cuni.cz}}

\affil[1]{Mathematical Institute of Charles University 

Sokolovská 49/83, 18675 Prague 8, Czech Republic.}

\title{A Dolbeault complex for \\  the contact quantised flag manifolds \\ of type $\mathrm{C}_{2}$ and $\mathrm{D}_{4}$ }
\begin{document}
\maketitle

\begin{abstract}
\noindent Differential calculi on non-irreducible quantised flag manifolds of the $\mathrm{A}$-series have recently been investigated using Lusztig positive root vectors, revealing that several features of classical differential geometry persist in the noncommutative setting. We extend this investigation to quantised flag manifolds associated with the lowest-rank quantum groups of types $\mathrm{C}$ and $\mathrm{D}$ whose classical spaces are holomorphic contact manifolds. Building on the Heckenberger--Kolb framework, we construct covariant first order differential calculi. The maximal prolongations of calculi restricted to contact directions have lower dimensions than their classical counterparts. In type $\mathrm{C}_2$, however, adjoining the Reeb direction yields a three-dimensional calculus on the contact flag manifold whose maximal prolongation has classical dimension. Adjoining the remaining simple-root direction yields a four-dimensional calculus on the full flag manifold, again with a maximal prolongation of classical dimension. For both calculi, we explicitly present the exterior algebras and identify them with those of the corresponding framing calculi on the quantum group. In type $\mathrm{D}_4$, we show that adjoining the Reeb direction yields a nine-dimensional calculus on the contact flag. We conjecture that its maximal prolongation also has classical dimension.
\end{abstract}

\bigskip


\tableofcontents

\section*{Introduction}
The theory of quantum groups, initiated in the mid-1980s by the foundational work of Drinfel'd \cite{Drinfeld} and Jimbo \cite{Jim85} on quantum universal enveloping algebras, and by Woronowicz's parallel development of compact quantum groups via the formalism of $C^*$-algebras \cite{Wor88}, has by now reached considerable maturity from a purely algebraic and representation-theoretic standpoint. In parallel, and largely independently, the framework of noncommutative differential geometry, as systematised by Connes \cite{Con90} and, in the Hopf-algebraic setting, by Woronowicz \cite{Woronowicz1989} and by Majid \cite{Maj88}, provides a general apparatus for extending differential calculus to noncommutative algebras. The natural problem of reconciling these two traditions, that is, of equipping the coordinate algebras of quantum groups with geometrically meaningful, covariant differential calculi, has proven to be substantially more difficult than one might initially expect. Indeed, a fundamental obstruction already appears at the level of the simplest quantum group $\mathcal{O}_q(\mathrm{SU}_2)$: Woronowicz's three-dimensional calculus fails to be bicovariant \cite{Woronowicz1989}. This behavior is not confined to the particular construction, as subsequent classification results establish the nonexistence of classical dimension bicovariant calculi for general quantum groups \cite{SchmudgenSchuler1995,HeckenbergerSchmudgen}. Nevertheless, rich and geometrically meaningful covariant calculi of classical dimension do exist. The paradigmatic example is again the three-dimensional left-covariant calculus on $\mathcal{O}_q(\mathrm{SU}_2)$ \cite{Wor87}, which restricts to the celebrated two-dimensional calculus on the Pod\-le\'{s} sphere \cite{Pod89}, and whose behavior closely mirrors that of the classical geometry of $\mathrm{SU}_2$ and $S^2$ respectively. In view of these difficulties at the level of the quantum group itself, a more tractable approach is to shift attention to quantum homogeneous spaces, and in particular to the quantised flag manifolds. This strategy is further motivated by the classical fibration $G \to G/P$, which expresses a semisimple Lie group in terms of its flag varieties, so that a systematic understanding of the differential geometry of the latter provides essential structural information about the former. 

A covariant first order differential calculus on a quantum homogeneous space must be compatible with the coaction of the underlying quantum group. The general classification of such calculi is well understood \cite{Hermisson1},  and follows essentially from the equivalence of categories established in \cite{Takeuchi79}. However, singling out those covariant differential structures that carry meaningful geometric content remains a substantially more delicate problem. Significant progress has nonetheless been made in the setting of the \emph{irreducible} quantised flag manifolds, that is, the $q$-deformations of those classical flag varieties that are of compact Hermitian symmetric type. The foundational work of Heckenberger and Kolb \cite{hek-kolb-2003} established a complete classification of all covariant first order differential calculi of classical dimension over these spaces, and subsequently investigated their maximal prolongations and the resulting higher-order differential geometry \cite{hek-kolb-2006}. This family includes, as its most studied members, the quantum projective spaces $\mathbb{CP}^n_q$ and the quantum Grassmannians of the $\mathrm{A}$-series, which have been the subject of intensive investigation by numerous authors \cite{Mat22,CMO23,DDL08}. 

More recently, progress has been made towards the understanding of quantised flag manifolds in the non-irreducible setting. The full flag manifolds of the $\mathrm{A}$-series were studied in full generality by \'{O} Buachalla and Somberg in \cite{Re-Sombrg}, who demonstrated that their noncommutative geometry is remarkably faithful to the classical picture, in particular the exterior algebra of the maximal prolongation recovers the Dolbeault complex, and a noncommutative version of the Borel--Weil theorem holds. Moreover, restricting such calculi to the corresponding contact quantised flag manifold, namely the one associated with the parabolic subset obtained by removing the first and last simple roots, is shown to yield a differential calculus of classical dimension \cite{RéGio}. Along the same lines, in \cite{ale_re_jun} the authors carry out a detailed study of the noncommutative complex geometry of the full flag manifold of $\mathcal{O}_q(\mathrm{SU}_3)$. A similar case study is \cite{VY15}, where a quantum analogue of the BGG resolution for the full quantised flag manifold of $\mathcal{O}_q(\mathrm{SU}_3)$ is employed to verify the Baum--Connes conjecture with trivial coefficients for the dual discrete quantum group. Matassa \cite{Matassa2025} provides a complementary construction of covariant differential calculi on quantised flag manifolds through equivariant quantisations of the positive nilradical, recovering the Heckenberger--Kolb calculi in the irreducible case. Beyond the $\mathrm{A}$-series, the irreducible rank-two setting has also been studied from a spectral perspective: in \cite{WDO22}, the quantum BGG resolution is used to construct a Dolbeault--Dirac operator for the $\mathrm{B}_2$-irreducible quantised flag manifold, yielding a spectral triple. More directly related to the non-irreducible setting considered here, Carotenuto, \'{O} Buachalla, and Pagani \cite{COP} investigate a Lusztig differential calculus on the type $\mathrm{C}_2$ quantised flag manifolds.

Despite these developments, the covariant differential geometry of non-irreducible quantised flag manifolds remains comparatively little explored, particularly outside the $\mathrm{A}$-series. These spaces are $q$-deformations of partial and full flag varieties whose classical counterparts are \emph{not} of Hermitian symmetric type. The geometric and algebraic complications arising beyond the Hermitian symmetric setting are substantial even classically, while quantum deformation introduces additional constraints on the construction of differential calculi and their higher-order extensions. The existing examples therefore motivate a broader investigation of how much of the classical differential geometry survives quantisation, especially for non-irreducible flag manifolds in other Lie types.
	
The present paper constitutes a first step in this direction. We consider the quantisations of those non-irreducible flag manifolds that carry a holomorphic contact structure \cite{parabook} in the lowest-rank examples of the $\mathrm{C}$-series and $\mathrm{D}$-series of complex simple Lie algebras, namely those associated to $\mathfrak{sp}_4$ and $\mathfrak{so}_8$, respectively (and, by the exceptional isomorphism $B_2 \cong C_2$, the $\mathrm{B}$-series as well in the rank-two case). Following the general strategy of Heckenberger--Kolb \cite{hek-kolb-2006}, we construct quantum tangent spaces corresponding to the antiholomorphic contact directions and use them to produce covariant first order differential calculi. We then study their maximal prolongations, comparing the graded dimensions of the resulting exterior algebras with those of the classical exterior algebras on the corresponding cotangent spaces. Additionally, for the case of $\mathrm{D}_4$, we construct a differential calculus directly at the level of the quantised coordinate algebra $\mathcal{O}_q(\mathrm{SO}_8)$, induced from an antiholomorphic tangent space generated by the Lusztig root vectors associated to a distinguished reduced decomposition of the longest element of the Weyl group.

Already for such low-rank examples, directly adapting the Heckenberger--Kolb construction to the contact directions produces a picture different from that of the previously studied $\mathrm{A}$-series calculi: the exterior algebras of the maximal prolongations have graded dimensions strictly smaller than their classical counterparts in certain degrees. The same phenomenon occurs for the calculus considered on $\mathcal{O}_q(\mathrm{SO}_8)$. These examples highlight the need for a careful, case-by-case investigation when extending constructions beyond the irreducible setting, since the recovery of classical dimension cannot be taken for granted.

In type $\mathrm{C}_2$, however, the dimension drop can be overcome by adjoining the tangent direction corresponding classically to the Reeb direction. The resulting three-dimensional first order calculus has a maximal prolongation of classical dimension. We give an explicit presentation of its exterior algebra and establish that the natural identification of the cotangent spaces of the flag calculus and its framing calculus extends to an isomorphism of their exterior algebras. Moreover, adjoining the remaining simple root vector to the corresponding tangent space yields a four-dimensional calculus on $\mathcal{O}_{q}(\mathrm{Sp}_{4})$ whose maximal prolongation also has classical dimension. The constructions are compatible with restriction  to contact and the full quantised flag manifolds respectively, and classical dimension is preserved.

A related four-dimensional calculus with a maximal prolongation of classical dimension was constructed in \cite{COP} using Lusztig root vectors. That construction, however, does not admit a restrictions to the contact flag manifold that is of classical dimension. The calculi here analysed, on the other hand, combine the recovery of classical dimension with compatibility with the quantum homogeneous space structure, in particular with both the contact and full flag manifolds. This compatibility provides a geometric distinction between the constructions and makes the calculi of interest natural candidates for studying the noncommutative differential geometry of type $\mathrm{C}_2$ quantised flag manifolds.

At the same time, our computations indicate that the dimension drop for calculi built from the contact directions alone should not be understood as a peculiarity of the Lie types outside the $\mathrm{A}$-series. Indeed, when the analogous construction is carried out for the full flag manifold of $\mathcal{O}_q(\mathrm{SU}_3)$, whose underlying classical variety is also of contact type, a comparable collapse of the exterior algebra occurs. Together with the extended $\mathrm{C}_2$ calculus, this suggests that the recovery of classical dimension depends on how the contact structure is incorporated into the quantum tangent space, rather than on Lie type alone.

The results presented here suggest several directions for further investigation. A forthcoming study will carry out a similar analysis in type $\mathrm{D}_4$, examining whether the Reeb adjoined calculus that we discuss in this paper has a maximal prolongation of classical dimension. In the contact setting, a natural next step is to explore whether the differential calculus can be adapted to construct a quantum analogue of the Rumin complex \cite{Rumin94}, reflecting the contact structure more closely than the full de Rham complex. Exploring the relationship between differential forms and quantum BGG resolutions \cite{HK07}, together with the full flag example studied in \cite{VY15}, is another promising direction of research. Finally, extending the present analysis to higher-rank contact quantised flag manifolds would provide a broader testing ground for the phenomena identified here.

\medskip

 In Section~\ref{section:preliminaries} we recall the necessary preliminary material on Hopf algebras and their homogeneous spaces, quantum groups, quantised flag manifolds, and differential calculi. 
 
  In Section~\ref{section:C2} we review the basic structure of the Lie algebra $\mathfrak{sp}_4$, together with its quantised enveloping algebra and quantised coordinate algebra, and explicitly construct a two-dimensional first order differential calculus on the contact quantised flag manifold of $\mathcal{O}_q(\mathrm{Sp}_4)$, computing its maximal prolongation. We then briefly compare this construction with the analogous one on the full flag manifold of $\mathcal{O}_q(\mathrm{SU}_3)$. Returning to type $\mathrm{C}_2$, we show that adjoining the Reeb direction yields a three-dimensional calculus whose maximal prolongation has classical dimension. We subsequently extend this calculus by adding the simple root vector $E_2$, obtaining a four-dimensional calculus with the same property. For both calculi, we give explicit presentations of the exterior algebras and establish compatibility with restriction to the corresponding quantised flag manifolds, where classical dimension is recovered.

In Section~\ref{section:lusztig_for_so(8)}, we recall the relevant algebraic structures and construct an eight-dimensional first order differential calculus on the contact quantised flag manifold of $\mathcal{O}_{q}(\mathrm{SO}_{8})$. We compute its maximal prolongation and show that it has non-classical dimension. We then extend this calculus by adjoining the Reeb direction, obtaining a nine-dimensional calculus, whose corresponding maximal prolongation remains undetermined. Finally, we discuss how Lusztig calculi do not supply geometrically meaningful framings for the Reeb-extended calculus.

\begin{figure}[p]
\centering
\begin{tikzcd}[
  column sep=10mm,
  row sep=16mm,
  arrows={
    -{Stealth[length=2mm,width=1.4mm]},
    line width=0.65pt
  },
  cells={nodes={
    draw,
    rounded corners=3pt,
    line width=0.6pt,
    inner xsep=5pt,
    inner ysep=8pt,
    minimum height=16mm,
    text width=0.21\textwidth,
    align=center,
    font=\small
  }},
  /tikz/bluebox/.style={
    draw=flagblue!55,
    fill=flagblue!5,
    text=flagblue
  },
  /tikz/purplebox/.style={
    draw=contactpurple!55,
    fill=contactpurple!5,
    text=contactpurple
  },
  /tikz/orangebox/.style={
    draw=collapseorange!65,
    fill=collapseorange!8,
    text=collapseorange
  },
  /tikz/conjecturebox/.style={
    bluebox,
    densely dashed
  },
  /tikz/comparison/.style={
    draw=black!50,
    densely dashed,
    line width=0.5pt
  },
  /tikz/toparc/.style={
    to path={
      (\tikztostart.north east)
      .. controls
        ([yshift=4mm]\tikztostart.north east)
        and
        ([yshift=4mm]\tikztotarget.north west)
      .. (\tikztotarget.north west)
      \tikztonodes
    }
  }
]
|[bluebox]|
\begin{gathered}
  \text{Hermitian symmetric}\\
  \text{flag manifolds}
\end{gathered}
\arrow[r, toparc, draw=flagblue,
  "\text{\scriptsize quantise}"{text=flagblue}]
\arrow[d, comparison,
  "\text{\scriptsize generalise}"{text=black!65}]
\arrow[ddddd, comparison,
  "\text{\scriptsize generalise}"'
    {text=black!65, sloped},
  to path={
    (\tikztostart.west)
    .. controls
      ([xshift=-22mm]\tikztostart.west)
      and
      ([xshift=-22mm]\tikztotarget.west)
    .. (\tikztotarget.west)
    \tikztonodes
  }]
&
|[bluebox]|
\begin{gathered}
  \text{Heckenberger--Kolb}\\
  \text{differential calculi}
\end{gathered}
\arrow[r, toparc, draw=flagblue,
  "\text{\scriptsize prolong}"{text=flagblue}]
&
|[bluebox]|
\begin{gathered}
  q\text{-deformed}\\
  \text{de Rham complexes}
\end{gathered}
\\
|[purplebox]|
\begin{gathered}
  \text{Holomorphic contact}\\
  \text{flag manifold}\\
  \text{of type }\mathrm C_2
\end{gathered}
\arrow[r, toparc, draw=contactpurple,
  "\text{\scriptsize quantise}"{text=contactpurple}]
&
|[purplebox]|
\begin{gathered}
  2\text{-dimensional calculus}\\
  \text{on the contact}\\
  \text{directions}
\end{gathered}
\arrow[r, toparc, draw=collapseorange,
  "\text{\scriptsize prolong}"{text=collapseorange}]
\arrow[d, draw=contactpurple,
  "\substack{\text{\scriptsize include the}\\
             \text{\scriptsize Reeb }}"
  {text=contactpurple}]
&
|[orangebox]|
\begin{gathered}
  \text{Maximal prolongation}\\
  \text{\bfseries with dimension drop}
\end{gathered}
\\
&
|[purplebox]|
\begin{gathered}
  3\text{-dimensional calculus}\\
  \text{including the}\\
  \text{Reeb direction}
\end{gathered}
\arrow[r, toparc, draw=flagblue,
  "\text{\scriptsize prolong}"{text=flagblue}]
\arrow[d, draw=contactpurple,
  "\text{\scriptsize adjoin }E_2"{text=contactpurple}]
&
|[bluebox]|
\begin{gathered}
  \text{\bfseries classical dimension}\\
  \text{on the quantum group}\\
  \text{and the contact flag}
\end{gathered}
\\
&
|[purplebox]|
\begin{gathered}
  4\text{-dimensional calculus}\\
  \text{corresponding to}\\
  \text{the full flag}
\end{gathered}
\arrow[r, toparc, draw=flagblue,
  "\text{\scriptsize prolong}"{text=flagblue}]
&
|[bluebox]|
\begin{gathered}
  \text{\bfseries classical dimension}\\
  \text{on the quantum group}\\
  \text{and the full flag}
\end{gathered}
\\[3mm]
|[purplebox]|
\begin{gathered}
  \text{Quantum group}\\
  \mathcal O_q(\mathrm{SO}_8)
\end{gathered}
\arrow[r, toparc, draw=contactpurple,
  "\substack{\text{\scriptsize Lusztig}\\
             \text{\scriptsize construction}}"
  {text=contactpurple}]
\arrow[d, draw=contactpurple,
  "\text{\scriptsize pass to flag}"'{text=contactpurple}]
&
|[purplebox]|
\begin{gathered}
  12\text{-dimensional}\\
  \text{Lusztig calculus}
\end{gathered}
\arrow[r, toparc, draw=collapseorange,
  "\substack{\text{\scriptsize restriction}\\
             \text{\scriptsize fails}}"
  {text=collapseorange}]
\arrow[d, draw=contactpurple,
  "\substack{\text{\scriptsize contact direction}\\
             \text{\scriptsize restriction}}"
  {text=contactpurple}]
&
|[orangebox]|
\begin{gathered}
  1\text{-dimensional}\\
  \text{irreducible Levi module}\\
  \text{(highest-root direction)}
\end{gathered}
\\
|[purplebox]|
\begin{gathered}
  \text{Holomorphic contact}\\
  \text{flag manifold}\\
  \text{of type }\mathrm D_4
\end{gathered}
\arrow[r, toparc, draw=contactpurple,
  "\substack{\text{\scriptsize Levi}\\
             \text{\scriptsize construction}}"
  {text=contactpurple}]
&
|[purplebox]|
\begin{gathered}
  8\text{-dimensional calculus:}\\
  \text{Levi construction}\\[-2pt]
  {}=\\
  \text{ restricted Lusztig}
\end{gathered}
\arrow[r, toparc, draw=collapseorange,
  "\text{\scriptsize prolong}"{text=collapseorange}]
\arrow[d, draw=contactpurple,
  "\substack{\text{\scriptsize include the}\\
             \text{\scriptsize Reeb }}"
  {text=contactpurple}]
&
|[orangebox]|
\begin{gathered}
  \text{Maximal prolongation}\\
  \text{\bfseries with dimension drop}
\end{gathered}
\\
&
|[purplebox]|
\begin{gathered}
  9\text{-dimensional calculus}\\
  \text{on the contact flag}\\
  \text{including Reeb}
\end{gathered}
\arrow[r, toparc, draw=flagblue, densely dashed,
  "\text{\scriptsize prolong}"{text=flagblue}]
&
|[conjecturebox]|
\begin{gathered}
  \text{Maximal prolongation}\\
  \text{of classical dimension}\\
  \text{\itshape conjectured}
\end{gathered}
\end{tikzcd}

\caption{Roadmap of the constructions studied in this paper.}
\label{figure:roadmap}
\end{figure}
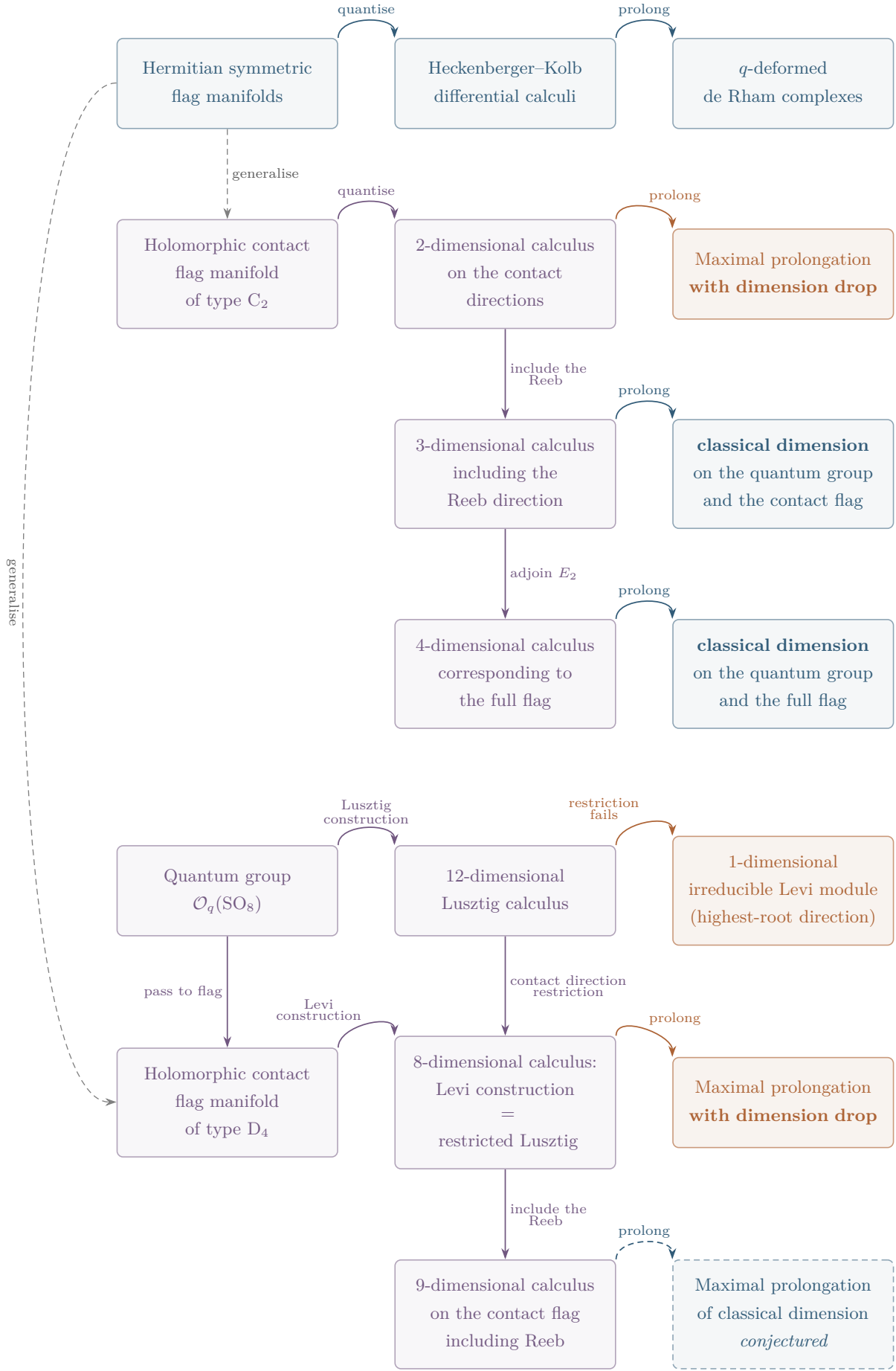

\section{Preliminaries}
\label{section:preliminaries}
In this Section, we collect the algebraic and geometric background required for the developments in the subsequent sections. Our aim here is not to provide a comprehensive account of the material, but rather to fix notation, recall the main constructions, and highlight those aspects of the theory that will play a central role in what follows. In particular, we review the Hopf-algebraic framework underlying quantum homogeneous spaces, the Drinfeld--Jimbo quantum groups and their associated quantised flag manifolds, as well as the theory of differential calculi over Hopf algebras and homogeneous spaces. These ingredients provide the natural setting for the general constructions developed later.

In Section \ref{section:hopf_modules} we recall basic notions concerning Hopf modules, quantum homogeneous spaces, and the categorical equivalences that will be used throughout the paper. In Section \ref{subsection:DJ_quantum_groups_and_flags} we review Drinfeld--Jimbo quantum groups and the construction of quantised flag manifolds. Finally, in Section \ref{subsection:basics_on_dcs} we discuss differential calculi.

\subsection{Hopf modules, homogeneous spaces and some categorical equivalences} 
\label{section:hopf_modules}
Let $H$ be a Hopf algebra, and consider the category ${}^{H}_{H}\mathrm{Mod}$, whose objects are left $H$-covariant left $H$-modules, that is modules $\mathcal{M}$ such that given the left $H$-coaction ${}_{H}\Delta\colon \mathcal{M}\to H\otimes \mathcal{M}$ and left action $\cdot \colon H\otimes \mathcal{M} \to \mathcal{M}$, the property 
\[ 
{}_{H}\Delta(h\cdot m) = \Delta(h){}_{H}\Delta(m) 
\] 
holds for all $h\in H$ and $m\in \mathcal{M}$, and whose morphisms are left $H$-linear, left $H$-colinear maps. The latter is known as the category of left-left $H$-Hopf modules and it is denoted by ${}^{H}_{H}\mathrm{Mod}$. The other instances of Hopf module categories are constructed in a similar fashion. Additionally, let $V$ be a $\Bbbk$-vector space. 

The following is a standard result in the representation theory of Hopf algebras, known as the \emph{fundamental theorem of Hopf modules}. It describes the equivalence between the category of $H$-Hopf modules introduced, and the category of $\Bbbk$-vector spaces. 

\begin{thm}[\cite{Sweedler1969}, \emph{fundamental theorem of Hopf modules}]
\label{th:fundamental_theorem_hopf}
There is an equivalence of categories
\medskip
\begin{figure}[!htbp]
\centering	
	\begin{tikzcd}[column sep=large]
^{H}_{H}\mathrm{Mod}\arrow[rr, "\Phi", bend left] & \cong &{}_{\Bbbk}\mathrm{Vec} \arrow[ll, "\Psi", bend left]. 
\end{tikzcd}	

\end{figure}
\end{thm} 

\noindent A similar result holds for the other flavors of Hopf modules, and in each case the equivalence is implemented by an analogous pair of inverse functors. As we will be mostly focusing on the case of left-left Hopf modules throughout this paper, we briefly recall the precise datum of this equivalence in this specific context.

\begin{enumerate}
	\item  Given $\mathcal{M},\mathcal{M}'\in {}^{H}_{H}\mathrm{Mod}$ and ${}^{H}_{H}\mathrm{Mod}\ni\phi\colon \mathcal{M}\to \mathcal{M}'$, with left $H$-coaction ${}_{H}\Delta\colon \mathcal{M} \to H\otimes \mathcal{M}$ and left $H$-action $\cdot \colon H\otimes \mathcal{M}\to \mathcal{M}$, there is a functor 
	
\begin{equation}
\begin{aligned}
	\Phi\colon {}^{H}_{H}\mathrm{Mod} &\to {}_{\Bbbk}\mathrm{Vec}, 	\\
	 \mathcal{M} &\mapsto {}^{\mathrm{co}H}\mathcal{M}, \\ 
	\phi & \mapsto \Phi(\phi)=\phi|_{{}^{\mathrm{co}H}\mathcal{M}}, 
\end{aligned}
\end{equation}

\noindent assigning to objects the space ${}^{\mathrm{co}H}\mathcal{M}:=\{m\in \mathcal{M} \mid {}_{H}\Delta(m)=1\otimes m\}$ of coinvariant elements under the left $H$-coaction, and to each morphism the corresponding restriction to the latter. 
\item Given $V,V'\in {}_{\Bbbk}\mathrm{Vec}$ and $\psi\colon V\to V'$ a $\Bbbk$-linear map, there is a functor 
\begin{equation}\label{eq:free_functor_hopf_modules}
	\begin{aligned}
		\Psi\colon {}_{\Bbbk}\mathrm{Vec} &\to {}^{H}_{H}\mathrm{Mod},  \\ 
		 V &\mapsto H\otimes V, \\ 
		 \psi &\mapsto (\mathrm{id}\otimes \psi). 
	\end{aligned}
\end{equation}
\item The functors $\Phi$ and $\Psi$ provide an  equivalence of categories with unit 
\begin{equation}
	\begin{aligned}
		\Theta\colon \mathcal{M} &\to H \otimes  {}^{\mathrm{co}H}\mathcal{M}, \\
		            m&\mapsto m_{(-2)}\otimes S(m_{(-1)}) m_{(0)}, \\
	\end{aligned}
\end{equation}
and counit 
\begin{equation}
	\begin{aligned}
				\Xi \colon {}^{\mathrm{co}H}(H\otimes V) & \to V, \\ 
		h\otimes v & \mapsto \varepsilon(h) v. 
	\end{aligned}
\end{equation}
\end{enumerate}
In particular, the fundamental theorem of Hopf modules also accounts, as a special case, for the equivalence ${}^{H}_{H}\mathrm{Mod}_{H}\cong \mathrm{Mod}_{H}$ with the same functorial data. 

The next discussion consists of a generalisation of the latter to the case of quantum homogeneous spaces of a Hopf algebra, of which we first recall the two equivalent definitions.

\begin{defn}
	Let $A$ be a Hopf algebra with invertible antipode, and let $B\subseteq A$ be a left coideal subalgebra, namely $\Delta(B)\subseteq A\otimes B\,$.  The coideal subalgebra $B$ is called a \emph{quantum homogeneous space} if $A$ is faithfully flat as a right $B$-module and the condition $B^{+}A=AB^{+}$ holds. 
\end{defn}
\noindent We may provide an analogous notion by considering right coideal subalgebras, where now $A$ is understood as a faithfully flat left $B$-module. 

The above definition may be equivalently rephrased considering a Hopf algebra surjection $\pi_{B}\colon A \to A/B^{+}A$. In particular, consider the associated right coaction of the \emph{isotropy quantum group} $\pi_{B}(A)$ of $B$, given as
\begin{equation}
	\begin{aligned}
		\Delta_{\pi_{B}}\colon A \to A\otimes \pi_{B}(A), \qquad a \mapsto (\mathrm{id}_{A}\otimes \pi_{B})\circ \Delta(a). 
	\end{aligned}
\end{equation}
As shown in \cite{Takeuchi79}, the space of right $\pi_{B}(A)$ coinvariant elements 
\[ 
A^{\mathrm{co}(\pi_{B}(A))}:=\Big\{b\in A \mid \Delta_{\pi_{B}}(b)=b\otimes 1\Big\}
\]
is precisely the coideal subalgebra $B$. Accordingly, we may think of quantum homogeneous spaces as coinvariant elements under the coaction of the isotropy subalgebra on a total space Hopf algebra. 

Fix a quantum homogeneous space $B:=A^{\mathrm{co}(\pi_{B}(A))}$. Consider the category ${}_{B}^{A}\mathrm{Mod}_{B}$ of $(A,B)$-relative Hopf modules, whose objects are left $A$-covariant $B$-bimodules, \emph{i.e.} left $A$-comodules $\mathcal{M}$ with a $B$-bimodule structure such that given the left $A$-coaction ${}_{A}\Delta\colon \mathcal{M} \to A\otimes \mathcal{M}$ we have 
	\[
	{}_{A}\Delta(b\cdot m\cdot b') = \Delta(b){}_{A}\Delta(m)\Delta(b'), 
	\]
	 and whose morphisms are left $A$-colinear, $B$-bilinear maps.  Consider  the category ${}^{\pi_{B}(A)}\mathrm{Mod}_{B}$ whose objects are left $\pi_{B}(A)$-covariant right $B$-modules, \emph{i.e.} left $\pi_{B}(A)$-comodules with a right $B$-action such that, letting ${}_{\pi_{B}(A)}\Delta\colon \mathcal{M}\to \pi_{B}(A)\otimes \mathcal{M}$ we have 
	\[
	{}_{\pi_{B}(A)}\Delta(m\cdot b) =  {}_{\pi_{B}(A)}\Delta(m)\Delta(b), 
	\]
	 and whose morphisms are left $\pi_{B}(A)$-colinear, right $B$-linear maps.
	 
 We now briefly recall the datum of \emph{Takeuchi's  equivalence} between the categories of modules aforementioned.
 
 \smallskip
 
\begin{thm}[\cite{Takeuchi79}, \emph{Takeuchi's equivalence}]
\label{thm:takeuchi_equivalence}
	Let $B\subseteq A$ be a quantum homogeneous space and consider  $\pi_{B}(A):=A/B^{+}A$. We have the equivalence of categories 
	\bigskip
	\begin{figure}[!htbp]
	\centering
	\begin{tikzcd}
{}^{A}_{B}\mathrm{Mod}_{B} \arrow[rr, "\Phi", bend left] & \cong & {}^{\pi_{B}(A)}\mathrm{Mod}_{B} \arrow[ll, "\Psi", bend left].
\end{tikzcd}	
	\end{figure}
\end{thm}
\noindent In particular:
	 \begin{enumerate}
	 	\item Given $\mathcal{M},\mathcal{M}'\in {}^{A}_{B}\mathrm{Mod}_B$ and ${}^{A}_{B}\mathrm{Mod}_{B}\ni \phi\colon \mathcal{M}\to \mathcal{M}'$, there is a functor 
	 	\begin{equation}
	 		\begin{aligned}
	 			\Phi\colon {}^{A}_{B}\mathrm{Mod}_{B} & \to {}^{\pi_{B}(A)}\mathrm{Mod}_{B}, \\ 
	 			 \mathcal{M}&\mapsto \mathcal{M}/B^{+}\mathcal{M},\\
	 			 \phi & \mapsto \phi\,\big|_{\mathcal{M}/B^{+}\mathcal{M}}, 
	 		\end{aligned}
	 	\end{equation}
	 	where $\mathcal{M}/B^{+}\mathcal{M}$ has the structure of a left $\pi_{B}(A)$-comodule given by ${}_{\pi_{B}}\Delta(m) = \pi_{B}(m_{-1})\otimes [m_{0}]$, with square brackets denoting the coset in $\mathcal{M}/B^{+}\mathcal{M}$, and the obvious structure of a right $B$-module. 
	 	
	 	\item Given $V,V'\in {}^{\pi_{B}(A)}\mathrm{Mod}_{B}$ and ${}^{\pi_{B}(A)}\mathrm{Mod}_{B}\ni\psi\colon V\to V'$, there is a functor 
	 
	 	\begin{equation}\label{eq:free_functor_takeuchi}
	 		\begin{aligned}
	 			\Psi\colon {}^{\pi_{B}(A)}\mathrm{Mod}_{B} & \to {}^{A}_{B}\mathrm{Mod}_{B}, \\ 
	 			V &\mapsto A\,\square_{\pi_{B}}V,\\ 
	 			\psi&\mapsto (\mathrm{id}_{A}\otimes \psi), 
	 		\end{aligned}
	 	\end{equation}
	 	
	 	\noindent where $\square_{\pi_{B}}$ denotes the balanced cotensor product over $\pi_{B}(A)$, the left $A$-comodule structure on $\Psi(V)$ is defined on the first tensor factor, and the right $B$-module structure is the diagonal.
	 	
	 	\item The functors $\Phi$ and $\Psi$ provide an equivalence of categories with unit 
	 	
	 	\begin{equation}
	 		\begin{aligned}
	 			\Theta\colon \mathcal{M} &\to A \,\square_{\pi_{B}}\mathcal{M}/B^{+}\mathcal{M}, \\
	 			m & \mapsto m_{-1}\otimes [m_{0}], 
	 		\end{aligned}
	 	\end{equation}
	 	
	 	\noindent and counit 
	 	
	 	\begin{equation}
	 		\begin{aligned}
	 			\Xi\colon \left(A\,\square_{\pi_{B}}V\right)/\left(B^{+}\left(A\,\square_{\pi_{B}}V\right)\right) & \to V, \\
	 			[a\otimes v] & \mapsto \varepsilon(a)v. 
	 		\end{aligned}
	 	\end{equation}
	 \end{enumerate}
	 The next theorem specifies under which conditions the latter is even a monoidal equivalence.
\begin{thm}[\cite{OB16}]
\label{thm:takeuchi_equivalence_0}
	Let $B\subseteq A$ be a quantum homogeneous space. Let ${}^{A}_{B}\mathrm{Mod}_{0}$ denote the full subcategory of ${}^{A}_{B}\mathrm{Mod}_{B}$ endowed with the right $B$-action $m\cdot b = (m_{(-2)}\, b \, S(m_{-1}))\cdot m_{0}$,  and let ${}^{\pi_{B}(A)}\mathrm{Mod}_{0}$ denote the full subcategory of ${}^{\pi_{B}(A)}\mathrm{Mod}_{B}$ endowed with trivial right action. The equivalence in Theorem \ref{thm:takeuchi_equivalence} restricts to the monoidal equivalence of categories $\left({}^{A}_{B}\mathrm{Mod}_{0},\otimes_{B}\right) \cong \left({}^{\pi_{B}(A)}\mathrm{Mod}_{0},\otimes\right)$.
\end{thm}
	
	\noindent From a differential geometric perspective, Theorem \ref{thm:takeuchi_equivalence} provides a noncommutative analogue of the correspondence between equivariant vector bundles over a homogeneous space $G/H$ and representations of the structure group $H$. From this perspective, $B$-bimodules endowed with compatible $A$-coactions play the role of homogeneous vector bundles, while $\pi_B(A)$-comodules encode their fibrewise symmetry data. 
	 \subsection{Drinfeld--Jimbo quantum groups and quantised flag manifolds}
\label{subsection:DJ_quantum_groups_and_flags}
Let $\mathfrak{g}$ be a complex, semisimple Lie algebra of rank $\ell$, with simple roots $\alpha_{1},\dots,\alpha_{\ell}\in\Pi$. Denote by $\mathcal{P}$ the set of weights of $\mathfrak{g}$ and by $\mathcal{P}^{+}$ the subset of integral dominant weights. Fix $q_{i}=q^{\frac{(\alpha_{i},\alpha_{i})}{2}}$, with $q_{i}$ not a root of unity, and let $a_{ij}$ denote the Cartan matrix of $\mathfrak{g}$. We recall that the \emph{Drinfeld-Jimbo quantised enveloping algebra} $\mathrm{U}_{q}(\mathfrak{g})$ \cite{Drinfeld} is the  algebra on $4\ell$ generators $E_{i},F_{i},K_{i}^{\pm 1}$ subject to the relations
\[
\begin{aligned}
K_i K_j &= K_j K_i,  
&\quad K_i K_i^{-1} &= K_i^{-1} K_i = 1,  \\
K_i E_j &= q_{i}^{a_{ij}} E_j K_i,  
&\quad K_i F_j &= q_{i}^{-a_{ij}} F_j K_i, 
\end{aligned}
\]
\[
E_{i}F_{j}-F_{j}E_{i}= \delta_{ij}\, \frac{K_{i}-K_{i}^{-1}}{q_{i}-q_{i}^{-1}}, 
\]

\noindent alongside the \emph{q-Serre relations}

\begin{equation}
		\sum_{k=0}^{1-a_{ij}}(-1)^{k}\begin{bmatrix}
			1- a_{ij} \\ k
		\end{bmatrix}_{q_{i}}\, E_{i}^{1-a_{ij}-k}E_{j}E_{i}^{k} =0, \quad \sum_{k=0}^{1-a_{ij}}(-1)^{k}\begin{bmatrix}
			1- a_{ij} \\ k
		\end{bmatrix}_{q_{i}}\, F_{i}^{1-a_{ij}-k}F_{j}F_{i}^{k}  = 0 , 
\end{equation}
where 
\begin{equation*} \begin{bmatrix}
	n \\ r
\end{bmatrix}_{q}= \frac{[n]_{q}!}{[r]_{q}![n-r]_{q}!}, \quad [n]_{q}=\frac{q^{n}-q^{-n}}{q-q^{-1}}.  
\end{equation*} 
The algebra $\mathrm{U}_{q}(\mathfrak{g})$ has a Hopf algebra structure that is essentially a unique deformation of the Hopf algebra structure on the enveloping algebra $\mathrm{U}(\mathfrak{g})$, and it is given on the generators by \cite{Kli-Sch}
\medskip
\[
\begin{aligned}
		\Delta(E_{i}) &= E_{i}\otimes K_{i} + 1\otimes E_{i},   &     \quad \varepsilon(E_{i})& =0,    &  \quad   S(E_{i}) & =-E_{i}K_{i}^{-1}, \\ 
		\Delta(F_{i}) &= F_{i}\otimes 1 + K_{i}^{-1}\otimes F_{i},     &   \varepsilon(F_{i})& =0,    &   S(F_{i}) & =-K_{i}F_{i}, \\ 
		\Delta(K_{i})&=K_{i}\otimes K_{i},    &   \varepsilon(K_{i})& =1,   &   S(K_{i})&=K_{i}^{-1}. 
\end{aligned}
\]

 A finite-dimensional $\mathrm{U}_{q}(\mathfrak{g})$-module $V$ is said to be of \emph{type-$1$} if it admits a weight-space decomposition $V=\bigoplus_{\mu\in\mathcal{P}}V_{\mu}$,
where the \emph{$\mu$-weight space} is defined by
\[
V_{\mu}:= \left\{ v\in V \mid K_i v=q_i^{\langle\mu,\alpha_i^\vee\rangle}v, \quad i=1,\ldots,\ell \right\}.
\]
The full subcategory of ${}_{\mathrm{U}_{q}(\mathfrak{g})}\mathrm{Mod}$ consisting of finite-dimensional type-$1$ modules is semisimple, see \emph{e.g.} \cite{Lusz88} or \cite{Kli-Sch}. In particular, every such module decomposes as a direct sum of finite-dimensional irreducible modules. Highest and lowest weight vectors are defined analogously to the classical case. In particular, a nonzero vector $v\in V$ is said to be a highest weight vector if it is a weight vector annihilated by all the $E_i$-generators. For every dominant integral weight $\mu\in\mathcal{P}^{+}$, there exists, up to isomorphism, a unique finite-dimensional irreducible type-$1$ representation $V_{\mu}$ of highest weight $\mu$. Moreover, its character, and hence its dimension, coincides with that of the corresponding irreducible $\mathfrak{g}$-representation. Consequently, the finite-dimensional irreducible type-$1$ representations of $\mathrm{U}_{q}(\mathfrak{g})$ are parametrised by $\mathcal{P}^{+}$.

We now outline the construction of the quantum group, or quantised coordinate algebra, associated to $\mathrm{U}_{q}(\mathfrak{g})$. Given any left $\mathrm{U}_{q}(\mathfrak{g})$-module $V$ we consider the corresponding linear dual $V^{\ast}$ and functions 

\[
\begin{aligned}
	c_{f,v}^{V}\colon \mathrm{U}_{q}(\mathfrak{g})&\to \mathbb{C}, \\ 
	X & \mapsto f(X\triangleright v), 
\end{aligned}
\]

\noindent  for any  $v\in V$ and $f\in V^{\ast}$. The space of matrix coefficients of $V$ is the subspace 

\begin{equation}
	\mathcal{C}(V) := \mathrm{Span}_{\mathbb{C}}\left\{c_{f,v}^{V}\mid v\in V,\, f\in V^{\ast}\right\}\subseteq \mathrm{U}_{q}(\mathfrak{g})^{\ast}.
\end{equation}

\noindent Letting $\mathrm{U}_{q}(\mathfrak{g})^{\circ}$ be the dual Hopf algebra of $\mathrm{U}_{q}(\mathfrak{g})$, a Hopf subalgebra of $\mathrm{U}_{q}(\mathfrak{g})^{\circ}$ is given by the \emph{quantised coordinate algebra of $\mathrm{G}$}, \emph{i.e.} the direct sum of all spaces of matrix coefficients associated to the type-$1$ representations of $\mathrm{U}_{q}(\mathfrak{g})$

\begin{equation}
	\mathcal{O}_{q}(\mathrm{G}):= \bigoplus_{\mu\in \mathcal{P}^{+}}\mathcal{C}(V_{\mu}), 
\end{equation}

\noindent where $\mathrm{G}$ is the Lie group whose complexified Lie algebra is $\mathfrak{g}$. For any weight basis $\{e_{i}\}_{i=1}^{n}$ of $V_{\varpi_{n}}$, $\varpi_{n}\in\mathcal{P}^{+}$, with dual basis $\{f_{i}\}_{i=1}^{n}$, the associated matrix coefficients are $ u_{ij}:=c^{\varpi_{n}}_{f_{i},e_{j}}$. 

The algebra $\mathcal{O}_{q}(\mr{G})$ is generated by such elements $u_{ij}$ with some commutation relations (see, for example, \cite{FRT}). Furthermore, for $q$ not a root of unity, the category of finite-dimensional type-1 $\mathrm{U}_q(\mathfrak g)$-modules is semisimple, and therefore $\mathcal{O}_{q}(\mathrm{G})$ is cosemisimple as a coalgebra by construction.

 We now recall the notion of quantised flag manifold, which is of central importance for the development of the paper. Let $\{\alpha_{i}\}_{i\in S}\subseteq \Pi$ be a subset of simple roots of $\mathfrak{g}$, and consider the \emph{quantised Levi subalgebra}, that is the Hopf subalgebra 
 
\[ \mathrm{U}_{q}(\ell_{S}):=\left\langle \, K_{i}^{\pm 1},E_{j},F_{j}\mid i=1,\dots \ell,\ j\in S\, \right\rangle,  \]

\noindent along with its semisimple Hopf subalgebra

\[
\mathrm{U}_{q}(\ell_{S}^{\mathrm{s}}):=\left\langle \, K_{i}^{\pm 1},E_{i},F_{i}\mid i\in S \right\rangle. 
\]

\noindent The category ${}_{\mathrm{U}_{q}(\ell_{S})}\mathrm{Mod}$ of $\mathrm{U}_{q}(\ell_{S})$ finite dimensional type-1 representations is also semisimple \cite{Re-Sombrg}. From the natural Hopf algebra embedding $\iota_{S}\colon \mathrm{U}_{q}(\ell_{S})\hookrightarrow \mathrm{U}_{q}(\mathfrak{g})$ we have an induced dual Hopf algebra map $\iota_{S}^{\circ}\colon \mathrm{U}_{q}(\mathfrak{g})^{\circ}\to \mathrm{U}_{q}(\ell_{S})^{\circ}$, and since $\mathcal{O}_{q}(\mathrm{G})\subset \mathrm{U}_{q}(\mathfrak{g})^{\circ}$ we may also consider the restriction $\pi_{S}:=\iota_{S}^{\circ}|_{\mathcal{O}_{q}(\mathrm{G})}\colon \mathcal{O}_{q}(\mathrm{G}) \to \mathrm{U}_{q}(\ell_{S})^{\circ}$, which defines an action of $\mathrm{U}_{q}(\ell_{S})$ on $\mathcal{O}_{q}(\mathrm{G})$ by the non-degenerate dual pairing between the dual pair of Hopf algebras
\[
\langle-,-\rangle \colon \mathrm{U}_{q}(\mathfrak{g})\times \mathcal{O}_{q}(\mathrm{G}) \to \mathbb{C}. 
\]
 \begin{defn}
 	The \emph{quantised flag manifold} associated to the subset of simple roots $S\subset \Pi$ is the subalgebra 
 	\begin{equation}
 		\mathcal{O}_{q}(\mathrm{G}/\mathrm{L}_{S}):={}^{\mathrm{U}_{q}(\ell_{S})}\mathcal{O}_{q}(\mathrm{G})=\Big\{h\in \mathcal{O}_{q}(\mathrm{G}) \ \mid \ X\triangleright h= h_{(1)}\langle X,h_{(2)}\rangle = \varepsilon(X)h, \ \forall X\in \mathrm{U}_{q}(\ell_{S})\Big\}
 	\end{equation}
 	of invariant elements under the action of $\mathrm{U}_{q}(\ell_{S})$. The quantised flag manifold $\mathcal{O}_{q}(\mathrm{G}/\mathrm{L}_{S})$ is called \emph{irreducible} if $\Pi\setminus S=\{\alpha\}$ where $\alpha$ has coefficient one in the highest root, and \emph{non-irreducible} otherwise. We say that it is of \emph{contact type} if the grading determined by the crossed roots $\Pi\setminus S$ is a contact grading, namely
\[
\mathfrak{g}
=\mathfrak{g}_{-2}\oplus\mathfrak{g}_{-1}
\oplus\mathfrak{g}_{0}\oplus\mathfrak{g}_{1}\oplus\mathfrak{g}_{2},
\]
where $\dim\mathfrak{g}_{-2}=1$ and the Lie bracket
$\mathfrak{g}_{-1}\times\mathfrak{g}_{-1}\to\mathfrak{g}_{-2}$
is non-degenerate. In this case, the underlying classical flag manifold carries its natural invariant holomorphic contact structure.
\end{defn}

In Table \ref{tab:contact-flags} we list the subsets of crossed simple roots $\Pi\setminus S$ that determine contact gradings of $\mathfrak{g}$, and hence the corresponding choices of quantised Levi subalgebras giving rise to quantised flag manifolds $\mathcal{O}_{q}(\mathrm{G}/\mathrm{L}_{S})$ of contact type, and whose classical counterparts carry natural invariant holomorphic contact structures. 

\begingroup
\tikzset{
  contact circle/.style={circle, draw=black, fill=lightgray,
    minimum size=2.5mm, inner sep=0pt},
  contact cross/.style={rectangle, minimum size=3mm, inner sep=0pt}
}
\tikzcdset{
  contact diagram/.style={
    ampersand replacement=\&, column sep=6mm, row sep=4mm,
    arrows={no head, /tikz/>={Computer Modern Rightarrow}}, cells={nodes={inner sep=0pt, anchor=center}},
    /tikz/every label/.append style={inner sep=1pt},
    /tikz/baseline=(current bounding box.center)
  }
}
\begin{table}[h]
\centering
\setlength{\tabcolsep}{7mm}
\begin{tabular}{@{}lll@{}}

Type & $\Pi\setminus S$ & Crossed Dynkin diagram \\[3mm]
\midrule
$\mathrm{A}_n\ (n\geq2)$ & $\{\alpha_1,\alpha_n\}$ & \begin{tikzcd}[contact diagram, ampersand replacement=\&]
|[contact cross, /tikz/label={above:\alpha_{1}}]| \scalebox{1.5}{$\times$} \arrow[r] \& \cdots \arrow[r] \& |[contact cross, /tikz/label={above:\alpha_{n}}]| \scalebox{1.5}{$\times$}
\end{tikzcd} \\[8mm]
$\mathrm{B}_n\ (n\geq3)$ & $\{\alpha_2\}$ & \begin{tikzcd}[contact diagram, ampersand replacement=\&]
|[contact circle, /tikz/label={above:\alpha_{1}}]| {} \arrow[r] \& |[contact cross, /tikz/label={above:\alpha_{2}}]| \scalebox{1.5}{$\times$} \arrow[r] \& \cdots \arrow[r] \& |[contact circle, /tikz/label={above:\alpha_{n-1}}]| {} \arrow[r, /tikz/double, /tikz/double distance=1pt, ->] \& |[contact circle, /tikz/label={above:\alpha_{n}}]| {}
\end{tikzcd} \\[8mm]
$\mathrm{C}_n\ (n\geq2)$ & $\{\alpha_1\}$ & \begin{tikzcd}[contact diagram, ampersand replacement=\&]
|[contact cross, /tikz/label={above:\alpha_{1}}]| \scalebox{1.5}{$\times$} \arrow[r] \& \cdots \arrow[r] \& |[contact circle, /tikz/label={above:\alpha_{n-1}}]| {} \arrow[r, /tikz/double, /tikz/double distance=1pt, <-] \& |[contact circle, /tikz/label={above:\alpha_{n}}]| {}
\end{tikzcd} \\[8mm]
$\mathrm{D}_n\ (n\geq4)$ & $\{\alpha_2\}$ & \begin{tikzcd}[contact diagram, ampersand replacement=\&]
 \&  \&  \&  \& |[contact circle, /tikz/label={right:\alpha_{n}}]| {} \\[-1mm]
|[contact circle, /tikz/label={above:\alpha_{1}}]| {} \arrow[r] \& |[contact cross, /tikz/label={above:\alpha_{2}}]| \scalebox{1.5}{$\times$} \arrow[r] \& \cdots \arrow[r] \& |[contact circle, /tikz/label={above:\alpha_{n-2}}]| {} \arrow[ur] \arrow[dr] \&  \\[-1mm]
 \&  \&  \&  \& |[contact circle, /tikz/label={right:\alpha_{n-1}}]| {}
\end{tikzcd} \\[8mm]
$\mathrm{E}_6$ & $\{\alpha_2\}$ & \begin{tikzcd}[contact diagram, ampersand replacement=\&]
 \&  \& |[contact cross, /tikz/label={above:\alpha_{2}}]| \scalebox{1.5}{$\times$} \&  \&  \\[-1mm]
|[contact circle, /tikz/label={above:\alpha_{1}}]| {} \arrow[r] \& |[contact circle, /tikz/label={above:\alpha_{3}}]| {} \arrow[r] \& |[contact circle, /tikz/label={below:\alpha_{4}}]| {} \arrow[r] \arrow[u] \& |[contact circle, /tikz/label={above:\alpha_{5}}]| {} \arrow[r] \& |[contact circle, /tikz/label={above:\alpha_{6}}]| {}
\end{tikzcd} \\[8mm]
$\mathrm{E}_7$ & $\{\alpha_1\}$ & \begin{tikzcd}[contact diagram, ampersand replacement=\&]
 \&  \& |[contact circle, /tikz/label={above:\alpha_{2}}]| {} \&  \&  \&  \\[-1mm]
|[contact cross, /tikz/label={above:\alpha_{1}}]| \scalebox{1.5}{$\times$} \arrow[r] \& |[contact circle, /tikz/label={above:\alpha_{3}}]| {} \arrow[r] \& |[contact circle, /tikz/label={below:\alpha_{4}}]| {} \arrow[r] \arrow[u] \& |[contact circle, /tikz/label={above:\alpha_{5}}]| {} \arrow[r] \& |[contact circle, /tikz/label={above:\alpha_{6}}]| {} \arrow[r] \& |[contact circle, /tikz/label={above:\alpha_{7}}]| {}
\end{tikzcd} \\[8mm]
$\mathrm{E}_8$ & $\{\alpha_8\}$ & \begin{tikzcd}[contact diagram, ampersand replacement=\&]
 \&  \& |[contact circle, /tikz/label={above:\alpha_{2}}]| {} \&  \&  \&  \&  \\[-1mm]
|[contact circle, /tikz/label={above:\alpha_{1}}]| {} \arrow[r] \& |[contact circle, /tikz/label={above:\alpha_{3}}]| {} \arrow[r] \& |[contact circle, /tikz/label={below:\alpha_{4}}]| {} \arrow[r] \arrow[u] \& |[contact circle, /tikz/label={above:\alpha_{5}}]| {} \arrow[r] \& |[contact circle, /tikz/label={above:\alpha_{6}}]| {} \arrow[r] \& |[contact circle, /tikz/label={above:\alpha_{7}}]| {} \arrow[r] \& |[contact cross, /tikz/label={above:\alpha_{8}}]| \scalebox{1.5}{$\times$}
\end{tikzcd} \\[8mm]
$\mathrm{F}_4$ & $\{\alpha_1\}$ & \begin{tikzcd}[contact diagram, ampersand replacement=\&]
|[contact cross, /tikz/label={above:\alpha_{1}}]| \scalebox{1.5}{$\times$} \arrow[r] \& |[contact circle, /tikz/label={above:\alpha_{2}}]| {} \arrow[r, /tikz/double, /tikz/double distance=1pt, ->] \& |[contact circle, /tikz/label={above:\alpha_{3}}]| {} \arrow[r] \& |[contact circle, /tikz/label={above:\alpha_{4}}]| {}
\end{tikzcd} \\[8mm]
$\mathrm{G}_2$ & $\{\alpha_2\}$ & \begin{tikzcd}[contact diagram, ampersand replacement=\&]
|[contact circle, /tikz/label={above:\alpha_{1}}]| {} \arrow[r] \arrow[r, shift left=1.4pt] \arrow[r, shift right=1.4pt] \arrow[r, /tikz/>={Computer Modern Rightarrow[length=4pt,width=8pt]}, <-] \& |[contact cross, /tikz/label={above:\alpha_{2}}]| \scalebox{1.5}{$\times$}
\end{tikzcd} \\[8mm]

\end{tabular}
\caption{Crossed Dynkin diagrams defining contact gradings of complex simple
Lie algebras. Crosses mark the roots in $\Pi\setminus S$;
uncrossed nodes belong to $S$. See \cite{parabook}, Section 3.2.4 for further details.}
\label{tab:contact-flags}
\end{table}
\endgroup

 \begin{thm}\label{th:generators_of_a_flag}
A set of generators for $\mathcal{O}_{q}(\mathrm{G}/\mathrm{L}_{\mathrm{S}}^{\mathrm{s}}):={}^{\mathrm{U}_{q}(\ell_{\mathrm{S}}^{\mathrm{s}})}\mathcal{O}_{q}(\mathrm{G})$ is given by 

		\begin{equation}\nonumber 
		z_{i}:=c_{f_{i},v_{N_{x}}}^{\varpi_{x}}, \qquad \overline{z}_{j}:=c_{v_{j},f_{N_{x}}}^{\varpi_{x}},   
		\end{equation} 
		
		\noindent for  $i,j=1,\dots,N_{x}:=\mathrm{dim}(V_{\varpi_{x}})$, $x\in \Pi \backslash S$. 
\end{thm}

\begin{rmk} It is an easy check that any quantised flag manifold is automatically a left coideal subalgebra. Indeed, let $g\in \mathcal{O}_{q}(\mathrm{G}/\mathrm{L}_{S})$. Then $X\triangleright g = g_{1}\langle X,g_{2}\rangle=\varepsilon(X)g$, thus 
 	\begin{align*}
 		g_{(1)}\otimes g_{(2)}\langle X,g_{(3)}\rangle = \varepsilon(X)g_{(1)}\otimes g_{(2)},
 	\end{align*} 
 	meaning $\Delta(\mathcal{O}_{q}(\mathrm{G}/\mathrm{L}_{S}))\subseteq \mathcal{O}_{q}(\mathrm{G})\otimes \mathcal{O}_{q}(\mathrm{G}/\mathrm{L}_{S})$. Moreover, it can be proven that $\mathcal{O}_{q}(\mathrm{G})$ is right faithfully flat over $\mathcal{O}_{q}(\mathrm{G}/\mathrm{L}_{S})$ and $\mathcal{O}_{q}(\mathrm{G})\left(\mathcal{O}_{q}(\mathrm{G}/\mathrm{L}_{S})\right)^{+} = \left(\mathcal{O}_{q}(\mathrm{G}/\mathrm{L}_{S})\right)^{+}\mathcal{O}_{q}(\mathrm{G})$, see \emph{e.g.} \cite{Sch90,MS99}, and thus quantised flag manifolds supply a particularly rich class of quantum homogeneous spaces. \qed 
\end{rmk}
\subsection{Differential calculi over Hopf algebras and homogeneous spaces}
\label{subsection:basics_on_dcs}
After introducing the algebraic framework underlying our approach, we now turn to one of the fundamental structures of noncommutative differential geometry, namely differential calculi. For the reader’s convenience, we briefly recall the basic notions and constructions that will be needed in the sequel.
\begin{defn}
	Let $A$ be an algebra.
	\begin{enumerate}
		\item  A first order Differential Calculus (FODC) over $A$ is a tuple $(\Omega^{1}(A),\mathrm{d})$, where $\Omega^{1}(A)$ is an $A$-bimodule and $\mathrm{d}\colon A \to \Omega^{1}(A)$ is a derivation such that $\Omega^{1}(A)$ is generated as a left (or equivalently right) $A$-module by elements of the form $\mathrm{d}A$. 
		\item A Differential Calculus (DC) over $A$ is a Differential Graded Algebra (DGA) $(\Omega^{\bullet}(A),\mathrm{d})$ generated by elements $a,\mathrm{d}b$, for $a,b\in\Omega^{0}(A)=A$. 
	\end{enumerate}
   We say that a DC $(\Omega^{\bullet}(A),\mathrm{d}_{\Omega})$ extends a FODC $(\Gamma(A),\mathrm{d}_{\Gamma})$ if there is a bimodule isomorphism $\varphi\colon \Omega^{1}(A)\to \Gamma(A)$ such that $\mathrm{d}_{\Gamma}=\varphi\circ\mathrm{d}_{\Omega}$. 
\end{defn}
In contrast to the classical smooth geometrical setting, where a fixed smooth structure on a manifold has its canonical de Rham differential calculus, differential structure over algebras can be realised by means of multiple choices of FODCi, which, furthermore, can be extended to DCs in multiple ways. However, given a FODC, there is a canonical procedure to extend the latter to a DC \cite{BM20} that goes as follows. Let $(\Gamma,\mathrm{d})$ be a FODC over $A$. Consider the tensor algebra 
\[ 
\Gamma^{\otimes_{A}}=\bigoplus_{n=0}^{\infty}\Gamma^{\otimes^{n}_{A}}=A\oplus\Gamma\oplus (\Gamma\otimes_{A}\Gamma)\oplus \cdots, 
\]
and consider the graded ideal $S\subset \Gamma^{\otimes_{A}}(A)$ generated by elements of the form $\mathrm{d}a^{i}\otimes_{A}\mathrm{d}b^{i}$ such that $a^{i}\mathrm{d}b^{i}=0$, where summation over repeated indices is understood. The quotient $\Gamma^{\bullet}:=\Gamma^{\otimes_{A}}/S$ is a graded associative unital algebra, whose product induced by the quotient we denote by $\wedge\colon \Gamma^{k}\otimes \Gamma^{\ell}\to \Gamma^{k+\ell}$. Furthermore, the differential $\mathrm{d}\colon \Gamma^{0}\to \Gamma^{1}$ extends as a derivation on $\Gamma^{1}$ by
\begin{equation}
	\begin{aligned}
		\mathrm{d}^{1} \colon \Gamma^{1} & \to\Gamma^{2},  \\ 
		          a\mathrm{d}b & \mapsto [\mathrm{d}a\otimes_{A}\mathrm{d}b], 
	\end{aligned}
\end{equation}
and therefore extends as a derivation $\mathrm{d}^{\bullet}\colon \Gamma^{\bullet}\to \Gamma^{\bullet +1}$. It follows that $(\Gamma^{\bullet},\wedge,\mathrm{d}^{\bullet})$ is a differential calculus over $A$, the \emph{maximal prolongation of} $\Gamma$, whose truncation to the first order consists of the original datum $(\Gamma,\mathrm{d})$. This extension is maximal, in the sense that any other DC $(\Omega^{\bullet},\mathrm{d}_{\Omega})$ over $A$ whose truncation to the first order is $(\Gamma,\mathrm{d})$ can be induced as a quotient of $\Gamma^{\bullet}$. 

    Next, we turn to covariant calculi, namely differential calculi on a comodule algebra equipped with a compatibility condition with the underlying coaction. 
    \begin{defn}
	Let $H$ be a Hopf algebra, and let $A$ be a left $H$-comodule algebra. A differential calculus $(\Omega^{\bullet},\mathrm{d})$ over $A$ is called \emph{left covariant} if the coaction ${}_{A}\Delta\colon A\to H\otimes A$ extends to a coaction ${}_{\Omega^{\bullet}}\Delta\colon \Omega^{\bullet}\to H\otimes \Omega^{\bullet}$ giving $\Omega^{\bullet}$ the structure of an object in the category ${}^{H}_{A}\mathrm{Mod}_{A}$ of left $H$-covariant $A$-bimodules, and the differential $\mathrm{d}$ is left $H$-colinear. Right $H$-covariant and bicovariant differential calculi are defined similarly. 
\end{defn}

\noindent In the following $B\subseteq A$ will be a quantum homogeneous space of a Hopf algebra $A$, with $A^{\circ}$ denoting the corresponding dual Hopf algebra of $A$ and $B^{\circ}$ denoting the dual coalgebra of $B$. We briefly describe the relation of tangent spaces and differential calculi at the level of the total space Hopf algebra $A$ and at the level of its homogeneous spaces, in the spirit of the classification theorems of Woronowicz and Hermisson. 
\begin{defn}
A \emph{tangent space} for $A$ is a vector subspace $T\subset A^{\circ}$ such that $X(1)=0$ for any $X\in T$ and $T\oplus \Bbbk 1$ is a right coideal of $A^{\circ}$. 
\end{defn} 

\noindent For any tangent space $T\subset A^{\circ}$ a right $A$-ideal in $A^{+}:=\ker\varepsilon$ is given by 
\begin{equation}\label{eq:ideal_from_tangent_space}
I=\left\{a\in A^{+}\mid X(a)=0, \ \text{for all}\ X\in T\right\}, 
\end{equation}
and therefore the quotient $A^{+}/ I$ is an object in the category of right $A$-modules $\mathrm{Mod}_{A}$. Let $\Psi\colon \mathrm{Mod}_{A}\to {}^{A}_{A}\mathrm{Mod}_{A}$ be the free functor in Theorem \ref{th:fundamental_theorem_hopf}, as appearing in Equation \eqref{eq:free_functor_hopf_modules}. It follows that $\Psi(A^{+}/ I)=A\otimes A^{+}/ I:=\Omega^{1}(A)$ is naturally an object in ${}^{A}_{A}\mathrm{Mod}_{A}$. Furthermore, there is a derivation 
\begin{equation}
	\begin{aligned}
		\mathrm{d}\colon A & \to \Omega^{1}(A), \\
		 				 a & \mapsto a_{(1)}\otimes [a_{(2)}^{+}], 
	\end{aligned}
\end{equation}
where $(\ \ )^{+}\colon A\to A^{+}$ denotes the projection to the kernel of the counit of $A$ and where square brackets stand for cosets in $A^{+}/I$. The pair $(\Omega^{1}(A):=A\otimes \Lambda^{1}_{A},\mathrm{d})$ is a left $A$-covariant FODC over $A$ with \emph{cotangent space} $\Lambda^{1}_{A}:=A^{+}/ I$. 

For the convenience of the reader, we include a proof of the following standard observation.

\begin{lemma}
Let $W \subseteq A^\circ$ be a Hopf subalgebra, and let $\pi\colon A \to W^\circ$ be the Hopf algebra map defined by $\pi(a)(w) := \langle w,a\rangle.$ Set $K:=\pi(A)$. Suppose that the quantum tangent space $T\subseteq A^\circ$ is invariant under the right adjoint action $X\triangleleft_{\mathrm{Ad},\mathrm{R}}w := S(w_{(1)})Xw_{(2)}.$ Then the classifying ideal
	\[
    I:=\{a\in A^+ \mid \langle X,a\rangle=0
                       \text{ for all }X\in T\}
	\]
    is stable under the relative right adjoint coaction $\mathrm{coAd}_{\mathrm{R}}\colon A\to A\otimes K$, $\mathrm{coAd}_{\mathrm{R}}(a)= a_{(2)} \otimes\pi\bigl(S(a_{(1)})a_{(3)}\bigr)$.

\begin{proof}
For $X\in T$, $w\in W$, and $a\in A$, we have
\begin{align*}
    \langle X\triangleleft_{\mathrm{Ad},\mathrm{R}}w,a\rangle
    &=
    \langle S(w_{(1)})Xw_{(2)},a\rangle\\
    &=
    \langle w_{(1)},\pi(S(a_{(1)}))\rangle
    \langle X,a_{(2)}\rangle
    \langle w_{(2)},\pi(a_{(3)})\rangle\\
    &=
    \langle X,a_{(2)}\rangle
    \langle w,\pi(S(a_{(1)})a_{(3)})\rangle\\
    &=
    \langle X\otimes w,\mathrm{coAd}_{\mathrm{R}}(a)\rangle.
\end{align*}
\end{proof}
\end{lemma}

\begin{thm}[{\cite{Woronowicz1989}}, \emph{Woronowicz's classification theorem}]\label{ThmWor}
There is a $1$$:$$1$-correspondence
\begin{equation}
\begin{Bmatrix}
\text{left covariant}\\
\text{FODC}\\
\text{on $H$}
\end{Bmatrix}\xleftrightarrow{1:1}
\begin{Bmatrix}
\text{right}\\
\text{$H$-ideals}\\
\text{in $H^+$}
\end{Bmatrix}
\end{equation}
Furthermore, a FODC on $A$ is bicovariant if and only if the classifying ideal $I$ is closed under the adjoint right $A$-coaction, i.e. if $\mathrm{coAd}(I)\subseteq I\otimes A$, where $\mathrm{coAd}\colon A\to A\otimes A$, $a\mapsto a_{(2)}\otimes S(a_{(1)})a_{(3)}$.
\end{thm}
\noindent In particular, the classifying ideal $I$ of a covariant FODC is given as the kernel of the \emph{quantum Maurer--Cartan} form 

\begin{equation}
	\begin{aligned}
		\varpi \colon A^{+}& \to \Lambda^{1}_{A}, \\ 
		       a &\mapsto S(a_{(1)})\mathrm{d}(a_{(2)}). 
	\end{aligned}
\end{equation}

\noindent Owing to the following Proposition, which can be found e.g. in \cite{BM20} (Proposition 2.31 page 107) or in \cite{Re-Sombrg} (page 7), it turns out that the relations for the maximal prolongation of a covariant FODC over a Hopf algebra are given in terms of the quantum Maurer--Cartan form as well. 

\begin{prop}
\label{proposition:quantum_exterior_algebra}
	Let $\Omega^{1}(A)\cong A\otimes \Lambda^{1}_{A}$ be a left covariant FODC. Its maximal prolongation is a left covariant DC of the form 
	
	\[
	\Omega^{\bullet}(A)\cong A \otimes \Lambda^{\bullet}_{A},  	\qquad  \Lambda^{\bullet}_{A}:=\bigoplus_{k\in\mathbb{Z}_{\geq 0}}\Lambda^{k}_{A}:= \mathcal{T}\left(\Lambda^{1}_{A}\right)/\langle \mathcal{I}^{(2)}\rangle, 
	\]
	
	\noindent where $\mathcal{T}\left(\Lambda^{1}_{A}\right)$ denotes the tensor algebra of $\Lambda^{1}_{A}$, $\pi_{\varepsilon}\colon A\to A^{+}$, and where 
	\begin{equation}\label{eq:I2}
	\begin{aligned}
		\mathcal{I}^{(2)}& := \Big\{(\varpi\circ \pi_{\varepsilon} \otimes \varpi\circ \pi_{\varepsilon})\circ \Delta(x) \mid x\in \mathcal{I} \Big\}\\ 
		& = \Big\{\varpi^{\otimes_{2}}(x):=[x_{(1)}^{+}]\otimes [x_{(2)}^{+}] \mid x\in \mathcal{I} \Big\}. 
	\end{aligned}	
	\end{equation}
\end{prop}

\noindent We refer to $\Lambda^{\bullet}_{A}$ as the \emph{quantum exterior algebra} associated to the DC $\Omega^{1}(A)$.

 We now present a similar construction for the more general case of a quantum homogeneous space of a Hopf algebra. Let $B\subseteq A$ be the quantum homogeneous space defined by the space of invariant elements  
\[
B:={}^{W}A=\Big\{b\in A \mid b_{(1)}\langle X,b_{(2)}\rangle=\varepsilon(X)b, \ \text{for all} \ X \in W\Big\}
\]
under the action of a subalgebra $W\subseteq A^{\circ}$ in the dual Hopf algebra of $A$. A \emph{tangent space} for the quantum homogeneous space $B$ is a subspace $T\subseteq B^{\circ}$ of the dual coalgebra of $B$, such that $T\oplus \Bbbk 1$ is a right coideal in $B^{\circ}$, $T(1)=0$, and $T$ is stable under the right adjoint $W$-action. 

As before, for any tangent space we have a right ideal in $B^{+}:=B\cap \ker\varepsilon$ as in Equation \eqref{eq:ideal_from_tangent_space}. The quotient $B^{+}/\mathcal{I}$ is therefore an object in the category ${}^{\pi_{B}}\mathrm{Mod}_{B}$ of left $\pi_{B}$-covariant, right $B$-modules. Let $\Psi\colon {}^{\pi_{B}}\mathrm{Mod}_{B} \to {}^{A}_{B}\mathrm{Mod}_{B}$ be the free functor in Theorem \ref{thm:takeuchi_equivalence}, appearing in Equation \eqref{eq:free_functor_takeuchi}. It follows that $\Psi(B^{+}/I)=A\square_{\pi_{B}}(B^{+}/\mathcal{I})$ is an object in ${}^{A}_{B}\mathrm{Mod}_{B}$ and, furthermore, the map
\begin{equation}
	\begin{aligned}
		\mathrm{d}\colon B & \to \Omega^{1}(B), \\
		 				 b & \mapsto b_{(1)}\otimes [b_{(2)}^{+}], 
	\end{aligned}
\end{equation}
is a derivation. The pair $\left(\Omega^{1}(B):=A\square_{\pi_{B}} V^{1}_{B},\mathrm{d}\right)$ is a left $A$-covariant FODC over $B$ with cotangent space $V^{1}_{B}:=B^{+}/\mathcal{I}$. 
This is the essential content of two classification theorems for covariant calculi on quantum homogeneous spaces due to Hermisson and Heckenberger--Kolb. 
	\begin{thm}[\cite{Hermisson1}, \emph{Hermisson's classification theorem}]
	\label{th:hermisson_class} The pair $\left(\Omega^{1}(B):=A\square_{\pi_{B}}(B^{+}/ I),\mathrm{d}\right)$ is a left $A$-covariant FODC on $B$.  Moreover, every left $A$-covariant FODC on $B$ is of this form. 
	\end{thm}
		\begin{thm}[\cite{hek-kolb-2003}]
		\label{th:hk_tangent}
			There is a bijective correspondence between isomorphism classes of finite-dimensional tangent spaces and finitely generated left $A$-covariant FODCi on $B$.
		\end{thm}
		
A FODC over a quantum homogeneous $A$-space $B$ is \emph{irreducible} if it is simple as an $(A,B)$-relative Hopf module. Finally, given any FODC over an algebra, and in particular a left covariant FODC $\Omega^{1}(A)=A\otimes \Lambda_A^1$ over a Hopf algebra $A$, one may consider its restriction to $B$, namely the $B$-subbimodule of $\Omega^{1}(A)$ generated by elements of the form $\mathrm{d}b$, with $b\in B$. This defines a covariant FODC on $B$, the \emph{restriction} of $\Omega^{1}(A)$ to $B$.

\section{Differential calculi on $\mathcal{O}_{q}(\mathrm{Sp}_{4})$ and its contact quantised flag manifold}
\label{section:C2}
We now specialise the preceding constructions to the contact quantised flag manifold corresponding to the symplectic Lie algebra $\mathfrak{sp}_4$. 

In Section \ref{section:basics_C2} we recall the Lie algebra $\mathfrak{sp}_4$, its quantised enveloping algebra, and the associated quantum group. In Section \ref{section:contact_calculus_C2} we construct the two-dimensional contact calculus on $\mathcal O_q(\mathrm{Sp}_4/\mathrm L_S)$ together with its framing calculus on $\mathcal O_q(\mathrm{Sp}_4)$, and examine the dimension drop in its maximal prolongation. In \ref{section:reeb_calculus_C2} we adjoin the Reeb direction, obtaining a three-dimensional calculus whose maximal prolongation defines a Dolbeault complex of classical dimension on the contact flag manifold. Finally, Section \ref{section:contact_adapted_calculus_C2} extends this construction by adjoining $E_2$, producing a four-dimensional contact-adapted calculus on $\mathcal O_q(\mathrm{Sp}_4)$ and its full flag manifold. For ease of reference, the dual pairings used throughout this section are collected in a table displayed in Appendix~\ref{appendix}.

\subsection{The Lie algebra $\mathfrak{sp}_{2\ell}$, its quantised envelope and quantum group}\label{section:basics_C2}
 	Consider the symplectic Lie algebra $\mathfrak{g}:=\mathfrak{sp}_{2\ell}$. We briefly recall some basics about the construction of the corresponding root system, providing an explicit basis for the latter. Moreover, we write down the Cartan matrix and Dynkin diagram, eventually presenting the associated Weyl group. Finally, we discuss the structure of the quantised enveloping algebra $\mathrm{U}_{q}(\mathfrak{sp}_{2\ell})$ and the quantised coordinate algebra $\mathcal{O}_{q}(\mathrm{Sp}_{4})$. The material presented in this Section is standard and can be found in many textbooks on Lie algebras and representation theory, see \emph{e.g.} \cite{Hump78} and \cite{Kli-Sch}. 
 	
The following elements form a basis of $\mathfrak{sp}_{2\ell}$:
\begin{align*}
h_i&=e_{ii}-e_{\ell+i,\ell+i}, & &1\leq i\leq\ell, &  m_{ij}&=e_{ij}-e_{\ell+j,\ell+i}, & & 1\leq i\ne j\leq\ell,\\
p_{ij}&=e_{i,\ell+j}+e_{j,\ell+i}, & &1\leq i\leq j\leq\ell, & q_{ij}&=e_{\ell+j,i}+e_{\ell+i,j}, & &1\leq i\leq j\leq\ell.
\end{align*}
Letting $\mathfrak{h} \subseteq \mathfrak{g}$ be the set of diagonal matrices, with entries labeled from $1$ to $2\ell$, we consider a generic element $h\in \mathfrak{h}$ with diagonal entries $a_{1},\dots,a_{\ell},- a_{1},\dots,-a_{\ell}$. One easily shows that the adjoint $\mathfrak{h}$-action on $\mathfrak{g}$ reads  
\[
\begin{aligned}
		\mathrm{Ad}_{\mathfrak{h}}(m_{ij})& =(a_{i}-a_{j})m_{ij}, \qquad  \mathrm{Ad}_{\mathfrak{h}}(p_{ij})&=(a_{i}+a_{j})p_{ij}, \qquad  \mathrm{Ad}_{\mathfrak{h}}(q_{ij})&=-(a_{i}+a_{j})q_{ij}.  
\end{aligned}
\]
For $i=1,\dots,\ell$ consider maps $\epsilon_{i}\colon \mathfrak{h}\to \mathbb{C}$, sending $h \mapsto a_{i}$, i.e.  each element $h\in \mathfrak{h}$ is sent to the corresponding $i$-th diagonal entry $a_{i}$. We summarise the roots corresponding to the common eigenvectors for the action of the Cartan subalgebra on $\mathfrak{g}$:

\begin{equation*}
\begin{tabular}{c|c|c|c|c|c}
 
  root & $\varepsilon_{i}-\varepsilon_{j}$ & $\varepsilon_{i}+\varepsilon_{j}$ & $-(\varepsilon_{i}+\varepsilon_{j})$ &$2\varepsilon_{i}$ &$-2\varepsilon_{i}$ \\
  \hline
  eigenvector & $m_{ij}\ (i\neq j)$ & $p_{ij} \ (i<j)$ & $q_{ji} \ (i<j)$ & $p_{ii}$ & $q_{ii}$\\

\end{tabular}
\end{equation*}
Fixing the labeling $\alpha_{i}:=\varepsilon_{i}-\varepsilon_{i+1}$ and $\beta_{\ell}:= 2\varepsilon_{\ell}$, the \emph{standard basis} for the root system is given by 
\[
\mathcal B= \{\alpha_i:=\varepsilon_i-\varepsilon_{i+1}\mid 1\leq i<\ell\} \cup\{\beta_\ell:=2\varepsilon_\ell\}.
\]
The associated Cartan matrix and Dynkin diagram read 
\[
\mathrm{A}_{\mathrm{C}_{n}}=
\begin{pmatrix}
2 & -1 & 0 & 0 & \cdots & 0 & 0 & 0 \\
-1 & 2 & -1 & 0 & \cdots & 0 & 0 & 0 \\
0 & -1 & 2 & -1 & \cdots & 0 & 0 & 0 \\
0 & 0 & -1 & 2 & \cdots & 0 & 0 & 0 \\
\vdots & \vdots & \vdots & \vdots & \ddots & -1 & 0 & 0 \\
0 & 0 & 0 & 0 & -1 & 2 & -1 & 0 \\
0 & 0 & 0 & 0 & 0 & -1 & 2 & -2 \\
0 & 0 & 0 & 0 & 0 & 0 & -1 & 2
\end{pmatrix}
\]
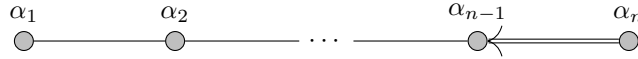
\begin{figure}[!htbp]
\centering	
\begin{tikzpicture}[scale=1]
\tikzstyle{smallcircle} = [circle, draw=black, fill=lightgray, minimum size=2.5mm, inner sep=0pt]

\node[smallcircle, label=above:$\alpha_1$] (n1) at (0,0) {};
\node[smallcircle, label=above:$\alpha_2$] (n2) at (2,0) {};
\node (dots) at (4,0) {$\cdots$};
\node[smallcircle, label=above:$\alpha_{n-1}$] (n3) at (6,0) {};
\node[smallcircle, label=above:$\alpha_n$] (n4) at (8,0) {};

\draw (n1) -- (n2);
\draw (n2) -- (dots);
\draw (dots) -- (n3);

\draw[double, double distance=1pt, <-] (n3) -- (n4);

\end{tikzpicture}
\end{figure}

\noindent and the associated set of fundamental weights is $\varpi_{i}=e_{1}+\ldots+e_{i}$.  The Weyl group of the root system is generated by reflections with respect to simple roots and an explicit presentation in terms of generators and relations can be easily deduced from the Dynkin diagram, namely
\begin{equation*}
 \begin{aligned}
 	 	\mathcal{W}=\langle s_{\alpha_{1}},s_{\alpha_{2}},\dots,s_{\alpha_{n-1}},s_{\beta} & \mid
(s_i)^2=1,\;
(s_i s_j)^{m_{ij}}=1\rangle ,  
\end{aligned}
\end{equation*}
where  $m_{ii}=1$,  $m_{ij}=3$ if $i,j$ are connected by a single edge, $m_{n-1,n}=4$, and $m_{ij}=2$ otherwise.

Let us now fix $q_{i}=q^{\frac{(\alpha_{i},\alpha_{i})}{2}}$, with $q_{i}^{2}\neq 1$ for each $i=1,\dots, n$. Note that, since the Dynkin diagram of $\mathfrak{sp}_{2\ell}$ contains edges of two different lengths, the parameters $q_i$ take two distinct values. With the standard normalization, we have $q_i=q$ for $i<\ell$ and $q_\ell=q^2$. Thus the short simple roots carry the parameter $q$, while the long simple root carries the parameter $q^2$. With the above datum, we construct the quantised enveloping algebra $\mathrm{U}_{q}(\mathfrak{sp}_{2\ell})$ precisely as in Section~\ref{subsection:DJ_quantum_groups_and_flags}. In particular, the defining relations among the generators are mostly the same as for the simple Lie algebras of the $\mathrm{A}$-series, which can be found for example in \cite{Re-Sombrg}, besides the relations involving the last two nodes of the Dynkin diagram, which we summarise here for convenience. 

\begin{equation}
\label{eq:algebra_rel_C2}
	\begin{aligned}
		K_{\ell-1}^{\pm 1}E_{\ell-1}& =q^{\pm 2} E_{\ell-1}K_{\ell-1}^{\pm1},  & K_{\ell-1}^{\pm1}E_{\ell} & =q^{\mp 2}E_{\ell}K_{\ell-1}^{\pm 1}, \\
		K_{\ell-1}^{\pm 1}F_{\ell-1}& =q^{\mp 2} F_{\ell-1}K_{\ell-1}^{\pm1},  & K_{\ell-1}^{\pm1}F_{\ell} & =q^{\pm 2}F_{\ell}K_{\ell-1}^{\pm 1}, \\
		K_{\ell}^{\pm 1}E_{\ell-1}& =q^{\mp 2} E_{\ell-1}K_{\ell}^{\pm1},  & K_{\ell}^{\pm1}E_{\ell} & =q^{\pm 4}E_{\ell}K_{\ell}^{\pm 1}, \\
		K_{\ell}^{\pm 1}F_{\ell-1}& =q^{\pm 2} F_{\ell-1}K_{\ell}^{\pm1},  & K_{\ell}^{\pm1}F_{\ell} & =q^{\mp 4}F_{\ell}K_{\ell}^{\pm 1}, \\
		E_{\ell-1}F_{\ell-1} & = F_{\ell-1}E_{\ell-1} + \frac{K_{\ell-1}-K_{\ell-1}^{-1}}{q-q^{-1}},  &  E_{\ell-1}F_{\ell} &= F_{\ell}E_{\ell-1}, \\ 
		E_{\ell}F_{\ell} & = F_{\ell}E_{\ell} + \frac{K_{\ell}-K_{\ell}^{-1}}{q^{2}-q^{-2}},  & E_{\ell}F_{\ell-1}&=F_{\ell-1}E_{\ell}, 
	\end{aligned}
\end{equation}

\noindent with $q$-Serre relations 

\begin{equation}
\label{eq:serre_c2}
	\begin{aligned}
		E_{\ell}^{2}E_{\ell-1}-[2]_{q^{2}}E_{\ell}E_{\ell-1}E_{\ell}+E_{\ell-1}E_{\ell}^{2}&=0, \\  
		E_{\ell-1}^{3}E_{\ell}-[3]_{q}E_{\ell-1}^{2}E_{\ell}E_{\ell-1} + [3]_{q}E_{\ell-1}E_{\ell}E_{\ell-1}^{2}- E_{\ell}E_{\ell-1}^{3} & =0, \\ 
		F_{\ell}^{2}F_{\ell-1}-[2]_{q^{2}}F_{\ell}F_{\ell-1}F_{\ell}+F_{\ell-1}F_{\ell}^{2}&=0, \\ 
		F_{\ell-1}^{3}F_{\ell}-[3]_{q}F_{\ell-1}^{2}F_{\ell}F_{\ell-1} + [3]_{q}F_{\ell-1}F_{\ell}F_{\ell-1}^{2}- F_{\ell}F_{\ell-1}^{3} & =0. 
	\end{aligned}
\end{equation}

Fix a weight basis $\{e_{i}\}_{i=1}^{2\ell}$ of $V_{\varpi_{1}}$. The quantised coordinate algebra $\mathcal{O}_{q}(\mr{Sp}_{4})$ is generated by the elements $u_{ij}:=c^{\varpi_{1}}_{e^{i},e_{j}}$, and the Hopf algebra structure on $\mathcal{O}_{q}(\mr{Sp}_{4})$ is given by coproduct and counit

\[
\Delta(u_{ij}):=\sum_{k=1}^{2\ell}u_{ik}\otimes u_{kj},\quad \varepsilon(u_{ij}):=\delta_{ij}, 
\]  

\noindent and an antipode that we omit. Furthermore, denoting by $\langle - , - \rangle \colon \mr{U}_{q}(\mathfrak{sp}_{2\ell})\times \mathcal{O}_{q}(\mr{Sp}_{2\ell})\to \mathbb{C}$ the dual pairing between the two Hopf algebras, and denoting $i'=2\ell+1-i$, we have (see e.g. \cite{Kli-Sch} Section 9.4)

\begin{equation}\label{eq:dual_pairing_C2}
\begin{aligned}
\langle K_{i}^{-1},u_{i,i}\rangle & = \langle K_{i}^{-1},u_{(i+1)',(i+1)'}\rangle = \langle K_{i},u_{i',i'}\rangle = \langle K_{i},u_{i+1,i+1}\rangle= q, \quad i\neq \ell, \\ 
\langle K_{i}^{-1},u_{j,j}\rangle & = 1,  \quad \text{for} \quad i\neq \ell,  \ j\neq i,i+1,i',(i+1)', \\
\langle K_{\ell}^{-1},u_{\ell,\ell}\rangle & = \langle K_{\ell},u_{\ell+1,\ell+1}\rangle = q^{2},  \quad \langle K_{\ell},u_{j,j}\rangle=1, \quad \text{for}\quad j\neq \ell,\ell+1, \\
\langle E_{i},u_{i+1,i}\rangle & =-\langle E_{i},u_{i',(i+1)'}\rangle =\langle F_{i},u_{i,i+1}\rangle = -\langle F_{i},u_{(i+1)',i'}\rangle =1, \quad i\neq \ell, \\ 
\langle E_{\ell},u_{\ell+1,\ell}\rangle & = \langle F_{\ell},u_{\ell,\ell+1}\rangle =1. 
\end{aligned}
\end{equation}

\subsection{Contact differential calculi for $\mathcal{O}_{q}(\mathrm{Sp}_4)$ and $\mathcal{O}_{q}(\mathrm{Sp}_4 / \mathrm{L}_{\alpha_{2}})$}\label{section:contact_calculus_C2}
From now until the end of the Section, we fix $\ell=2$ and consider the quantised enveloping algebra and quantised coordinate algebra corresponding to the complex semisimple Lie algebra $\mathfrak{sp}_{4}$, focusing on the corresponding non-irreducible (contact) quantised flag manifold, of which we now outline the realisation. 

Let $S:=\Pi\setminus\{\alpha_{1}\}=\{\alpha_{2}\}$ be a subset of simple roots of $\mathfrak{sp}_{4}$. We represent this choice graphically by placing a cross on the Dynkin diagram of the underlying Lie algebra as follows: 

\[
\begin{tikzpicture}
\tikzstyle{smallcircle} = [circle, draw=black, fill=lightgray, minimum size=2.5mm, inner sep=0pt]
\tikzstyle{cross} = [minimum size=3mm, inner sep=0pt]

\node[cross, label=above:$\alpha_1$] (n1) at (0,0) {\scalebox{1.5}{$\times$}};
\node[smallcircle, label=above:$\alpha_2$] (n2) at (2,0) {};

\draw[double distance=1.5pt,<-] (n1) -- (n2);
\end{tikzpicture}. 
\]

The quantised Levi subalgebra $\mathrm{U}_{q}(\ell_{\alpha_{2}})$ is the Hopf subalgebra 

\[
\mathrm{U}_{q}(\ell_{\alpha_{2}})=\left\langle E_{j},F_{j},K_{i}^{\pm 1}\mid i=1,\dots,\ell,\ j\in S\right\rangle\subset \mathrm{U}_{q}(\mathfrak{sp}_{4}). 
\] 
A tangent space for $\mathcal{O}_{q}(\mathrm{Sp}_{4})$ can be obtained by the right adjoint action of the quantised Levi subalgebra on the node excluded by the underlying set of simple roots. The resulting root vectors will be in 1-1 correspondence with the positive roots $\Delta_{\alpha_{1}}^{+}/\{\alpha_{\mathrm{max}}\}$, where $\alpha_{\mathrm{max}}=2\alpha_{1}+\alpha_{2}$ is the highest root of $\mathfrak{sp}_{4}$. 
\begin{prop}\label{prop:root_vectors_C2}
A tangent space for $\mathcal{O}_{q}(\mathrm{Sp}_{4})$ is given by 

\begin{equation}
	\underline{T}_{2}^{(0,1)}:=E_{1}\triangleleft_{\mathrm{Ad}} \mathrm{U}_{q}(\ell_{\alpha_{2}})=\mathrm{Span}_{\mathbb{C}}\left\{E_{1},[E_{1},E_{2}]_{q^{2}}\right\}. 
\end{equation}	

\begin{proof}
	The right adjoint action of the $K_{i}$-generators in $\mathrm{U}_{q}(\ell_{\alpha_{2}})$ is a rescaling by integer powers of $q$, whereas 
	\[	
		\begin{aligned}
		E_{1}\triangleleft E_{2} &  = [E_{1},E_{2}]_{q^{2}},  &  E_{1}\triangleleft F_{2} & = 0, \\ 
		[E_{1},E_{2}]_{q^{2}}\triangleleft E_{2} & = 0,  & [E_{1},E_{2}]_{q^{2}}\triangleleft F_{2} & = E_{1}, 	
	\end{aligned}
	\]
	other terms vanishing according to the algebraic structure of $\mathrm{U}_{q}(\mathfrak{sp}_{4})$. 
	
	It is clear that $\underline{T}_{2}^{(0,1)}\subset \ker\varepsilon$, and as the coproducts on its linear generators read 
	
\begin{align*}
\Delta(E_{1}) & = E_{1}\otimes K_{1} + 1\otimes E_{1}, \\ 
\Delta([E_{1},E_{2}]_{q^2}) & = 	[E_{1},E_{2}]_{q^2}\otimes K_{2}K_{1} + E_{1}\otimes (K_{1}E_{2}- q^2\, E_{2}K_{1})\\ 
& \quad   + E_{2} \otimes (E_{1}K_2 - q^2\, K_{2}E_{1})+ 1\otimes [E_{1},E_{2}]_{q^2} \\ 
& = [E_{1},E_{2}]_{q^2}\otimes K_{2}K_{1} + (q^{-2}-q^{2})E_{1}\otimes E_{2}K_{1} + 1\otimes [E_{1},E_{2}]_{q^2},  
\end{align*} 
we have $\Delta(\underline{T}_{2}^{(0,1)})\subseteq (\underline{T}_{2}^{(0,1)}\oplus \mathbb{C})\otimes \mathrm{U}_{q}(\mathfrak{sp}_4)$. 
\end{proof}
\end{prop} 
\noindent Letting $\mathcal{I}$ be the right ideal corresponding to the above tangent space according to Theorem \ref{ThmWor}, we have a left covariant FODC on $\mathcal{O}_{q}(\mathrm{Sp}_{4})$ as 
\begin{equation}
	\label{eq:framing_calculus}
\left(\underline{\Omega}^{(0,1)}_{2}(\mathrm{Sp}_{4}):=\mathcal{O}_{q}(\mathrm{Sp}_{4})\otimes \Lambda_2^{(0,1)},\overline{\partial}\right), 
\end{equation}
where $\Lambda^{(0,1)}_{2}:=\mathcal{O}_{q}(\mathrm{Sp}_{4})^{+}/\,\mathcal{I}$ is the associated cotangent space. We call this calculus the \emph{contact calculus}. 

We now construct an analogous \emph{contact} FODC on a quantum homogeneous space of the above Hopf algebra. In particular, we consider the contact quantised flag manifold, as the underlying classical variety is of holomorphic contact type $\mathcal{O}_{q}(\mathrm{Sp}_{4}/\mathrm{L}_{\alpha_{2}})$, i.e. the coideal subalgebra of invariant elements in $\mathcal{O}_{q}(\mathrm{Sp}_{4})$ under the action of $\mathrm{U}_{q}(\ell_{\alpha_{2}})$, for $S$ chosen as above: 

\[
\mathcal{O}_{q}(\mathrm{Sp}_{4}/\mathrm{L}_{\alpha_{2}}) := {}^{\mathrm{U}_{q}(\ell_{\alpha_{2}})}\mathcal{O}_{q}(\mathrm{Sp}_4)=\left\{ b\in \mathcal{O}_{q}(\mathrm{Sp}_{4}) \mid X\triangleright b = b_{(1)}\langle X,b_{(2)}\rangle =\varepsilon(X)b, \ X\in \mathrm{U}_{q}(\ell_{\alpha_{2}})\right\}. 
\]

\begin{prop}
	A tangent space for $\mathcal{O}_{q}(\mathrm{Sp}_{4}/\mathrm{L}_{\alpha_{2}})$ is given by 
	\begin{equation}
		T^{(0,1)}_{2}:=\mathrm{U}_{q}(\ell_{\alpha_{2}})E_{1}=\mathrm{Span}_{\mathbb{C}}\left\{E_{1},E_{2}E_{1}\right\}.  
	\end{equation}
	Moreover, the latter is understood as a restriction to $\mathcal{O}_{q}(\mathrm{Sp}_{4}/\mathrm{L}_{\alpha_{2}})$ of the $\mathcal{O}_{q}(\mathrm{Sp}_{4})$-tangent space $\underline{T}_{2}^{(0,1)}$.  
	\begin{proof}
		Let $Y\in\mathcal{O}_{q}(\mathrm{Sp}_{4}/\mathrm{L}_{\alpha_{2}})^{\circ}$. As $X\triangleright b=\varepsilon(X)b$, for any $b\in\mathcal{O}_{q}(\mathrm{Sp}_{4}/ \mathrm{L}_{\alpha_{2}})$ and $X\in\mathrm{U}_{q}(\ell_{\alpha_{2}})$, we have 
	\begin{align*}
	\langle Y\varepsilon(X),b\rangle & = \langle Y,\varepsilon(X)b\rangle \\ 
	& = \langle Y, X\triangleright b\rangle \\
	& = \langle Y, b_{(1)}\langle X,b_{(2)}\rangle \rangle \\ 
	& = \langle Y,b_{(1)}\rangle \langle X,b_{(2)}	\rangle \\ 
	& = \langle YX,b\rangle. 
	\end{align*}
	Thus, at the level of $\mathcal{O}_{q}(\mathrm{Sp}_{4}/\mathrm{L}_{\alpha_{2}})$, the right $\mathrm{U}_{q}(\ell_{\alpha_{2}})$-action reduces to a left multiplication. The claim now follows as evidently $\Delta(T_{2}^{(0,1)})\subseteq (T_{2}^{(0,1)}\oplus \mathbb{C})\otimes \mathcal{O}_{q}(\mathrm{Sp}_{4}/\mathrm{L}_{\alpha_{2}})^\circ$. 
	\end{proof}
\end{prop}

 Letting $\mathcal{I}_{2}$ be the right $\mathcal{O}_{q}(\mathrm{Sp}_{4}/\mathrm{L}_{\alpha_{2}})$-ideal in $\mathcal{O}_{q}(\mathrm{Sp}_{4}/\mathrm{L}_{\alpha_{2}})^{+}$ corresponding to the above tangent space according to Theorem \ref{th:hermisson_class}, we have a left $\mathcal{O}_{q}(\mathrm{Sp}_{4})$-covariant FODC on $\mathcal{O}_{q}(\mathrm{Sp}_{4}/\mathrm{L}_{\alpha_{2}})$ given by 

\begin{equation}
	\label{eq:calculus_downstairs_C2}
\left({\Omega}^{(0,1)}_{2}(\mathrm{Sp}_{4}/\mathrm{L}_{\alpha_{2}}):=\mathcal{O}_{q}(\mathrm{Sp}_{4})\square_{\pi_{B}} V_{2}^{(0,1)},\overline{\partial}\right), 
\end{equation}

\noindent where $V^{(0,1)}_{2}:=\mathcal{O}_{q}(\mathrm{Sp}_{4}/\mathrm{L}_{\alpha_{2}})^{+}/\,\mathcal{I}_{2}$ is the associated cotangent space. We may represent the relations between the two tangent spaces as in Figure  \ref{figure:embedding_tangent_C2}.  

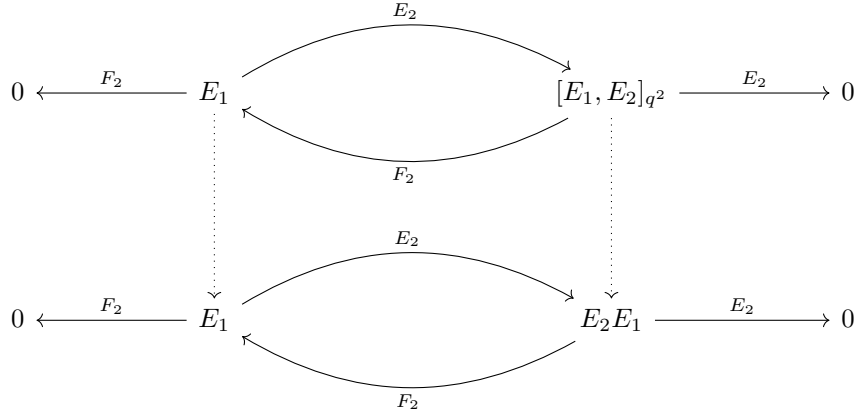
\begin{figure}[!htbp]
\centering
\begin{tikzcd}[row sep=1.2cm, column sep=1cm]
0                     &  & E_{1} \arrow[rrrr, "E_{2}", bend left] \arrow[ll, "F_{2}"'] \arrow[dd, dotted] &  &  &  & {[E_{1},E_{2}]_{q^2}} \arrow[llll, "F_{2}", bend left] \arrow[rr, "E_{2}"] \arrow[dd, dotted] &  & 0 \\
                      &  &                                                                                &  &  &  &                                                                                             &  &   \\
0  &  & E_{1} \arrow[ll, "F_{2}", swap] \arrow[rrrr, "E_{2}", bend left]        &  &  &  &  \arrow[llll, "F_{2}", bend left] E_{2}E_{1} \arrow[rr, "E_{2}"]                                                              &  & 0
\end{tikzcd}
\caption{The tangent spaces $\underline{T}_{2}^{(0,1)}$ and $T_{2}^{(0,1)}$}
\label{figure:embedding_tangent_C2}
\end{figure}

For the following discussion, it is useful to write down a dual pairing between elements of the tangent space $\underline{T}^{(0,1)}_{2}$ described above and the elements $u_{ij}\in\mathcal{O}_{q}(\mathrm{Sp}_4)$. The pairing descends from the non-degenerate pairing of dual Hopf algebras $\mathcal{O}_{q}(\mathrm{Sp}_4)\times \mathrm{U}_{q}(\mathfrak{sp}_{4})\to\mathbb{C}$ and from the relations in \eqref{eq:dual_pairing_C2}. We summarise the result in Table \ref{table:root_vectors_dual_pairing_C2_up}, which provides two bases for the anti-holomorphic cotangent space $\Lambda_{2}^{(0,1)}$.  Similar pairings may be deduced at the level of the quantised flag manifold $\mathcal{O}_{q}(\mathrm{Sp}_{4}/\mathrm{L}_{\alpha_{2}})$ by passing to the quotient via the ideal defining the homogeneous space, and verifying that the relevant relations descend compatibly to the quotient algebra. 

\begin{prop} 
The following elements generate $\mathcal{O}_{q}(\mathrm{Sp}_{4}/\mathrm{L}_{\alpha_{2}})$:

	\begin{equation}
	\label{eq:generators_of_sp(4)_flag}
		z_{ij}:=u_{i4}S(u_{4j}),  \quad i,j=1,\cdots 4.
	\end{equation}

	\begin{proof}
	We easily verify that these elements are invariant under the action of $\mathrm{U}_{q}(\ell_{\alpha_{2}})$: first of all we notice $\Delta(z_{ij})=\sum_{a,b}u_{ia}S(u_{bj})\otimes z_{ab}$. Furthermore, we have 
	\begin{align*}
	\sum_{a,b}u_{ia}S(u_{bj})\langle E_{2},z_{ab}\rangle & = 0, & 	 \sum_{a,b}u_{ia}S(u_{bj})\langle F_{2},z_{ab}\rangle & = 0,  \\ 
	\sum_{a,b}u_{ia}S(u_{bj})\langle K_{2}^{\pm 1},z_{ab}\rangle & = u_{i4}S(u_{4j}) & \sum_{a,b}u_{ia}S(u_{bj})\langle K_{1}^{\pm 1},z_{ab}\rangle & = u_{i4}S(u_{4j}).
	\end{align*}
 For the reverse inclusion, apply \cite[Theorem~2.5]{Sto03}.
\end{proof}
\end{prop}

\noindent We now build a basis for the cotangent space $V_{2}^{(0,1)}$ on the quantised flag manifold $\mathcal{O}_{q}(\mathrm{Sp}_{4}/\mathrm{L}_{\alpha_{2}})$ in terms of the generators $z_{ij}$ and the dual pairing with the generators of $T_{2}^{(0,1)}$. Moreover, we also deduce the pairing of the element $[E_{1},E_{2}]_{q^2}\in\underline{T}^{(0,1)}_{2}$ with the same set of generators.

\begin{prop} Let $\Lambda^{(0,1)}_{2}$ be the dual space to $\underline{T}_{2}^{(0,1)}$ with respect to the non-degenerate  dual pairing $\langle -,-\rangle\colon \mathrm{U}_{q}(\mathfrak{sp}_{4})\times  \mathcal{O}_{q}(\mathrm{Sp}_{4}) \to \mathbb{C}$, and let $V^{(0,1)}_{2}$ be the dual space to $T_{2}^{(0,1)}$ with respect to the non-degenerate  dual pairing $\langle -,-\rangle\colon \mathcal{O}_{q}(\mathrm{Sp}_{4}/\mathrm{L}_{\alpha_{2}})\times \mathcal{O}_{q}(\mathrm{Sp}_{4}/\mathrm{L}_{\alpha_{2}})^{\circ} \to \mathbb{C}$. Table \ref{table:root_vectors_dual_pairing_C2} provides two bases for the latter.

\begin{proof}
	The proof is an easy calculation using the dual pairing in Equation \eqref{eq:dual_pairing_C2}.
\end{proof}
\end{prop}

\noindent We now explicitly write down the $\mathcal{O}_{q}(\mathrm{Sp}_4)$-module structure of the antiholomorphic cotangent space $\Lambda^{0,1}_{2}$. Generally speaking, the action will send a basis element into a linear combination of basis elements, as 

\begin{equation*}
	\begin{aligned}
		e_{\beta}\triangleleft u_{ij}& =[u_{\beta}]\triangleleft u_{ij}=[u_{\beta}u_{ij}] =\sum_{\gamma\in\Delta^{+}}\theta^{\gamma}_{ij\beta}e_{\gamma}. 
	\end{aligned}
\end{equation*}

\noindent Dually pairing the last identity with a general root vector $E_{\gamma}\in T_{2}^{(0,1)}$, we deduce  

	\begin{equation}
	\begin{aligned}
	\langle E_{\gamma},e_{\beta}\triangleleft u_{ij}\rangle  = \theta_{ij\beta}^{\gamma}, 
		\end{aligned}
	\end{equation}	
	
\noindent and thus, understanding how the $\mathcal{O}_{q}(\mr{Sp}_{4})$-action operates on the cotangent space directly translates to the calculation of the complex numbers $\theta^{\gamma}_{ij\beta}$. Hence, the module structure is easily deduced from the coproduct of the root vectors of $T_{2}^{(0,1)}$ in Proposition \ref{prop:root_vectors_C2}. For clarity, we present the results of the calculations diagrammatically in Figure \ref{figure:module_structure_C2}. In particular, given positive roots $\beta,\gamma$ corresponding to elements $e_{\beta},e_{\gamma}\in\Lambda_{2}^{(0,1)}$, we draw an arrow from $\beta\to\gamma$ if the action $e_{\beta}\triangleleft u_{ij}$ sends $e_{\beta}$ to a direct sum in which $e_{\gamma}$ appears, that is if at least one of the $\theta_{ij\beta}^{\gamma}$ coefficients is non-zero. For the sake of intelligibility of the diagram, we choose to omit loops. The result is the diagram in Figure \ref{figure:module_structure_C2}.

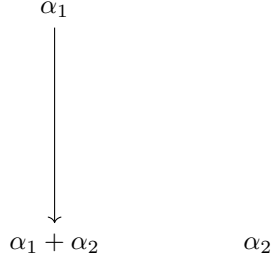
\begin{figure}[!htpb]
\centering	
\begin{tikzcd}[row sep=1.3cm, column sep=0.8cm]
\alpha_{1}  \arrow[dd]        &  &                                                                  \\
                                                                 &  &                                                                  \\
\alpha_{1}+\alpha_{2} &  & \alpha_{2}
\end{tikzcd}
\caption{The $\mathcal{O}_{q}(\mathrm{Sp}_{4})$-module structure of $\Lambda^{(0,1)}_{2}$}
\label{figure:module_structure_C2}
\end{figure}

\noindent We now explicitly write down the relations for the maximal prolongation of the FODC 
\[
\left(\Omega^{(0,1)}_{2}\left(\mathrm{Sp}_{4}/\,\mathrm{L}_{\alpha_{2}}\right):=\mathcal{O}_{q}(\mathrm{Sp}_{4})\square_{\pi_{B}} V_{2}^{(0,1)},\overline{\partial}\right)
\]
 on $\mathcal{O}_{q}(\mathrm{Sp}_{4}/\mathrm{L}_{\alpha_{2}})$. The relations will be calculated using Theorem 3.1 of \cite{Re-Jun-Ale-io} which, substantially, describes how the relations for the maximal prolongation of a FODC $\Omega^{1}(B)$ on a quantum homogeneous space $B\subseteq A$ can be understood in terms of the bimodule map of a torsion-free bimodule connection $\nabla\colon \Omega^{1}(B)\to \Omega^{1}(B)\otimes_{B} \Omega^{1}(B)$.
 
  Specialising to the setting of interest, it is already known that for any irreducible quantised flag manifold equipped with the Heckenberger--Kolb FODC \cite{hek-kolb-2006} there is a unique covariant connection that in particular is a Levi-Civita connection \cite{Levi-Civita_connections}, and therefore a torsion-free bimodule connection. Here, we apply the same machinery to the anti-holomorphic tangent space of $\mathcal{O}_{q}(\mathrm{Sp}_{4}/\mathrm{L}_{\alpha_{2}})$ analysed in the previous section. As already mentioned, non-irreducibility implies that the tangent space is not irreducible as a module of the quantised Levi subalgebra.
  
   Our analysis will be performed with a focus on the irreducible $\mathrm{U}_{q}(\ell_{\alpha_{2}})$-submodule of $T^{(0,1)}_{2}$ that is generated by the $\mathrm{U}_{q}(\ell_{\alpha_{2}})$-action on the element $E_{1}\in\mathrm{U}_{q}(\mathfrak{sp}_{4})$, that is the root vector corresponding to the only simple root excluded from $\mathrm{U}_{q}(\ell_{\alpha_{2}})$ itself. We adopt this choice because it closely mirrors the construction of Heckenberger–Kolb for covariant differential calculi in the irreducible case. Notably, in \cite{Arnab-Reamonn}, the existence of a torsion-free bimodule connection was shown for any 1-crossed bundle, therefore also encompassing the setting here at hand.
   
   Furthermore, we will also describe the other irreducible $\mathrm{U}_{q}(\ell_{\alpha_{2}})$-submodule   $T^{(0,1)}_{\theta}$, corresponding to the highest root of $\mathfrak{sp}(4)$. In particular, we will show how the Lusztig differential calculus on $\mathcal{O}_{q}(\mathrm{Sp}_{4})$ studied in \cite{COP} fails to give a restriction to the latter, and to the full flag, that has classical dimension. Finally, we deduce an explicit expression for the lift of such element $E_{\theta}$ to a three-dimensional $\mathcal{O}_{q}(\mathrm{Sp}_{4})$ tangent space, spanning a one-dimensional $\mathrm{U}_{q}(\ell_{\alpha_{2}})$-submodule, according to the well-known picture emerging from the classical differential geometry. 

\begin{defn}
Let $A$ be an algebra, $(\Omega^{\bullet}(A),\wedge,\mathrm{d})$ a DC over $A$, and $\mathcal{M}\in{}_{A}\mathrm{Mod}$. 
\begin{enumerate}
\item A \emph{connection} on $\mathcal{M}$ is a $\mathbb{C}$-linear map $\nabla\colon \mathcal{M}\to \Omega^{1}(A)\otimes_{A} \mathcal{M}$ satisfying the identity 
\begin{equation}
	\nabla(a m) = \mathrm{d}a\otimes_{A} m + a\nabla(m)\, \qquad a\in A, \ m\in \mathcal{M}. 
\end{equation}
\item A left bimodule connection on $\mathcal{M}\in {}_{A}\mathrm{Mod}_{A}$ is a pair consisting of a left connection $\nabla$ and a \emph{bimodule map} $\sigma\colon \mathcal{M}\otimes_{A}\Omega^{1}(A)\to \Omega^{1}(A)\otimes_{A}\mathcal{M}$ satisfying 
\begin{equation}
	\nabla( m a ) = \nabla(m)a+ \sigma(m\otimes_{A}\mathrm{d}a). 
\end{equation}
\item Considering $\mathcal{M}=\Omega^{1}(A)$, the \emph{torsion} of a connection is the left $A$-module map
\begin{equation}
\begin{aligned}
	T_{\nabla}\colon\Omega^{1}(A)& \to \Omega^{2}(A), \\
	 \omega_{A} & \mapsto \left(\wedge\circ \nabla - \mathrm{d}\right)\omega_{A},  
\end{aligned}
\end{equation}
and if $T_{\nabla}=0$ we say the connection $\nabla$ is \emph{torsion free}. 
\end{enumerate}
\end{defn}
We start with a theorem showing how the existence of a torsion-free bimodule connection on a FODC gives the relations of its maximal prolongation in a relatively simple, compact form. 
\begin{thm}[\cite{Re-Jun-Ale-io}, Theorem 3.1]\label{theorem:relations_max_prolongation_flag}
 Let $\nabla$ be a bimodule connection on $\Omega^1(A)$ that is torsion-free with respect to the maximal prolongation of $\Omega^1(A)$, and let $\sigma$ denote its associated bimodule map. It holds that the subbimodule $N^{(2)}$ containing the relations for the maximal prolongation of $\Omega^{1}(A)$ is generated as an $A$-bimodule by the sets 
	\begin{equation}\label{eq:sets_of_relations_max_prolongation}
	\begin{aligned}
		G_{1}&:=\{\mathrm{d}a\otimes_{A} \mathrm{d}a'+\sigma(\mathrm{d}a\otimes_{A}\mathrm{d}a') \mid a,a'\in A\},  \\
		G_{2}&:=\{\nabla(\mathrm{d}a)\mid a\in A\}.  
	\end{aligned}
	\end{equation}
\end{thm}

The expression for the bimodule map associated to a torsion-free bimodule connection is explicitly given once a covariant differential calculus over $A$ is fixed, and in particular it depends only on the bimodule connection itself. For a given left $A$-covariant FODC in ${}^{A}_{B}\mathrm{Mod}_{0}$ (recall the definition from Section \ref{section:hopf_modules}) on a quantum homogeneous space $B:=A^{\mathrm{co\pi_{B}(A)}}$, we have (see \cite{Re-Jun-Ale-io}, Lemma 4.1)
\begin{equation}
	\sigma(\mathrm{d}b\otimes_{B} \mathrm{d}b')=\mathrm{d}(b_{(-2)}\,b'S(b_{(-1)}))\otimes_{B}\mathrm{d}b_{(0)}. 
\end{equation}

In particular, consider now a quantised flag manifold $B:=\mathcal{O}_{q}(\mathrm{G}/\mathrm{L}_{S})\subseteq A:=\mathcal{O}_{q}(\mathrm{G})$, and fix such a FODC $\Omega^{1}(B)$ over $B$. As $\sigma\colon\Omega^{1}(B)\otimes_{B}\Omega^{1}(B)\to \Omega^{1}(B)\otimes_{B}\Omega^{1}(B)$ is a morphism in the category of left $A$-covariant $B$-bimodules $^{A}_{B}\mathrm{Mod}_{0}$ we may take advantage of Theorem \ref{thm:takeuchi_equivalence} to construct a morphism in the category ${}^{\pi_{B}(A)}\mathrm{Mod}_{0}$ of left $\pi_{B}(A)$-covariant right $B$-modules starting from the latter, where $\pi_{B}\colon A\to A/ B^{+}A$ and $\pi_{B}(A)=\mathrm{Im}(\pi_{B})$. In particular, denoting by abuse of notation the induced morphism $\sigma\colon V^{(0,1)}_{S}\otimes V^{(0,1)}_{S} \to V^{(0,1)}_{S}\otimes V^{(0,1)}_{S} \in {}^{\pi_{B}(A)}\mathrm{Mod}_{0}$ , it holds that
\begin{equation}\label{eq:sigma}
	\sigma([b]\otimes [c])=[c]\triangleleft S(b_{(1)})\otimes [b_{(2)}],\qquad \sigma^{-1}([b]\otimes [c]) = [c_{(2)}]\otimes [b]\triangleleft c_{(1)},
\end{equation}
where the square brackets are cosets in $V^{(0,1)}_{S}\cong B^{+}/\mathcal{I}_S$. 

In view of the last remarks, the calculation of the relations in $G_1$ for the maximal prolongation of a FODC on a quantised flag manifold where a torsion-free bimodule connection exists is essentially reduced to the evaluation of Equation \eqref{eq:sigma} on the cotangent space defining the calculus. In the next lemma, we discuss an even more convenient approach to the calculation of the relations in $G_{1}$. As a consequence of the non-degenerate pairing of Hopf algebras in Equation \eqref{eq:dual_pairing_C2}, we can indeed produce the relations in $G_{1}$ by dualising 
 
 \[
 \left(\sigma\colon V_{S}^{(0,1)}\otimes V_{S}^{(0,1)}\to V_{S}^{(0,1)}\otimes V_{S}^{(0,1)}\right) \longrightarrow  \left(\widecheck\sigma\colon T_{S}^{(0,1)}\otimes T_{S}^{(0,1)} \to T_{S}^{(0,1)}\otimes T_{S}^{(0,1)}\right),   
 \]
 where $T_{S}^{(0,1)}\subseteq \mathcal{O}_{q}(\mathrm{Sp}_{4}/ \mathrm{L}_{S})^{\circ}$. We present this result as a general statement for those quantised flag manifolds realised as invariant spaces under the action of a quantised Levi subalgebra corresponding to a single crossed node of the Dynkin diagram for the underlying Lie algebra. Notice that the same result appears in \cite{Re-Jun-Ale-io} in a more general setting.

\begin{lemma}\label{lemma:bimodule_map_tangent_space_sigma}
Let $\mathcal{O}_{q}(\mathrm{G}/\mathrm{L}_{S})$ be a quantised flag manifold corresponding to a choice of a single simple root $\Pi\backslash S=\{\alpha_{x}\}$ in the definition of $\mathrm{U}_{q}(\ell_{S})$, and let $T_{S}^{(0,1)}:=\mathrm{U}_{q}(\ell_{S})E_{x}$ be a quantum tangent space for $\mathcal{O}_{q}(\mathrm{G}/\mathrm{L}_{S})$. A $U_{q}(\ell_{S})$-module map $\widecheck{\sigma}\colon T_{S}^{(0,1)}\otimes T_{S}^{(0,1)} \to T_{S}^{(0,1)}\otimes T_{S}^{(0,1)}$ is given by  
\begin{equation}\label{eq:definition_sigma_check}
\widecheck\sigma(YE_{x}\otimes ZE_{x}):= Y_{(1)}K_{x}^{-1}S(Y_{(3)})ZE_{x}\otimes Y_{(2)}E_{x}, 
\end{equation} 
where $Y,Z\in \mathrm{U}_{q}(\ell_{S})$. 

Moreover, this map is dual to the bimodule map $\sigma\colon V_{S}^{(0,1)}\otimes V_{S}^{(0,1)} \to V_{S}^{(0,1)}\otimes V_{S}^{(0,1)}$ corresponding to a torsion free bimodule connection on $\big(\mathcal{O}_{q}(\mathrm{G})\square_{\pi_{B}}V^{(0,1)}_{S},\overline{\partial}\big)$.
\proof For any $E_{\alpha},E_{\beta}\in T^{(0,1)}_{S}$ and any $[b],[c]\in V^{(0,1)}_{S}:=\mathcal{O}_{q}(G/\mathrm{L}_{S})^{+}/\mathcal{I}$, we have 
\begin{equation}\label{eq:sigma_check}
\begin{aligned}
	\langle \widecheck{\sigma}(E_{\alpha}\otimes E_{\beta}), [b]\otimes [c]\rangle & = \langle E_{\alpha}\otimes E_{\beta}, \sigma([b]\otimes [c])\rangle \\ 
	& = \langle E_{\alpha}\otimes E_{\beta}, [b_{(1)} \,c\, S(b_{(2)})]\otimes [b_{(3)}] \rangle \\ 
	& = \langle \iota(E_{\alpha})\otimes \iota(E_{\beta}), b_{(1)} \,c\, S(b_{(2)})\otimes b_{(3)} \rangle \\
	& = \langle E_{\alpha}, b_{(1)} \,c\, S(b_{(2)})\rangle \langle E_{\beta}, b_{(3)}\rangle \\ 
	& = \langle (E_{\alpha})_{(1)},b_{(1)}\rangle \langle (E_{\alpha})_{(2)},c)\rangle \langle (E_{\alpha})_{(3)},S(b_{(2)})\rangle  \langle E_{\beta}, b_{(3)}\rangle\\ 
	& = \langle (YE_{x})_{(1)},b_{(1)}\rangle \langle (YE_{x})_{(2)},c)\rangle \langle (YE_{x})_{(3)},S(b_{(2)})\rangle  \langle E_{\beta}, b_{(3)}\rangle,\\ 
\end{aligned}
\end{equation}
where in the last line we used that every $E_{\alpha}\in T^{(0,1)}_{S}$ can be written as the action of the quantised Levi subalgebra on the element $E_{x}$ corresponding to the root $\alpha_{x}$. Now, as the iterated coproduct of an element of the form  $YE$ reads 
\begin{equation}
	(\Delta\otimes \mathrm{id})\Delta(YE) = Y_{(1)}E\otimes Y_{(2)}K\otimes Y_{(3)}K + Y_{(1)}\otimes Y_{(2)}E\otimes Y_{(3)}K + Y_{(1)}\otimes Y_{(2)} \otimes Y_{(3)}E, 
\end{equation}
the dual pairing in Equation \eqref{eq:dual_pairing_C2} shows that the first and last summands of this expression give zero contributions to Equation \eqref{eq:sigma_check}. Therefore, the only nonzero contribution is due to the second summand, from which we get 
\begin{align*}
		\langle \widecheck{\sigma}(YE_{x}\otimes E_{\beta}), [b]\otimes [c]\rangle & = \langle Y_{(1)},b_{(1)}\rangle \langle Y_{(2)}E_{x},c\rangle \langle S(Y_{(3)}K_{x}),b_{(2)}\rangle \langle E_{\beta},b_{(3)}\rangle \\ 
		& = \langle Y_{(1)}K_{x}^{-1}S(Y_{(3)})E_{\beta},b\rangle \langle Y_{(2)}E_{x},c\rangle \\ 
		& = \langle \iota(Y_{(1)}K_{x}^{-1}S(Y_{(3)})E_{\beta})\otimes \iota(Y_{(2)}E_{x}), b\otimes c\rangle \\ 
		& = \langle Y_{(1)}K_{x}^{-1}S(Y_{(3)})E_{\beta}\otimes Y_{(2)}E_{x}), [b]\otimes [c]\rangle. 
\end{align*}
The last form precisely defines $\widecheck{\sigma}$ as in Equation \eqref{eq:definition_sigma_check}, since $E_{\beta}=ZE_{x}$ for some $Z\in U_{q}(\ell_{S})$. \qed  
\end{lemma}

With the intent of discussing in more detail how the computations of the $G_2$-relations are performed, we first briefly recall a complementary description of the quadratic relations of the maximal prolongation. Let $N^{(2)}$ denote the bimodule of quadratic relations, and write
\[
q_2\colon \Omega^{1}(B)\otimes_B\Omega^{1}(B) \to \Phi\bigl(\Omega^{1}(B)\otimes_B\Omega^{1}(B)\bigr) \simeq V_S^{(0,1)}\otimes V_S^{(0,1)}
\]
for the canonical quotient map (Theorem \ref{thm:takeuchi_equivalence} and Theorem \ref{thm:takeuchi_equivalence_0}) followed by the cotangent-space identification. For a quantum homogeneous space admitting a framing calculus, such as the one considered here, Proposition A.1 of \cite{ale_re_jun} identifies $q_2(N^{(2)})$ with the image of the map
\[
\begin{aligned}
\varpi^{\otimes 2}\colon \mathcal I_S
&\to
V_S^{(0,1)}\otimes V_S^{(0,1)},\\
x&\mapsto
\bigl((\varpi\circ\pi_\varepsilon)\otimes
(\varpi\circ\pi_\varepsilon)\bigr)\Delta(x)
=
[x_{(1)}^{+}]\otimes[x_{(2)}^{+}],
\end{aligned}
\]
where the brackets are interpreted in the cotangent space of the framing calculus, using the embedding of $V_S^{(0,1)}$ into that space. Thus these tensors, for $x\in\mathcal I_S$, give the full quadratic relation space of the associated exterior algebra.

This description concerns the relations generated jointly by $G_1$ and $G_2$, but does not by itself identify the image of $G_2$ separately. Such an identification is available when the torsion-free bimodule connection $\nabla$ is induced by a strong principal connection (see e.g. \cite{BM20} for the definition) on the framing calculus, and the natural embedding $V_S^{(0,1)}\hookrightarrow\Lambda^{(0,1)}_{S}:=A^+/\widetilde{\mathcal I_S}$ is an isomorphism. For the case at hand, the FODCi under consideration are obtained by the same framing construction as in the Heckenberger--Kolb setting. In particular, the corresponding cotangent spaces  are naturally identified, and the argument of \cite[Theorem~3.3]{CO22} applies verbatim. Hence the framing calculus carries a strong principal connection given by the zero projection $\Pi=0$. The induced connection is
\[
\nabla=m^{-1}\circ(\mathrm d_A\otimes\mathrm{id}_{V_{S}^{(0,1)}}),
\]
where
\[
\begin{aligned}
m\colon\Omega^1(B)\otimes_B(A\square_{\pi_{B}} V^{(0,1)}_{S}) & \to (\Omega^1(B)A)\square_{\pi_{B}} V^{(0,1)}_{S},\\ 
\omega\otimes_B\Bigl(\sum_i a_i\otimes v_i\Bigr) & \mapsto\sum_i\omega a_i\otimes v_i.
\end{aligned}
\]
We obtain
\[
m\bigl(\nabla(\mathrm d_Bb)\bigr) =\mathrm d_A b_{(1)}\otimes[b_{(2)}^+]= b_{(1)}\otimes[b_{(2)}^+]\otimes[b_{(3)}^+].
\]
Using the framing realisation of the cotangent-space identification via $q_2$, passage to the cotangent space applies $\varepsilon$ to the first tensor factor, and hence
\[
q_2\bigl(\nabla(\mathrm d_Bb)\bigr) =[b_{(1)}^+]\otimes[b_{(2)}^+] =\varpi^{\otimes_2}(b),\quad b\in B,
\]
where the same coproduct expression now defines $\varpi^{\otimes2}$ (see Equation \eqref{eq:I2}) on $B$. Thus, under these additional hypotheses, the $G_2$-generators are represented by these tensors. By \cite[Lemma~3.5]{Re-Jun-Ale-io}, if $X$ generates $B$ as a unital algebra, then
\[
N^{(2)}
= \left\langle G_1\cup\{\nabla(d_Bx)\mid x\in X\} \right\rangle_{B}.
\]
Consequently, in addition to the $G_1$-relations, it suffices to consider the tensors
\[
q_2(\nabla(d_Bx)) =[x_{(1)}^+]\otimes[x_{(2)}^+], \quad x\in X.
\]

Finally, we now give the full set of relations for the maximal prolongation for the FODC $\left(\Omega^{(0,1)}\left(\mathrm{Sp}_{4}/\mathrm{L}_{S}\right),\overline{\partial}\right)$. 

\begin{prop}\label{prop:relations_C2_flag_S}
Let $\alpha,\beta\in\Delta_{1}^{+}$ and let $e_{\alpha},e_{\beta}$ be the corresponding basis elements in $V^{(0,1)}_{2}$. The relations for the exterior algebra of $V^{(0,1)}_{2}$ corresponding to the generating sets $G_{i}$, $i=1,2$,  are: 

\begin{equation}\label{eq:relations_C2}
\begin{aligned}
e_{i}\wedge e_{i} & = 0,  & e_{2+1}\wedge e_{1}& = - e_{1}\wedge e_{2+1},  & e_{2+1}\wedge e_{1} & =0.  	
\end{aligned}
\end{equation}

\noindent Thus, it follows that the maximal prolongation of  $\left(\Omega^{(0,1)}\left(\mathrm{Sp}_{4}/\mathrm{L}_{\alpha_{2}}\right),\overline{\partial}\right)$ is trivial, namely that $\Omega^{(0,k)}(\mathrm{Sp}_4/\mathrm{L}_{\alpha_{2}})=0$, for every $k\geq 2$. 
\begin{proof}
	The relations in the set $G_{1}$ are easily obtained using either Equation \eqref{eq:sigma_check} or Equation \eqref{eq:sigma} on the generators of $V_{2}^{(0,1)}$. In particular
	
	\begin{align*}
	\widecheck{\sigma}(E_{1}\otimes E_{1}) & = q^{-2}E_{1}\otimes E_{1}, &
	\widecheck{\sigma}(E_{2}E_{1}\otimes E_{2}E_{1}) & = q^{-2}E_{2}E_{1}\otimes E_{2}E_{1}, \\ 
	\widecheck{\sigma}(E_{1}\otimes E_{2}E_{1}) & = E_{2}E_{1}\otimes E_{1}, & 
	\widecheck{\sigma}(E_{2}E_{1}\otimes E_{1}) & = (q^{-2}-q^{2})E_{2}E_{1}\otimes E_{1} + E_{1}\otimes E_{2}E_{1}, 
	\end{align*}
	
and thus the relations in Equation \eqref{eq:relations_C2} hold. 

Furthermore, by direct inspection, the relations from the set $G_{2}$ turn out to be contained in those obtained from the set $G_{1}$, for example one finds 
\begin{align*}
\varpi^{\otimes_{2}}(z_{41}) & = [(z_{41})_{(1)}^{+}] \otimes 	[(z_{41})_{(2)}^{+}]\\ 
               & = [(u_{4a}S(u_{b1}))^{+}] \otimes [z_{ab}^{+}] \\ 
               & = [u_{44}S(u_{31})] \otimes [z_{43}] +[u_{44}S(u_{21})] \otimes [z_{42}] \\ 
               & =  q\, \langle 1,u_{44}\rangle \langle [E_{1},E_{2}]_{q^{2}},S(u_{31})\rangle e_{2+1}\otimes e_{1} -q^{3} \langle 1,u_{44}\rangle \langle E_{1},S(u_{21})\rangle e_{1}\otimes e_{2+1} \\ 
               & = q\, \langle S(E_{1}E_{2}),u_{31}\rangle e_{2+1}\otimes e_{1} - q^{3}\langle S(E_{1}),u_{21}\rangle e_{1}\otimes e_{2+1}\\ 
               & = q\langle E_{2}K_{2}^{-1}E_{1}K_{1}^{-1},u_{31}\rangle e_{2+1}\otimes e_{1}+ q^{3}\langle E_{1}K_{1}^{-1},u_{21}\rangle e_{1}\otimes e_{2+1}\\ 
                 & = q^{3}\langle E_{2}E_{1}K_{2}^{-1}K_{1}^{-1},u_{31}\rangle e_{2+1}\otimes e_{1}+ q^{3}\langle E_{1}K_{1}^{-1},u_{21}\rangle e_{1}\otimes e_{2+1}\\ 
               & = q^{4}( e_{2+1}\otimes e_{1} + e_{1}\otimes e_{2+1}),  
\end{align*}
with the corresponding relation for the exterior algebra being again $e_{2+1}\wedge e_{1} =- e_{1}\wedge e_{2+1}$.
\end{proof}
\end{prop}

\begin{rmk}
As is evident from the relations in Equation \eqref{eq:relations_C2}, it turns out that the exterior algebra $V_{2}^{(0,\bullet)}$ supplies a differential calculus over $\mathcal{O}_{q}(\mathrm{Sp}_{4})$ that is not a $q$-deformed version of the de Rham complex of the underlying flag variety. In particular, $\Omega^{(0,\bullet)}_{2}(\mathrm{Sp}_{4})$ is not of classical dimension, as the dimension is strictly smaller than its classical counterpart, and in particular vanishes already at the level of the differential 2-forms. \qed 
\end{rmk}

The collapse of the exterior algebra is not confined to flag manifolds outside the $\mathrm{A}$-series. As the following remark shows, a similar phenomenon occurs for a differential calculus on the full quantised flag manifold associated to $\mathfrak{sl}_{3}$, whose maximal prolongation vanishes in all degrees greater than one. As this paper is mainly concerned with the $\mathrm{B},\mathrm{C}$ and $\mathrm{D}$ cases, we omit providing full details for $\mathrm{U}_{q}(\mathfrak{sl}_{3})$ and refer the reader to \cite{ale_re_jun,Re-Sombrg} for further information.  

\begin{rmk}
Consider the quantised enveloping algebra $\mathrm{U}_{q}(\mathfrak{sl}_{3})$ alongside its finitary dual, namely the $q$-deformed coordinate algebra $\mathcal{O}_{q}(\mathrm{SU}_{3})$. We now show that a similar collapse of the maximal prolongation also occurs in the $\mathrm{A}$-series, by considering the full quantised flag manifold corresponding to the choice $\mathrm{U}_{q}(\ell_{\emptyset})=\mathrm{U}_{q}(\mathfrak{h}) =\langle K_{1}^{\pm 1},K_{2}^{\pm 1}\rangle$. Its underlying classical flag manifold carries a holomorphic contact structure, providing a further point of comparison with the examples discussed above.
\[
\begin{tikzpicture}
\tikzstyle{smallcircle} = [circle, draw=black, fill=lightgray,
minimum size=2.5mm, inner sep=0pt]
\tikzstyle{cross} = [minimum size=3mm, inner sep=0pt]
\node[cross, label=above:$\alpha_1$] (n1) at (0,0)
{\scalebox{1.5}{$\times$}};
\node[cross, label=above:$\alpha_2$] (n2) at (2,0)
{\scalebox{1.5}{$\times$}};
\draw (n1) -- (n2);
\end{tikzpicture}
\]
Building on the results of the previous section, we consider the quantum tangent space $T^{(0,1)}_{2}=\mathrm{Span}_{\mathbb{C}}\left\{E_{1},E_{2}\right\}$ and the associated first order differential calculus
\[
\left(\Omega^{(0,1)}_{2} \left( \mathrm{SU}_{3}/\mathrm{L}_{\alpha_{1+2}}\right):=\mathcal{O}_{q}(\mathrm{SU}_{3}) \square_{\pi_{B}} V^{(0,1)}_{2},\overline{\partial}\right),
\]

\noindent where
$V^{(0,1)}_{2}:= \mathcal{O}_{q}(\mathrm{SU}_{3}/\mathrm{L}_{\alpha_{1+2}})^{+}/\mathcal{I}_2$ is the corresponding quantum cotangent space. Unlike the classical-dimensional FODCi considered above, this calculus has a two-dimensional quantum tangent space, as the root vector corresponding to $\alpha_{1}+\alpha_{2}$ is not included. It therefore does not provide a deformation of the full classical antiholomorphic Dolbeault complex already at the level of the differential 1-forms. Nevertheless, this example allows us to examine the same higher-order phenomenon shown to hold for the contact flag of $\mathrm{C}_{2}$ even in an $A$-series example. In fact, by a direct application of Theorem~\ref{theorem:relations_max_prolongation_flag} we obtain

\begin{align*}
 \sigma(e_{1}\otimes e_{2})&  =q\,e_{2}\otimes e_{1}, &   \sigma(e_{2}\otimes e_{1})&=q\,e_{1}\otimes e_{2}, \\
 \sigma(e_{1}\otimes e_{1})&=q^{-2}\,e_{1}\otimes e_{1}, &  \sigma(e_{2}\otimes e_{2})& =q^{-2}\,e_{2}\otimes e_{2},
\end{align*}

\noindent and the corresponding quadratic relations force all products of degree-one elements to vanish:

\begin{equation}
e_{1}\wedge e_{1} =0, \qquad e_{2}\wedge e_{2}=0, \qquad e_{1}\wedge e_{2}=0, \qquad   e_{2}\wedge e_{1}=0.
\end{equation}

\noindent  Here, $e_{1},e_{2}\in V_{2}^{(0,1)}$ denote the dual elements to $E_{1},E_{2}\in T^{(0,1)}_{2}$ with respect to the restriction of the non-degenerate dual pairing of Hopf algebras $\langle-,-\rangle\colon \mathrm{U}_{q}(\mathfrak{sl}_{3}) \times\mathcal{O}_{q}(\mathrm{SU}_{3})\to\mathbb{C}$. One can also check that no non-trivial relations corresponding to the set $G_{2}$ in Equation \eqref{eq:sets_of_relations_max_prolongation} are produced. 

Consequently, the exterior algebra $V_{S}^{(0,\bullet)}$ vanishes in every degree at least two, while its degree-zero and degree-one components remain nonzero. In particular, even the single degree-two component expected for the ordinary exterior algebra of a two-dimensional space is lost.

This example thus exhibits a particularly pronounced form of the collapse observed above and shows that this behaviour, for the calculi constructed here, is not confined to Lie types outside the $\mathrm{A}$-series. However, although we do not provide further details here, we emphasise that the expectation of recovering a FODC of classical dimension on this contact quantised flag manifold becomes quite plausible once we take into account the root vector corresponding to the positive root $\alpha_1+\alpha_2$.  \qed 
\end{rmk}

\subsection{A Reeb extended Dolbeault complex for the contact flag manifold $\mathcal{O}_{q}(\mathrm{Sp}_{4}/\mathrm{L}_{\alpha_{2}})$}
\label{section:reeb_calculus_C2}
The relations for the exterior algebra $V^{(0,\bullet)}_{2}$ established in the last section are relative to the irreducible $\mathrm{U}_{q}(\ell_{\alpha_{2}})$-submodule $T^{(0,1)}_{2} := \mathrm{U}_{q}(\ell_{\alpha_{2}})E_{1}$. However, as already argued, the full antiholomorphic tangent space for the contact flag manifold $\mathcal{O}_{q}(\mathrm{Sp}_{4}/\mathrm{L}_{\alpha_{2}})$ is not irreducible as a $\mathrm{U}_{q}(\ell_{\alpha_{2}})$-module, but rather decomposes as a direct sum
\[
T^{(0,1)}_{3} := T^{(0,1)}_{2} \oplus T^{(0,1)}_{\theta}
\]
of irreducible $\mathrm{U}_{q}(\ell_{\alpha_{2}})$-modules. This decomposition reflects the fact that the corresponding classical flag variety is not of Hermitian symmetric type, and therefore its holomorphic tangent bundle is not itself irreducible as a homogeneous vector bundle. In order to maintain conformity with the classical differential geometric picture, the summand $T^{(0,1)}_{\theta}$ must be a $1$-dimensional $\mathrm{U}_{q}(\ell_{\alpha_{2}})$-module whose generating element corresponds to the highest root of the underlying Lie algebra $\mathfrak{sp}_{4}$.

The statement of the next result relies on the notion of Lusztig root vectors and the associated quantum tangent space; since its precise definition plays no essential role in the present discussion, we postpone it to Section \ref{section:lusztig_for_so(8)} (see, in particular, Equation \eqref{eq:definition_of_root_vectors}) where the definition is presented in the context of $\mathcal{O}_{q}(\mathrm{SO}_{8})$. Roughly speaking, these are quantum tangent spaces for the quantised coordinate algebra of a complex, connected, simply connected semisimple Lie group, arising from \emph{quantum root vectors} obtained from suitable reduced decompositions of the longest element of the Weyl group of the corresponding root system. For further details on Lusztig differential calculi, in particular for the case of $\mathrm{C}_{2}$ and its quantised flag manifolds, see \cite{COP}.
	
\begin{prop}
\label{prop:tangent_space_with_theta_C2}
There exists a unique three-dimensional quantum tangent space for $\mathcal O_q(\mathrm{Sp}_4/\mathrm{L}_{\alpha_{2}})$ which decomposes, as a $\mathrm{U}_q(\mathfrak \ell_{\alpha_{2}})$-module, into the two-dimensional irreducible module $T_2^{(0,1)}$ and a one-dimensional irreducible module 
	\[
	T_{\theta}^{(0,1)}:=\mathrm{U}_{q}(\ell_{\alpha_{2}})E_{\theta},  \quad E_{\theta}:=E_{2}E_{1}^{2}-(q^{2}+q^{-2})E_{1}E_{2}E_{1}.
	\]
	Moreover, $T^{(0,1)}$ cannot be understood as a restriction to $\mathcal{O}_{q}(\mathrm{Sp}_{4}/\mathrm{L}_{\alpha_{2}})$ of an $\mathcal{O}_{q}(\mathrm{Sp}_{4})$ Lusztig tangent space.  
\begin{proof}
 Using the relations in Equation \eqref{eq:algebra_rel_C2} and Equation \eqref{eq:serre_c2}, we find that $E_{\theta}$ is closed under the $\mathrm{U}_{q}(\ell_{\alpha_{2}})$-action up to multiplication by integer powers of $q$. Moreover, $\Delta(T^{(0,1)}_{\theta})\subseteq \left(T^{(0,1)}_{3}\oplus 1\right)\otimes \mathcal{O}_{q}(\mathrm{Sp}_{4}/\mathrm{L}_{\alpha_{2}})^{\circ}$, so $T^{(0,1)}_{3}$ is indeed a 3-dimensional $\mathcal{O}_{q}(\mathrm{Sp}_{4}/\mathrm{L}_{\alpha_{2}})$-tangent space, and it is unique as $E_{\theta}$ is the only possible combination such that $\mathrm{dim}(\mathrm{U}_{q}(\ell_{\alpha_{2}})E_{\theta})=1$. To see this, one considers a general linear combination of monomials with root content $2\alpha_{1}+\alpha_{2}$ and checks that closure as a $1$-dimensional $\mathrm{U}_{q}(\ell_{S})$-module implies the claimed expression. Furthermore, the last statement follows as neither of the (Lusztig) tangent spaces on $\mathcal{O}_{q}(\mathrm{Sp}_{4})$ arising from one of the two reduced decompositions of the longest element of the Weyl group correctly restricts to $E_{\theta}$.  
\end{proof}
\end{prop}

 Letting $\mathcal{I}_{3}$ be the right $\mathcal{O}_{q}(\mathrm{Sp}_{4}/\mathrm{L}_{\alpha_{2}})$-ideal in $\mathcal{O}_{q}(\mathrm{Sp}_{4}/\mathrm{L}_{\alpha_{2}})^{+}$ corresponding to the tangent space $T^{(0,1)}_{3}$ according to Theorem \ref{th:hermisson_class}, we have a left $\mathcal{O}_{q}(\mathrm{Sp}_{4})$-covariant FODC on $\mathcal{O}_{q}(\mathrm{Sp}_{4}/\mathrm{L}_{\alpha_{2}})$ as 

\begin{equation}
\left(\underline{\Omega}^{1}_{3}(\mathrm{Sp}_{4}/\mathrm{L}_{\alpha_{2}}):=\mathcal{O}_{q}(\mathrm{Sp}_{4})\square_{\pi_{B}} V_{3}^{(0,1)},\overline{\partial}\right), 
\end{equation}

\noindent where $V_{3}^{(0,1)}:=\mathcal{O}_{q}(\mathrm{Sp}_{4}/\mathrm{L}_{\alpha_{2}})^{+}/\,\mathcal{I}_{3}$ is the associated cotangent space. We call such FODC the \emph{Reeb calculus} of $\mathcal{O}_{q}(\mathrm{Sp}_{4}/\mathrm{L}_{\alpha_{2}})$. We now produce a framing calculus in $\mathcal{O}_{q}(\mathrm{Sp}_{4})$, that we will also call the \emph{Reeb calculus}, from the latter. 

\begin{lemma}\label{lemma:the_3d_tangent_space_c2}
	 The space $\underline{T}_{3}^{(0,1)}:=\underline{T}^{(0,1)}_{2}\oplus \underline{T}^{(0,1)}_{\theta}$, where 
	\[
	\underline{T}_{\theta}^{(0,1)}:=\underline{E}_{\theta}\triangleleft_{\mathrm{Ad}}\mathrm{U}_{q}(\ell_{S}),  \quad \underline{E}_{\theta}:=[[E_{2},E_{1}]_{q^{2}},E_{1}]_{q^{-2}}, 
	\]
 is a 3-dimensional tangent space for $\mathcal{O}_{q}(\mathrm{Sp}_{4})$. 
 \begin{proof}
 	As we already established that $\underline{T}^{(0,1)}$ is a 2-dimensional $\mathcal{O}_{q}(\mathrm{Sp}_{4})$ tangent space, we are basically left to the explicit calculation for the coproduct of $\underline{E}_{\theta}$, for which we obtain
 	\begin{align*}
 	\Delta(	[[E_{2},E_{1}]_{q^{2}},E_{1}]_{q^{-2}} ) & = [[E_{2},E_{1}]_{q^{2}},E_{1}]_{q^{-2}}\otimes K_{2}K_{1}^{2} \\ 
 	 & \quad + 1 \otimes [[E_{2},E_{1}]_{q^{2}},E_{1}]_{q^{-2}} \\ 
 	 & \quad + (1-q^{-2})E_{1}\otimes [E_{2},E_{1}]_{q^{2}}K_{1}\\ 
 	 & \quad + (1-q^{-2})[E_{1},E_{2}]_{q^{2}} \otimes E_{1}K_{1}K_{2}.
 	\end{align*}
	Furthermore, one checks that $\mathrm{Span}_{\mathbb{C}}\left\{\underline{E}_{\theta}\right\}$ is a 1-dimensional $\mathrm{U}_{q}(\ell_{\alpha_{2}})$-module.
 \end{proof}
\end{lemma}

\begin{rmk}
	It is useful to record how the tangent space constructed above behaves under projection to the corresponding cominuscule partial flag manifold, since this provides a direct consistency check with the Heckenberger--Kolb classification and clarifies the relation between the present calculus on $\mathcal{O}_{q}(\mathrm{Sp}_{4})$ and the standard covariant calculus on the quotient. 	Notice that the tangent space established in the previous lemma projects onto a $2$-dimensional tangent space of the cominuscule $\mathrm{C}_{2}$ quantised flag manifold whose closure under the action of the quantised Levi subalgebra leads to the three dimensional Heckenberger--Kolb $\mathcal{O}_{q}(\mathrm{Sp}_{4})$-covariant FODC. Indeed, upon projection to the flag manifold, the only non-trivial root vectors that survive are
	\[
	E_{1}E_{2}, \qquad E_{1}^{2}E_{2}.
	\]
	Closing their span under the $\mathrm{U}_{q}(\ell_{\alpha_{1}})$-action then recovers the remaining root vector $E_{2}$, while all possible higher contributions vanish as a consequence of the $q$-Serre relations. Hence the closure under the $\mathrm{U}_{q}(\ell_{\alpha_{1}})$-action of the projected tangent space is generated precisely by the expected root vectors spanning a three-dimensional tangent space. In particular, it coincides with the Heckenberger--Kolb tangent space of the corresponding cominuscule quantised flag manifold, in accordance with their classification theorem for finite-dimensional covariant FODCi over irreducible quantised flag manifolds.
	\qed
\end{rmk}
 
Let $\underline{\mathcal{I}}_{3}\subseteq\mathcal{O}_{q}(\mathrm{Sp}_{4})^{+}$ be the right $\mathcal{O}_{q}(\mathrm{Sp}_{4})$-ideal corresponding to the tangent space $\underline{T}^{(0,1)}_{3}$. By Woronowicz’s classification theorem, this ideal determines a left $\mathcal{O}_{q}(\mathrm{Sp}_{4})$-covariant FODC on $\mathcal{O}_{q}(\mathrm{Sp}_{4})$, namely
\begin{equation}\label{eq:full_calculus_C2}
\left(\underline{\Omega}^{1}_{3}(\mathrm{Sp}_{4}):=\mathcal{O}_{q}(\mathrm{Sp}_{4})\otimes\Lambda^{(0,1)}_{3},\overline{\partial}\right),
\end{equation}
where
\[
\Lambda^{(0,1)}_{3}\cong \mathcal{O}_{q}(\mathrm{Sp}_{4})^{+}/\underline{\mathcal{I}}_{3}
\]
is spanned by the classes of the elements listed in Table \ref{table:root_vectors_dual_pairing_C2}, together with the class of
\[
-q^{-4}(q^{2}+q^{-2})^{-1}z_{41}.
\]
A direct computation shows that this additional class is dual to the root vector $\tilde{E}_{\theta}$.  The above discussion provides a proof to the following lemma.
\begin{lemma}
The cotangent spaces $V^{(0,1)}_{3}$ and $\Lambda^{(0,1)}_{3}$ are spanned by the classes of the elements listed in  Table~\ref{table:root_vectors_dual_pairing_C2_full}, respectively. These classes form bases dual to the corresponding tangent space.

\end{lemma}

We now move towards the calculation of the maximal prolongation of the FODC in Equation \eqref{eq:full_calculus_C2}. First of all, we conveniently express the corresponding classifying ideal in terms of the union of some generating sets.
\begin{prop}\label{prop:generating_set_ideal_C2}
A generating set $\mathcal{G}$ for the right ideal $\underline{\mathcal{I}}_{3}$ associated to the quantum tangent space $\underline{T}^{(0,1)}_{3}$ is given by the union of the following sets:

\begin{align*}\mathcal{G}_{1}&:=\left\{u_{12},u_{13},u_{14},u_{23},u_{24},u_{32},u_{34}\right\},\\
\mathcal{G}'_{1}&:=\left\{u_{11}-1,\,u_{22}-1,\,u_{33}-1,\,u_{44}-1\right\},\\
\mathcal{G}''_{1}&:=\left\{u_{21}+u_{43},\,u_{31}-q^2u_{42}\right\},\\
\mathcal{G}_{2}&:=\left\{u_{\alpha}u_{k\ell}\;\middle|\;e_{\alpha}\triangleleft u_{k\ell}=0\right\},\\
\mathcal{G}'_{2}&:=\left\{u_{\alpha}u_{k\ell}-\theta_{k\ell\alpha}^{\beta}u_{\beta}\;\middle|\;e_{\alpha}\triangleleft u_{k\ell}=\theta_{k\ell\alpha}^{\beta}e_{\beta},\quad\theta_{k\ell\alpha}^{\beta}\neq0,\quad k\neq\ell\right\},\\
\mathcal{G}''_{2}&:=\left\{u_{\alpha}u_{kk}-\theta_{kk\alpha}^{\alpha}u_{\alpha}\;\middle|\;e_{\alpha}\triangleleft u_{kk}=\theta_{kk\alpha}^{\alpha}e_{\alpha}\right\},
\end{align*}
where $u_{\alpha},u_{\beta}\in\{u_{43},u_{42},u_{41}\}$, $e_{\alpha}=[u_{\alpha}]$,  $1\leq k,\ell\leq4$, and the coefficients $\theta_{k\ell\alpha}^{\beta}\in\mathbb C$ are determined by the right $\mathcal{O}_{q}(\mathrm{Sp}_{4})$-module structure of $\Lambda^{(0,1)}_{3}$.

\begin{proof}
The proof follows the lines of Proposition 3.7 in \cite{Re-Sombrg}. All the proposed generators have zero counit. The nonzero pairings of the tangent vectors with the matrix generators are given by Table \ref{table:root_vectors_dual_pairing_total_C2}. All other pairings with matrix generators not displayed in this table vanish. It follows immediately that

\[
\mathcal{G}_{1}\cup\mathcal{G}'_{1}
\cup\mathcal{G}''_{1}\subseteq\underline{\mathcal{I}}_{3}.
\]

\noindent For the quadratic generators, we use

\[
\langle E_{\gamma},u_{ij}u_{k\ell}\rangle
=
\left\langle (E_{\gamma})_{(1)},u_{ij}\right\rangle
\left\langle (E_{\gamma})_{(2)},u_{k\ell}\right\rangle.
\]

\noindent Substituting the coproducts of the root vectors and the pairings on matrix generators shows that every element of $\mathcal{G}_{2}\cup\mathcal{G}'_{2}\cup\mathcal{G}''_{2}$ is annihilated by every $E_{\gamma}\in \underline{T}^{(0,1)}_{3}$. For example,
\[
\left\langle
[[E_2,E_1]_{q^2},E_1]_{q^{-2}},u_{42}u_{43} \right\rangle=q-q^{-1},\quad  \left\langle [[E_2,E_1]_{q^2},E_1]_{q^{-2}},u_{41} \right\rangle=q^2+q^{-2}.
\]
Moreover, both $E_1$ and $[E_1,E_2]_{q^2}$ vanish on these two elements. Hence
\[
u_{42}u_{43}
-\frac{q-q^{-1}}{q^2+q^{-2}}u_{41}
\in\underline{\mathcal{I}}_{3}.
\]
The other quadratic generators are checked in the same way. Thus $\mathcal{G}\subseteq\mathcal{I}_{3}$.

Set
\[
A:=\mathcal{O}_q(\mathrm{Sp}_4), \qquad \underline{\mathcal{J}}_{3}:=\sum_{g\in\mathcal{G}}gA.
\]
Since $\underline{\mathcal{I}}_{3}$ is a right ideal containing $\mathcal{G}$, we have $\underline{\mathcal{J}}_{3}\subseteq\underline{\mathcal{I}}_{3}$. For the opposite inclusion, we first show that the classes of $1,u_{43},u_{42},u_{41}$ span $A/\underline{\mathcal{J}}_{3}$. The generators in $\mathcal{G}_{1}\cup\mathcal{G}'_{1} \cup\mathcal{G}''_{1}$ give
\[
u_{ii} \equiv 1\pmod{\underline{\mathcal{J}}_{3}}, \qquad u_{21} \equiv -u_{43} \pmod{\underline{\mathcal{J}}_{3}},\qquad u_{31}\equiv q^2 u_{42} \pmod{\underline{\mathcal{J}}_{3}},
\]
and
\[
u_{12},u_{13},u_{14},u_{23},u_{24},u_{32},u_{34} \equiv 0 \pmod{\underline{\mathcal{J}}_{3}}.
\]
Therefore every matrix generator $u_{ij}$ is congruent modulo $\underline{\mathcal{J}}_{3}$
to a linear combination of $1,u_{43},u_{42},u_{41}$. The quadratic generators cover every product
\[
u_{43}u_{k\ell},\qquad
u_{42}u_{k\ell},\qquad
u_{41}u_{k\ell},
\qquad 1\leq k,\ell\leq4.
\]
Each such product is consequently congruent (modulo $\underline{\mathcal{J}}_{3}$) to a linear combination of $u_{43},u_{42},u_{41}$, and since $\underline{\mathcal{J}}_{3}$ is a right ideal, this congruence remains valid after multiplying on the right by any matrix generator $u_{k\ell}$. The resulting terms are reduced using the linear and quadratic relations above. The same spanning statement therefore holds for the longer word by inductive reasoning. Since the matrix generators generate $A$, the classes of $1,u_{43},u_{42},u_{41}$ span $A/\underline{\mathcal{J}}_{3}$. Finally, the counit and the pairing table show that the classes of $1,u_{43},u_{42},u_{41}$ are linearly independent in $A/\underline{\mathcal{I}}_{3}$ (just consider a general linear combination of these classes, dually pair with the elements of the tangent space and apply the counit to show linear independence). Since these classes span $A/\underline{\mathcal{J}}_{3}$, the natural surjection  $A/\underline{\mathcal{J}}_{3}\twoheadrightarrow A/\underline{\mathcal{I}}_{3}$ maps a spanning family to a linearly independent family and is therefore an isomorphism. Its kernel is $\underline{\mathcal{I}}_{3}/\underline{\mathcal{J}}_{3}$, so $\underline{\mathcal{I}}_{3}=\underline{\mathcal{J}}_{3}$.
\end{proof}
\end{prop}

We are now ready to determine the relations for the maximal prolongation of $\Omega^{1}(\mathrm{Sp}_{4})$ at the level of its exterior algebra $\Lambda^{(0,\bullet)}_{3}$, following Proposition~\ref{proposition:quantum_exterior_algebra}. Using the presentation of the classifying ideal $\underline{\mathcal{I}}_{3}$ established above, we compute the map
\begin{equation}\label{eq:the_map_varpi_C2}
\begin{aligned}
\varpi^{\otimes_{2}}\colon \underline{\mathcal{I}}_{3} &\to \Lambda^{(0,1)}_{3}\otimes\Lambda^{(0,1)}_{3},\\
x&\mapsto  \bigl((\varpi\circ\pi_{\varepsilon})\otimes(\varpi\circ\pi_{\varepsilon})\bigr)\Delta(x) =[x_{(1)}^{+}]\otimes[x_{(2)}^{+}]
\end{aligned}
\end{equation}
on the generating families $\mathcal{G}_{1}$, $\mathcal{G}'_{1}$, $\mathcal{G}''_{1}$, $\mathcal{G}_{2}$, $\mathcal{G}'_{2}$, and $\mathcal{G}''_{2}$. To discuss the results, it is helpful to give a diagrammatic representation of the module structure of $\Lambda^{(0,1)}_{3}$. The latter can easily be deduced from the coproducts of the elements in $\underline{T}^{(0,1)}_{3}$, in the same spirit of Section \ref{section:contact_calculus_C2}. 
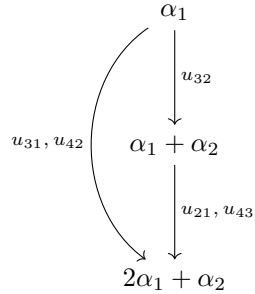
\begin{figure}[htbp]
\centering
\begin{tikzcd}[row sep=huge, column sep=large]
\alpha_{1}
    \arrow[d, "{u_{32}}"]
    \arrow[dd, bend right=55, swap, "{u_{31},\,u_{42}}"]
\\
\alpha_{1}+\alpha_{2}
    \arrow[d, "{u_{21},\,u_{43}}"]
\\
2\alpha_{1}+\alpha_{2}
\end{tikzcd}
\caption{Nonzero off-diagonal right actions of the matrix generators on $\Lambda^{(0,1)}_{3}$. An arrow labelled $u_{ij}$ indicates that its action sends the source vector to a nonzero scalar multiple of the target vector.}
\label{figure:module_structure_full_C2}
\end{figure}
 
In the next lemma we discuss which generating sets for the ideal $\underline{\mathcal{I}}_{3}$ give no non-trivial relations for the exterior algebra. 

\begin{lemma}
\label{lem:vanishing_relations_for_C2_3D}
We have 
\[
\varpi^{\otimes_{2}}(\mathcal{G}_{1})=\varpi^{\otimes_{2}}(\mathcal{G}'_{1})=\varpi^{\otimes_{2}}(\mathcal{G}''_{1})=\varpi^{\otimes_{2}}(\mathcal{G}''_{2})=0.
\]
\begin{proof}
	The defining formula for $\varpi^{\otimes_{2}}$ extends to every element of $\mathcal{O}_{q}(\mathrm{Sp}_{4})$, and in particular we have 
\[
\varpi^{\otimes_{2}}(u_{ij})= \sum_{a=1}^{4} [u_{ia}^{+}]\otimes[u_{aj}^{+}].
\]
The linear relations in the cotangent space give
\[
[u_{21}]=-[u_{43}], \qquad [u_{31}]=q^2[u_{42}],
\]
while $[u_{ij}^{+}]=0$ for all $(i,j)\notin \{(2,1),(4,3),(3,1),(4,2),(4,1)\}$. Consequently, since $\varpi^{\otimes_{2}}(1)=0$, this proves
\[
\varpi^{\otimes_{2}}(\mathcal{G}_{1}) =\varpi^{\otimes_{2}}(\mathcal{G}'_{1}) =0.
\]
Furthermore, we have 
\begin{align*}
\varpi^{\otimes_{2}}(u_{21}+u_{43}) & = 	\varpi^{\otimes_{2}}(u_{21})+\varpi^{\otimes_{2}}(u_{43}) =0, & \varpi^{\otimes_{2}}(u_{31}-q^{2}u_{42}) & = \varpi^{\otimes_{2}}(u_{31}) -q^{2}\varpi^{\otimes_{2}}(u_{42})= 0,
\end{align*}
and thus $\varpi^{\otimes_{2}}(\mathcal{G}''_{1})=0$.   Finally, as neither $\alpha_1$ nor $\alpha_1+\alpha_2$ is a sum of two weights of $\underline{T}^{(0,1)}_{3}$, for $1\leq i\leq 4$ we have 
\[
\varpi^{\otimes_{2}}(u_{43}u_{ii}) = \varpi^{\otimes_{2}}(u_{42}u_{ii}) = \varpi^{\otimes_{2}}(u_{43}) = \varpi^{\otimes_{2}}(u_{42}) =0,
\]
implying that every generator of $\mathcal{G}''_{2}$ with left factor $u_{43}$ or $u_{42}$ has zero image. For the remaining generators, the coproduct formulas and the pairing on matrix generators give
\[
q^{2}\varpi^{\otimes_{2}}(u_{41}u_{11}) = \varpi^{\otimes_{2}}(u_{41}u_{22})  = \varpi^{\otimes_{2}}(u_{41}u_{33})=q^{-2}\varpi^{\otimes_{2}}(u_{41}u_{44})=\varpi^{\otimes_{2}}(u_{41}).
\]
In each case, the image of the quadratic term is cancelled by that of its corresponding linear correction, hence $\varpi^{\otimes_{2}}(\mathcal{G}''_{2})=0$.
\end{proof} 
\end{lemma}

Next, we calculate the relations coming from the remaining generating sets for $\underline{\mathcal{I}}_{3}$. Indeed, the non-trivial relations for the maximal prolongation will be supplied from these generating sets. 
\begin{lemma}
\label{lem:relations_from_G2_C2_3D} 
The relations arising from $\mathcal{G}_{2}$ and $\mathcal{G}'_{2}$ are generated by
\[
\begin{aligned}
e_{1}\wedge e_{1} &= 0, & e_{2+1}\wedge e_{2+1} &= 0, & e_{2+1+1}\wedge e_{2+1+1} &= 0,\\
e_{2+1+1}\wedge e_{1}+e_{1}\wedge e_{2+1+1} &= 0, & e_{2+1+1}\wedge e_{2+1} +e_{2+1}\wedge e_{2+1+1} &= 0, & e_{1}\wedge e_{2+1} + q^{2}e_{2+1}\wedge e_{1} & = 0.
\end{aligned}
\]
\end{lemma}

\begin{proof}
The proof is a direct calculation involving the map $\varpi^{\otimes_{2}}$. We find the only non trivial contributions from the generating set $\mathcal{G}_{2}$ to be 
\begin{align*}
\varpi^{\otimes_{2}}(u_{43}u_{21}) & = -(q+q^{-1})e_{1}\otimes e_{1},\\ 
\varpi^{\otimes_{2}}(u_{43}u_{41}) & = -(q^{3}+q^{-3})(e_{2+1+1}\otimes e_{1} + e_{1}\otimes e_{2+1+1}),\\ 
\varpi^{\otimes_{2}}(u_{43}u_{43}) & = (q+q^{-1})e_{1}\otimes e_{1},\\ 
\varpi^{\otimes_{2}}(u_{42}u_{31}) & = q(q^{2}+1)e_{2+1}\otimes e_{2+1},\\ 
\varpi^{\otimes_{2}}(u_{42}u_{41}) & = -(q^{3}+q^{-3})(e_{2+1+1}\otimes e_{2+1} + e_{2+1}\otimes e_{2+1+1}),\\ 
\varpi^{\otimes_{2}}(u_{42}u_{42}) & = (q+q^{-1})e_{2+1}\otimes e_{2+1},\\
\varpi^{\otimes_{2}}(u_{41}u_{21}) & = (q^{2}+q^{-4})(e_{2+1+1}\otimes e_{1}+e_{1}\otimes e_{2+1+1}),\\ 
\varpi^{\otimes_{2}}(u_{41}u_{31}) & = -(q^{4}+q^{-2})(e_{2+1+1}\otimes e_{2+1}+e_{2+1}\otimes e_{2+1+1}) \\
\varpi^{\otimes_{2}}(u_{41}u_{41}) & = (q+q^{-1})(q^{2}+q^{-2})(q^{3}+q^{-3}) e_{2+1+1}\otimes e_{2+1+1},\\ 
\varpi^{\otimes_{2}}(u_{41}u_{42}) & = -(q^{4}+q^{-2})(e_{2+1+1}\otimes e_{2+1} + e_{2+1}\otimes e_{2+1+1}),\\
\varpi^{\otimes_{2}}(u_{41}u_{43})  & = -(q^{4}+q^{-2})(e_{2+1+1}\otimes e_{1}+e_{1}\otimes e_{2+1+1}).
\end{align*}
Proceeding similarly, the only non-trivial contributions from the generating set $\mathcal{G}'_{2}$ are
\begin{align*}
	\varpi^{\otimes_{2}}\left(u_{43}u_{31}
 +\frac{q^{-1}-q^{-3}}{q^2+q^{-2}}u_{41}\right) & = \frac{q(q^{2}+1)}{q^{2}+q^{-2}}(e_{1}\otimes e_{2+1} + q^{2}e_{2+1}\otimes e_{1})\\ 
 \varpi^{\otimes_{2}}\left(u_{43}u_{42} +\frac{q^3-q}{q^2+q^{-2}}u_{41}\right)& = \frac{q^{-3}(q^{2}+1)}{q^{2}+q^{-2}}(e_{1}\otimes e_{2+1} + q^{2}e_{2+1}\otimes e_{1}),\\ 
 \varpi^{\otimes_{2}}\left(u_{42}u_{21} +\frac{q^{-1}-q^{-3}}{q^2+q^{-2}}u_{41}\right) & = -\frac{q^{-3}(q^{2}+1)}{q^{2}+q^{-2}}(e_{1}\otimes e_{2+1} + q^{2}e_{2+1}\otimes e_{1}),\\ 
 \varpi^{\otimes_{2}}\left(u_{42}u_{43}  -\frac{q-q^{-1}}{q^2+q^{-2}}u_{41}\right)& = \frac{q^{-1}(q^{2}+1)}{q^{2}+q^{-2}}(e_{1}\otimes e_{2+1} + q^{2}e_{2+1}\otimes e_{1}).
\end{align*}
Passing to the exterior algebra gives the claim. 
\end{proof}

Finally, in the next theorem we provide the full set of relations for the exterior algebra corresponding to the maximal prolongation of the FODC in Equation \eqref{eq:full_calculus_C2}. In particular, we show that the latter shares the same graded dimension of the classical de Rham complex for the underlying holomorphic contact variety. The proof consists of showing that the set of relations obtained above is closed under the action of $\mathcal{O}_{q}(\mathrm{Sp}_{4})$ and that the corresponding set of monomials provides a Gr\"{o}bner basis. 
\begin{thm}
Let $\Lambda^{(0,\bullet)}_{3}$ be the antiholomorphic exterior algebra of the maximal prolongation of the calculus $\underline{\Omega}_{3}^{(0,1)}(\mathrm{Sp}_{4})$ over $\mathcal{O}_{q}(\mathrm{Sp}_{4})$. Then $\Lambda^{(0,\bullet)}_{3}$ is generated in degree one by $e_{1},e_{2+1},e_{2+1+1}$, subject to the same relations as in Lemma \ref{lem:relations_from_G2_C2_3D}. Moreover, the ordered monomials
\[
\begin{gathered}
1,\\
e_{1},\quad e_{2+1},\quad e_{2+1+1},\\
e_{1}\wedge e_{2+1},\quad
e_{1}\wedge e_{2+1+1},\quad
e_{2+1}\wedge e_{2+1+1},\\
e_{1}\wedge e_{2+1}\wedge e_{2+1+1}
\end{gathered}
\]
form a $\mathbb{C}$-basis of $\Lambda^{(0,\bullet)}_{3}$. In particular, the antiholomorphic exterior algebra has classicaldimensions:
\[
\dim_{\mathbb{C}}\Lambda^{(0,k)}_{3} =\binom{3}{k}, \quad 0\leq k\leq 3, \qquad \Lambda^{(0,k)}_{3}=0, \quad k>3.
\]

\begin{proof}
First, we prove that the set of relations obtained in Lemma \ref{lem:relations_from_G2_C2_3D} is closed under the right $\mathcal{O}_{q}(\mathrm{Sp}_{4})$-action. This basically follows as a direct consequence of how $\Lambda^{(0,1)}_{3}$ is structured as an $\mathcal{O}_{q}(\mathrm{Sp}_{4})$-module. We provide explicit calculations to show the claim. To begin, consider
\[
(e_{1}\otimes e_{1})\triangleleft_{\otimes}u_{ij}  = \sum_{a=1}^{4}[e_{1} u_{ia}]\otimes [e_{1} u_{aj}]. \\ 
\]
By Figure \ref{figure:module_structure_full_C2}, the only matrix generators $u_{ij}$ with $i\neq j$ acting nontrivially on $e_{1}$ are $u_{32},u_{31}$ and $u_{42}$.  We obtain
\begin{align*}
(e_{1}\otimes e_{1})\triangleleft_{\otimes}u_{32} &=(q^{-2}-q^{2})\bigl(e_{1}\otimes e_{2+1}+q^{2}e_{2+1}\otimes e_{1} \bigr),\\
(e_{1}\otimes e_{1})\triangleleft_{\otimes}u_{31} &=(q^{-2}-q^{-4})\bigl(e_{1}\otimes e_{2+1+1}+e_{2+1+1}\otimes e_{1}\bigr),\\
(e_{1}\otimes e_{1})\triangleleft_{\otimes}u_{42}&=(q^{4}-q^{2})\bigl(e_{1}\otimes e_{2+1+1}+e_{2+1+1}\otimes e_{1}\bigr).
\end{align*}
All other off-diagonal matrix generators annihilate $e_{1}\otimes e_{1}$, while each diagonal generator acts by a scalar. Hence the action of every matrix generator on $e_{1}\otimes e_{1}$ belongs to the span of the six quadratic relations.  We repeat the same procedure for the other relations, listing only the nonzero off-diagonal actions.
 
 \begin{align*}
(e_{2+1}\otimes e_{2+1})\triangleleft_{\otimes}u_{21} &=(q^{-2}-q^{-4})\bigl(e_{2+1+1}\otimes e_{2+1}+e_{2+1}\otimes e_{2+1+1}\bigr),\\
(e_{2+1}\otimes e_{2+1})\triangleleft_{\otimes}u_{43} &=(1-q^{2})\bigl(e_{2+1}\otimes e_{2+1+1}+e_{2+1+1}\otimes e_{2+1}\bigr),\\ 
(e_{2+1+1}\otimes  e_{1}+e_{1}\otimes  e_{2+1+1}) \triangleleft_{\otimes} u_{32}& = q(q^{-2}-q^{2})(e_{2+1+1}\otimes e_{2+1} + e_{2+1}\otimes e_{2+1+1}),\\ 
(e_{2+1+1}\otimes  e_{1}+e_{1}\otimes  e_{2+1+1}) \triangleleft_{\otimes} u_{31}&  =(q^{-1}-q^{-5})e_{2+1+1}\otimes e_{2+1+1},\\ 
(e_{2+1+1}\otimes  e_{1}+e_{1}\otimes  e_{2+1+1}) \triangleleft_{\otimes} u_{42}& = (q^{5}-q) e_{2+1+1}\otimes e_{2+1+1},\\
	(e_{2+1+1}\otimes e_{2+1} +e_{2+1}\otimes e_{2+1+1})\triangleleft_{\otimes} u_{21} & = (1-q^{-2})(q^{-1}+q^{-3})e_{2+1+1}\otimes e_{2+1+1}, \\ 
    (e_{2+1+1}\otimes e_{2+1} +e_{2+1}\otimes e_{2+1+1})\triangleleft_{\otimes} u_{43} & = -(1-q^{-2})(q^{3}+q)e_{2+1+1}\otimes e_{2+1+1},\\ 
(e_{1}\otimes e_{2+1} + q^{2}e_{2+1}\otimes e_{1})\triangleleft_{\otimes} u_{21} & = (1-q^{-2})(e_{1}\otimes e_{2+1+1} + e_{2+1+1}\otimes e_{1}),\\ 
(e_{1}\otimes e_{2+1} + q^{2}e_{2+1}\otimes e_{1})\triangleleft_{\otimes} u_{32} & = (q^{-2}-q^{2})(1+q^{4})e_{2+1}\otimes e_{2+1},\\ 
(e_{1}\otimes e_{2+1} + q^{2}e_{2+1}\otimes e_{1})\triangleleft_{\otimes} u_{43} & =(1-q^{2})(e_{1} \otimes e_{2+1+1} + e_{2+1+1}\otimes e_{1}),\\ 
(e_{1}\otimes e_{2+1} + q^{2}e_{2+1}\otimes e_{1})\triangleleft_{\otimes} u_{42} & = (q^{6}-q^{4})(e_{2+1+1}\otimes e_{2+1}+  e_{2+1}\otimes e_{2+1+1}),\\ 
(e_{1}\otimes e_{2+1}+q^{2}e_{2+1}\otimes e_{1}) \triangleleft_{\otimes}u_{31} &=(q^{-2}-q^{-4}) \bigl(e_{2+1+1}\otimes e_{2+1} +e_{2+1}\otimes e_{2+1+1} \bigr).
\end{align*}

All off-diagonal actions not listed above vanish, and each diagonal matrix generator acts on every displayed relation tensor by a scalar. Thus the linear span of the six relation tensors is a right $\mathcal{O}_{q}(\mathrm{Sp}_{4})$-submodule. The computations of $\varpi^{\otimes_{2}}(\mathcal{G})$ show that these images span precisely this subspace. Since the quadratic relation space of the maximal prolongation is the right $\mathcal{O}_{q}(\mathrm{Sp}_{4})$-submodule generated by these images, it coincides with their linear span. Consequently, the six displayed relations give a complete presentation of $\Lambda^{(0,\bullet)}_{3}$.  It remains to prove that the ordered monomials in the statement form a basis. 

For this, we apply the Diamond Lemma \cite[Theorem~1.2]{Ber78}, using the ordering

\[
e_{1} < e_{2+1} < e_{2+1+1} 
\]

\noindent on the generators of $\Lambda^{(0,1)}_{3}$, and orient the relations as 
\[
\begin{aligned}
e_{2+1}e_{1}&\longrightarrow -q^{-2}e_{1}e_{2+1}, & e_{2+1+1}e_{1}&\longrightarrow -e_{1}e_{2+1+1},\\
e_{2+1+1}e_{2+1}&\longrightarrow -e_{2+1}e_{2+1+1}, & e_{1}^{2}, e_{2+1}^{2}, e_{2+1+1}^{2} &\longrightarrow 0.
\end{aligned}
\]
These reductions strictly decrease the degree-lexicographic order induced by the chosen ordering of the generators. Since all leading words have length two, there are no inclusion ambiguities. The overlap ambiguities are precisely the length-three words whose generators occur in weakly decreasing order. For the three cubes, reducing either adjacent pair gives zero.  Next consider the six overlap ambiguities involving exactly two distinct generators. In each case, reducing the adjacent square gives zero immediately (the alternative reduction paths also give zero):

 \[
\begin{alignedat}{3}
e_{2+1}e_{1}e_{1}
&\;\longrightarrow\; -q^{-2}e_{1}e_{2+1}e_{1}
& &\longrightarrow\; q^{-4}e_{1}e_{1}e_{2+1}
& &\longrightarrow\; 0,\\
e_{2+1}e_{2+1}e_{1}
&\;\longrightarrow\; -q^{-2}e_{2+1}e_{1}e_{2+1}
&&\longrightarrow\; q^{-4}e_{1}e_{2+1}e_{2+1}
&&\longrightarrow\; 0,\\
e_{2+1+1}e_{1}e_{1}
&\;\longrightarrow\; -e_{1}e_{2+1+1}e_{1}
&&\longrightarrow\; e_{1}e_{1}e_{2+1+1}
&&\longrightarrow\; 0,\\
e_{2+1+1}e_{2+1+1}e_{1}
&\;\longrightarrow\; -e_{2+1+1}e_{1}e_{2+1+1}
&&\longrightarrow\; e_{1}e_{2+1+1}e_{2+1+1}
&&\longrightarrow\; 0,\\
e_{2+1+1}e_{2+1}e_{2+1}
&\;\longrightarrow\; -e_{2+1}e_{2+1+1}e_{2+1}
&&\longrightarrow\; e_{2+1}e_{2+1}e_{2+1+1}
&&\longrightarrow\; 0,\\
e_{2+1+1}e_{2+1+1}e_{2+1}
&\;\longrightarrow\; -e_{2+1+1}e_{2+1}e_{2+1+1}
&&\longrightarrow\; e_{2+1}e_{2+1+1}e_{2+1+1}
&&\longrightarrow\; 0.
\end{alignedat}
\]

Finally, we analyse the remaining combination, corresponding to the products of the three generators of $\Lambda^{(0,1)}_{3}$:
\begingroup
\small
\[
\begin{tikzcd}[
  column sep={24mm,between origins},
  row sep={12mm,between origins},
  cells={nodes={inner sep=1.5pt}},
  arrows={draw=black!100, line width=0.45pt}
]
&
e_{2+1+1}e_{2+1}e_1
  \arrow[ld] \arrow[rd]
&
\\
-e_{2+1}e_{2+1+1}e_1
  \arrow[d]
&
&
-q^{-2}e_{2+1+1}e_1e_{2+1}
  \arrow[d]
\\
e_{2+1}e_1e_{2+1+1}
  \arrow[rd]
&
&
q^{-2}e_1e_{2+1+1}e_{2+1}
  \arrow[ld]
\\
&
-q^{-2}e_1e_{2+1}e_{2+1+1}
&
\end{tikzcd}
\]
\endgroup
Thus all overlap ambiguities are resolvable. By the Diamond Lemma, the irreducible words form a vector space basis of $\Lambda^{(0,\bullet)}_{3}$. These are precisely the strictly increasing monomials listed in the statement. Counting them by degree gives 
\[
\dim_{\mathbb{C}}\Lambda^{(0,k)}_{3}=\binom{3}{k}, \quad 0\leq k\leq 3,  \qquad \Lambda^{(0,k)}_{3}=0, \quad k> 3.
\]

\end{proof}
\end{thm}

Finally, let $V^{(0,\bullet)}_{3}$ be the exterior algebra of the maximal prolongation of the projected calculus $\Omega^{(0,1)}_{3}(\mathcal{O}_{q}(\mathrm{Sp}_{4}/\mathrm{L}_{\alpha_{2}}))$ over the corresponding quantised flag manifold. In the case at hand, the natural map

\[
\begin{aligned}
j\colon V^{(0,1)}&\to \Lambda^{(0,1)},\\
[b]_{V}&\mapsto [b]_{\Lambda},
\end{aligned}
\]

\noindent where $[-]_{V}$ and $[-]_{\Lambda}$ denote cosets in $\mathcal{O}_{q}(\mathrm{Sp}_{4}/\mathrm{L}_{\alpha_{2}})^{+}/\mathcal{I}_{3}$ and $\mathcal{O}_{q}(\mathrm{Sp}_{4})^{+}/\underline{\mathcal{I}}_{3}$, respectively, is a $\mathbb{C}$-linear isomorphism. Under this identification, the quadratic relations defining $V^{(0,\bullet)}_{3}$ coincide with those defining $\Lambda^{(0,\bullet)}_{3}$.

\begin{thm}\label{thm:maximal_prolongation_for_the_contact_flag_c2_3d}
The above identification induces an isomorphism of graded algebras
\[
j^{(\bullet)}\colon V^{(0,\bullet)}_{3}\to \Lambda^{(0,\bullet)}_{3}.
\]
Equivalently, the quadratic relation spaces of the two maximal prolongations coincide under $j\otimes j$, that is, $(j\otimes j)\bigl(\mathcal{I}^{(2)}_{3}\bigr)=\underline{\mathcal{I}}_{3}^{(2)}$, and in particular,
\[
\dim_{\mathbb{C}}V^{(0,k)}_{3} = \binom{3}{k}, \quad 0\leq k\leq 3,\qquad V^{(0,k)}_{3}=0,\quad k>3
\]

\end{thm}

\begin{proof}
Consider the  map $j\colon V_{3}^{(0,1)}\to\Lambda_{3}^{(0,1)}$ sending  $[b]_{V} \mapsto[b]_{\Lambda}$. For the case at hand, this map is a $\mathbb{C}$-linear isomorphism, and in particular, its image is a right $\mathcal{O}_{q}(\mathrm{Sp}_{4})$-submodule. By injectivity, we have 
\[
\mathcal{I}_{3}=\underline{\mathcal{I}}_{3}\cap \mathcal{O}_{q}(\mathrm{Sp}_{4}/\mathrm{L}_{\alpha_{2}}).
\] 
Generally speaking, Proposition A.1 of \cite{ale_re_jun} describes the quadratic relations of the maximal prolongation of a left $A$-covariant FODC on a quantum homogeneous space $B\subseteq A$, in the presence of a framing calculus on $A$. In particular, for the case at hand, it provides an identification of the quadratic relation space $\mathcal{I}^{(2)}_{3}$ of the calculus on the contact quantised flag manifold $\mathcal{O}_{q}(\mathrm{Sp}_{4}/\mathrm{L}_{\alpha_{2}})$ through
\[
\mathcal{I}^{(2)}_{3} := \operatorname{span}_{\mathbb{C}} \left\{ j^{-1}\left([x_{(1)}^{+}]_{\Lambda}\right)\otimes j^{-1}\left([x_{(2)}^{+}]_{\Lambda}\right)\mid x \in\underline{\mathcal{I}}\cap \mathcal{O}_{q}(\mathrm{Sp}_{4}/\mathrm{L}_{\alpha_{2}}) \right\}.
\]
The resulting exterior algebra is
\[
V^{(0,\bullet)}_{3} = \mathcal{T}\bigl(V^{(0,1)}_{3}\bigr)/ \bigl\langle\mathcal{I}^{(2)}_{3}\bigr\rangle,
\]
Since $\mathcal{I}_{3} = \underline{\mathcal{I}}_{3}\cap \mathcal{O}_{q}(\mathrm{Sp}_{4}/\mathrm{L}_{\alpha_{2}})$, the above description immediately gives
\[
(j\otimes j)\bigl(\mathcal{I}^{(2)}_{3}\bigr)
\subseteq
\underline{\mathcal{I}}^{(2)}_{3}.
\]
Moreover, since $j\otimes j$ is a $\mathbb C$-linear isomorphism, its restriction to $\mathcal I^{(2)}_3$ is injective and it induces a surjective map $j^{(2)}$ on the corresponding quotients. Thus, we have the following commuting diagram

\[
\begin{tikzcd}[column sep=large, row sep=huge]
0 \arrow[r]
&
\mathcal I^{(2)}_{3}
    \arrow[r, hook]
    \arrow[d, "j\otimes j|_{\mathcal{I}^{(2)}_{3}}", hook]
&
\bigl(V^{(0,1)}_{3}\bigr)^{\otimes 2}
    \arrow[r, "\pi_V", two heads]
    \arrow[d, "j\otimes j", "\simeq"']
&
V^{(0,2)}_{3}
    \arrow[r]
    \arrow[d, "j^{(2)}", two heads]
&
0
\\
0 \arrow[r]
&
\underline{\mathcal I}^{(2)}_{3}
    \arrow[r, hook]
&
\bigl(\Lambda^{(0,1)}_{3}\bigr)^{\otimes 2}
    \arrow[r, "\pi_\Lambda", two heads]
&
\Lambda^{(0,2)}_{3}
    \arrow[r]
&
0.
\end{tikzcd}
\]
 By the snake lemma, we have a long exact sequence 

\[
\begin{tikzcd}[
  column sep=large,
  row sep=22mm
]
0 \arrow[r]
&
\ker(j\otimes j|_{\mathcal I^{(2)}_{3}})
  \arrow[r]
&
\ker(j\otimes j)
  \arrow[r]
&
\ker(j^{(2)})
  \arrow[dlll,
    to path={
      let
        \p1=(\tikztostart.east),
        \p2=(\tikztotarget.west),
        \n1={(\y1+\y2)/2}
      in
        (\p1)
        .. controls
          ({\x1+12mm},\y1)
          and
          ({\x1+12mm},\n1)
        .. (\x1,\n1)
        -- (\x2,\n1)
        .. controls
          ({\x2-12mm},\n1)
          and
          ({\x2-12mm},\y2)
        .. (\p2)
        \tikztonodes
    }]
\\
\operatorname{coker}(
  j\otimes j|_{\mathcal I^{(2)}_{3}}
)
  \arrow[r]
&
\operatorname{coker}(j\otimes j)
  \arrow[r]
&
\operatorname{coker}(j^{(2)})
  \arrow[r]
&
0
\end{tikzcd}
\]

reducing to  
\[
0 \to \ker (j^{(2)}) \stackrel{\cong}{\to}  \operatorname{coker}(j\otimes j|_{\mathcal I^{(2)}_{3}}) \to 0 
\]
Consequently, to show injectivity of $j^{(2)}$ holds if and only if $j\otimes j|_{\mathcal I^{(2)}_{3}}$ is surjective. For the case at hand, surjectivity amounts to showing that every defining quadratic relation for the exterior algebra $\Lambda^{(0,\bullet)}_{3}$ lies in $(j\otimes j)(\mathcal I^{(2)}_{3})$. By the description of $\mathcal I^{(2)}_{3}$ above, this amounts to expressing each of the six relation tensors as a linear combination of elements of the form $\varpi^{\otimes_{2}}(z)$, with $z\in\underline{\mathcal{I}_{3}}\cap \mathcal{O}_{q}(\mathrm{Sp}_{4}/\mathrm{L}_{\alpha_{2}})$. 

For the last computations, we infer the possible elements of $\mathcal{O}_{q}(\mathrm{Sp}_{4}/\mathrm{L}_{\alpha_{2}})$ giving the already established quadratic relations for $\Lambda^{(0,2)}_{3}$ reasoning on the underlying root system. For example, we have 

\begin{align*}
\varpi^{\otimes_{2}}(z_{43}z_{43}) & = \langle E_{1}E_{1},z_{43}z_{43}\rangle e_{1}\otimes e_{1}\\ 
& = \big(\langle E_{1}^{2},z_{43}\rangle\langle K_{1}^{2},z_{43}\rangle + \langle E_{1},z_{43}\rangle\langle K_{1}E_{1},z_{43}\rangle  \\ 
& \quad + \langle E_{1},z_{43}\rangle\langle E_{1}K_{1},z_{43}\rangle   + \langle 1,z_{43}\rangle\langle E_{1}^{2},z_{43}\rangle \big)e_{1}\otimes e_{1} \\ 
& =  q\left( \langle K_{1}E_{1},z_{43}\rangle + \langle E_{1}K_{1},z_{43}\rangle \right) e_{1}\otimes e_{1}\\ 
& =  q\left( \langle K_{1},u_{44}S(u_{33})\rangle \langle E_{1},z_{43}\rangle + \langle E_{1},u_{44}S(u_{43})\rangle \langle K_{1},z_{44}\rangle\right) e_{1}\otimes e_{1}\\ 
& = q^{3}(q + q^{-1}) e_{1}\otimes e_{1},
\end{align*}

and similarly

\begin{align*}
\varpi^{\otimes_{2}}(z_{42}z_{42}) & = 	q^{11}(q+q^{-1})e_{2+1}\otimes e_{2+1},\\ 
\varpi^{\otimes_{2}}(z_{41}z_{41}) & =  q^{10}(q+q^{-1})(q^2+q^{-2})(q^3+q^{-3})e_{2+1+1}\otimes e_{2+1+1}.
\end{align*}

To obtain the relations featuring mixed cosets, we start by evaluating

\begin{align*}
\varpi^{\otimes_{2}}(z_{43}z_{41}) & = -(q^{9}+q^{3})(e_{2+1+1}\otimes e_{1}+ e_{1}\otimes e_{2+1+1}),\\ 
\varpi^{\otimes_{2}}(z_{42}z_{41}) & = -q^{10}(q^{3}+q^{-3})(e_{2+1+1}\otimes e_{2+1}+ e_{2+1}\otimes e_{2+1+1}).
\end{align*}
To obtain the remaining relation, we determine the scalar $\lambda$ for which $z_{43}z_{42}+\lambda z_{41}\in\mathcal{I}_{3}$. This element has zero counit and, by weight considerations, pairs trivially with $E_1$ and $[E_1,E_2]_{q^2}$ for every $\lambda$. Moreover 

\begin{align*}
\langle [[E_{2},E_{1}]_{q^{2}},E_{1}]_{q^{-2}},z_{43}z_{42}+ \lambda z_{41}\rangle & = \langle [[E_{2},E_{1}]_{q^{2}},E_{1}]_{q^{-2}},z_{43}z_{42}\rangle  -\lambda q^{4}(q^{2}+q^{-2}) \\ 
& = (1-q^{-2})\langle E_{1},z_{43}\rangle \langle [E_{2},E_{1}]_{q^{2}}K_{1}, z_{42}\rangle  -\lambda q^{4}(q^{2}+q^{-2}) \\ 
& = q(1-q^{-2}) \langle [E_{2},E_{1}]_{q^{2}},u_{44}S(u_{42})\rangle \langle K_{1}, z_{44}\rangle  -\lambda q^{4}(q^{2}+q^{-2}) \\
& = -q^{4}(1-q^{-2}) -\lambda q^{4}(q^{2}+q^{-2}),
\end{align*}
and thus we calculate  

\[
\varpi^{\otimes_{2}}\left(z_{43}z_{42}-\frac{1-q^{-2}}{q^{2}+q^{-2}}z_{41}\right)=\frac{q^{8}(q^{2}+1)}{q^{4}+1}\left(e_{1}\otimes e_{2+1}+q^{2}e_{2+1}\otimes e_{1}\right).
\]

We deduce surjectivity of $j\otimes j|_{\mathcal{I}^{(2)}_{3}}$, and thus injectivity of $j^{(2)}$. Thus, $j^{(2)}$ is an isomorphism and the claim follows.  
\end{proof}

\subsection{A Dolbeault complex for the full flag manifold of $\mathcal{O}_{q}(\mathrm{Sp}_{4})$}
\label{section:contact_adapted_calculus_C2}
In light of the results obtained in the previous section, it is natural to ask whether the three-dimensional tangent space can be completed to a four-dimensional one for $\mathcal{O}_{q}(\mathrm{Sp}_{4})$ defining a covariant FODC whose maximal prolongation again has the classical dimension.  Such a construction would provide a complete and geometrically meaningful picture for the $\mathrm{C}_{2}$ case, encompassing both partial flag manifolds as well as the full flag manifold. With this aim, we construct the required tangent space and the associated FODC, and analyse their main properties and structural features. That adding the remaining simple root to the $3$-dimensional tangent space introduced in the previous section again yields a quantum tangent space is straightforward to verify; this is the content of the following lemma.

\begin{lemma}\label{lemma:the_4d_tangent_space_c2}
	 The following is a 4-dimensional tangent space for $\mathcal{O}_{q}(\mathrm{Sp}_{4})$: 
	 \[
	 \underline{T}^{(0,1)}_{4}:=\underline{T}^{(0,1)}_{2}\oplus \underline{T}^{(0,1)}_{\theta}\oplus \mathrm{Span}_{\mathbb{C}}\{E_{2}\}.
	 \]
 \begin{proof}
 The proof is automatic, as the condition $\Delta(\underline{T}^{(0,1)}_{4})\subseteq (T^{(0,1)}_{4}\oplus \mathbb{C}1) \otimes \mathrm{U}_{q}(\mathfrak{sp}_{4})$ follows in light of Lemma \ref{lemma:the_3d_tangent_space_c2}. 
 \end{proof}
\end{lemma}

Table \ref{table:root_vectors_dual_pairing_total_C2} provides the dual pairing between the elements of the $4$-dimensional tangent space and elements of the quantised coordinate algebra. We now provide a generating set for the classifying ideal associated to the tangent space above. The proof that the union of the following sets generates the full ideal after closure under the right action is entirely analogous to that of the previous section, and we therefore omit it.

\begin{prop}\label{prop:generating_set_ideal_4d_C2}
A generating set $\mathcal{G}$ for the right ideal $\underline{\mathcal{I}}_{4}$ associated to the quantum tangent space $\underline{T}_{4}^{(0,1)}$ is given by the union of the sets of linear elements
\begin{align*}
\mathcal G_1 &:=\{u_{12},u_{13},u_{14},u_{23},u_{24},u_{34}\},\\ 
\mathcal G'_1&:=\{u_{ii}-1\mid 1\leq i\leq 4 \},\\
\mathcal G''_1&:=\{u_{21}+u_{43},\ u_{31}-q^2u_{42}\},
\end{align*}

and the sets of quadratic elements

\begin{align*}
\mathcal G_2 &:= \left\{u_{43}u_{ij}\ \middle|\ 1\leq i,j\leq 4,\ i\neq j,\ (i,j)\notin \{(3,1),(3,2),(4,2)\}\right\}\\
&\quad\cup \left\{u_{32}u_{ij}\ \middle|\ 1\leq i,j\leq 4,\ i\neq j\right\}\\
&\quad\cup\left\{u_{42}u_{ij}\ \middle|\ 1\leq i,j\leq 4,\ i\neq j,\ (i,j)\notin\{(2,1),(4,3)\}\right\}\\
&\quad\cup \left\{ u_{41}u_{ij}\ \middle|\ 1\leq i,j\leq 4,\ i\neq j \right\},\\
\mathcal G'_2 &:=\left\{u_\alpha u_{k\ell}-\theta_{k\ell\alpha}^{\beta}u_\beta\ \middle|\ e_\alpha\triangleleft u_{k\ell}=\theta_{k\ell\alpha}^{\beta}e_\beta,\ \theta_{k\ell\alpha}^{\beta}\neq 0,\ k\neq\ell\right\},\\
\mathcal G''_2 &:=\left\{u_\alpha u_{kk}-\theta_{kk\alpha}^{\alpha}u_\alpha\ \middle|\ e_\alpha\triangleleft u_{kk}=\theta_{kk\alpha}^{\alpha}e_\alpha \right\},
\end{align*}
where
$u_\alpha,u_\beta\in\{u_{43},u_{32},u_{42},u_{41}\}$, $e_\alpha=[u_\alpha]$, and $1\leq k,\ell\leq4$, and the coefficients $\theta_{k\ell\alpha}^{\beta}\in\mathbb C$ are determined by the right $\mathcal{O}_{q}(\mathrm{Sp}_{4})$-module structure of the cotangent space.
\end{prop}

We next determine the relations in the maximal prolongation of $\underline{\Omega}_{4}^{1}(\mathrm{Sp}_{4})$ at the level of its antiholomorphic exterior algebra. As in Proposition~\ref{proposition:quantum_exterior_algebra}, these are obtained from the map
\begin{equation}\label{eq:the_map_varpi_C2}
\begin{aligned}
\varpi^{\otimes_{2}}\colon \underline{\mathcal{I}}_{4} &\to \Lambda^{(0,1)}_{4}\otimes\Lambda^{(0,1)}_{4},\\
x&\mapsto  \bigl((\varpi\circ\pi_{\varepsilon})\otimes(\varpi\circ\pi_{\varepsilon})\bigr)\Delta(x) =[x_{(1)}^{+}]\otimes[x_{(2)}^{+}],
\end{aligned}
\end{equation}
which we evaluate on the generating families $\mathcal{G}_{1}$, $\mathcal{G}'_{1}$, $\mathcal{G}''_{1}$, $\mathcal{G}_{2}$, $\mathcal{G}'_{2}$, and $\mathcal{G}''_{2}$.

The right $\mathcal{O}_{q}(\mathrm{Sp}_{4})$-module structure of $\Lambda^{(0,1)}_{4}$ is read off directly from the coproducts of the elements spanning $\underline{T}_{4}^{(0,1)}$. The nonzero off-diagonal actions are summarised in Figure~\ref{figure:module_structure_full_4d_C2}.
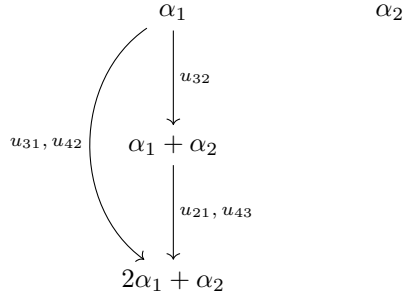
\begin{figure}[htbp]
\centering
\begin{tikzcd}[row sep=huge, column sep=huge]
\alpha_{1}
    \arrow[d, "{u_{32}}"]
    \arrow[dd, bend right=55, swap, "{u_{31},\,u_{42}}"]
& \alpha_{2}
\\
\alpha_{1}+\alpha_{2}
    \arrow[d, "{u_{21},\,u_{43}}"]
&
\\
2\alpha_{1}+\alpha_{2}
&
\end{tikzcd}
\caption{Nonzero off-diagonal right actions of the matrix generators on $\Lambda^{(0,1)}_{4}$. An arrow labelled $u_{ij}$ indicates that its action sends the source vector to a nonzero scalar multiple of the target vector.}
\label{figure:module_structure_full_4d_C2}
\end{figure}

Before turning to the nontrivial quadratic relations, it is useful to isolate those generators of $\underline{\mathcal{I}}_{4}$ whose images under $\varpi^{\otimes_{2}}$ vanish. Three of the four families that produced no relations in the $3$-dimensional calculus continue to do so after adjoining the additional tangent direction. The verification is formally identical to the proof of Lemma \ref{lem:vanishing_relations_for_C2_3D}, so we do not repeat it here.

\begin{lemma}
	We have 
\[
\varpi^{\otimes_{2}}(\mathcal{G}_{1})=\varpi^{\otimes_{2}}(\mathcal{G}'_{1})=\varpi^{\otimes_{2}}(\mathcal{G}''_{2})=0.
\]
\end{lemma}

Finally, we turn to the relations arising from the remaining generating families, thereby obtaining a complete generating set of relations for the maximal prolongation of the $4$-dimensional calculus.

\begin{lemma}
\label{lem:relations_from_G2_C2_4d} 
The relations arising from $\mathcal{G}''_{1}$, $\mathcal{G}_{2}$ and $\mathcal{G}'_{2}$ are generated by $e_i\wedge e_i = 0$ and 
\begin{align*}
e_1\wedge e_2+q^{-2}e_2\wedge e_1&=0, & e_1\wedge e_{2+1}+q^2e_{2+1}\wedge e_1&=0,\\
e_1\wedge e_{2+1+1}+e_{2+1+1}\wedge e_1&=0, & e_2\wedge e_{2+1}+q^{-2}e_{2+1}\wedge e_2&=0,\\
e_2\wedge e_{2+1+1}+e_{2+1+1}\wedge e_2&=0, & e_{2+1}\wedge e_{2+1+1} +e_{2+1+1}\wedge e_{2+1}&=0.
\end{align*}
\end{lemma}

\begin{proof}
The proof is a direct calculation involving the map $\varpi^{\otimes_{2}}$. For the linear set $\mathcal{G}''_{1}$ we have only a single non-trivial relation, namely
\[
\varpi^{\otimes_{2}}(u_{31}-q^{2}u_{42}) = e_{2}\otimes e_{1} + q^{2} e_{1}\otimes e_{2}.
\]

Turning to the first quadratic family, namely $\mathcal{G}_{2}$, we observe that all elements containing a factor $u_{ij}$ on the right with $i<j$ again give vanishing contributions. We are therefore left with the following nontrivial quadratic tensors:
 \begin{align*}
	\varpi^{\otimes_{2}}(u_{43}u_{21}) & = -(q+q^{-1})e_{1}\otimes e_{1},\\ 
	\varpi^{\otimes_{2}}(u_{43}u_{41}) & = -(q^{3}+q^{-3})(e_{2+1+1}\otimes e_{1}+e_{1}\otimes e_{2+1+1}),\\
	\varpi^{\otimes_{2}}(u_{43}u_{43}) & = (q+q^{-1})e_{1}\otimes e_{1},\\
	\varpi^{\otimes_{2}}(u_{32}u_{21}) & = e_{1}\otimes e_{2} + q^{-2}e_{2}\otimes e_{1},\\ 
	\varpi^{\otimes_{2}}(u_{32}u_{31}) & = -q^{2}(e_{2+1}\otimes e_{2} + q^{2}e_{2}\otimes e_{2+1}),\\  
	\varpi^{\otimes_{2}}(u_{32}u_{32}) & = (q^{2}+q^{-2})e_{2}\otimes e_{2},\\ 
	\varpi^{\otimes_{2}}(u_{32}u_{41}) & = (q^{2}+q^{-2})(e_{2}\otimes e_{2+1+1} + e_{2+1+1}\otimes e_{2}),\\
	\varpi^{\otimes_{2}}(u_{32}u_{42}) & = -(e_{2}\otimes e_{2+1} + q^{-2}e_{2+1}\otimes e_{2}),\\ 
	\varpi^{\otimes_{2}}(u_{32}u_{43}) & = e_{2}\otimes e_{1} + q^{2}e_{1}\otimes e_{2},\\ 
	\varpi^{\otimes_{2}}(u_{42}u_{32}) & = -q^{-1}(q^{-2}e_{2+1}\otimes e_{2} + e_{2}\otimes e_{2+1}),\\ 
	\varpi^{\otimes_{2}}(u_{42}u_{31}) & = q(q^{2}+1)e_{2+1}\otimes e_{2+1} - (q^{-1}-q^{-3})(e_{2+1+1}\otimes e_{2} + e_{2}\otimes e_{2+1+1}),\\ 
	\varpi^{\otimes_{2}}(u_{42}u_{42}) & = (q+q^{-1})e_{2+1}\otimes e_{2+1},\\ 
	\varpi^{\otimes_{2}}(u_{42}u_{41}) & = -(q^{3}+q^{-3})(e_{2+1+1}\otimes e_{2+1} + e_{2+1}\otimes e_{2+1+1}),\\ 
	\varpi^{\otimes_{2}}(u_{41}u_{21}) & = (q^{2}+q^{-4})(e_{2+1+1}\otimes e_{1}+ e_{1}\otimes e_{2+1+1}),\\ 
	\varpi^{\otimes_{2}}(u_{41}u_{31}) & =-(q^{4}+q^{-2})(e_{2+1+1}\otimes e_{2+1} + e_{2+1}\otimes e_{2+1+1}),\\ 
    \varpi^{\otimes_{2}}(u_{41}u_{32}) & =(q^{2}+q^{-2})(e_{2+1+1}\otimes e_{2} + e_{2}\otimes e_{2+1+1}),\\ 
    \varpi^{\otimes_{2}}(u_{41}u_{41}) & = (q+q^{-1})(q^2+q^{-2})(q^3+q^{-3})e_{2+1+1}\otimes e_{2+1+1},\\ 
    \varpi^{\otimes_{2}}(u_{41}u_{42}) & =-(q^{4}+q^{-2})(e_{2+1+1}\otimes e_{2+1} + e_{2+1}\otimes e_{2+1+1}),\\ 
    \varpi^{\otimes_{2}}(u_{41}u_{43}) & = -(q^{4}+q^{-2})(e_{2+1+1}\otimes e_{1} + e_{1}\otimes e_{2+1+1}). 
\end{align*}

Finally, moving to the set $\mathcal{G}'_{2}$, we find
 \begin{align*}
	\varpi^{\otimes_{2}}\left(u_{43}u_{31}+\frac{q^{-1}-q^{-3}}{q^2+q^{-2}}u_{41}\right) & = \frac{q(q^{2}+1)}{q^{2}+q^{-2}}(e_{1}\otimes e_{2+1} + q^{2}e_{2+1}\otimes e_{1}),\\
	\varpi^{\otimes_{2}}\left(u_{43}u_{32}+(q^{3}-q^{-1})u_{42}\right) & = -q(e_{2}\otimes e_{1} + q^{2}e_{1}\otimes e_{2}),\\ 
	\varpi^{\otimes_{2}}\left(u_{43}u_{42}+\frac{q^{3}-q}{q^2+q^{-2}}u_{41}\right) & = \frac{q^{-1}+q^{-3}}{q^{2}+q^{-2}}(e_{1}\otimes e_{2+1} + q^{2}e_{2+1}\otimes e_{1}) ,\\
	\varpi^{\otimes_{2}}\left(u_{42}u_{21}+\frac{q^{-1}-q^{-3}}{q^2+q^{-2}}u_{41}\right) & = -\frac{q^{-1}+q^{-3}}{q^{2}+q^{-2}}(e_{1}\otimes e_{2+1}+q^{2}e_{2+1}\otimes e_{1}),\\ 
	\varpi^{\otimes_{2}}\left(u_{42}u_{43}-\frac{q-q^{-1}}{q^2+q^{-2}}u_{41}\right) & = \frac{q+q^{-1}}{q^{2}+q^{-2}}(e_{1}\otimes e_{2+1}+q^{2}e_{2+1}\otimes e_{1}).
\end{align*}

Passing to the exterior algebra gives the claim. 
\end{proof}

The relations above are precisely those of a quantum symmetric algebra, which strongly suggests that the resulting exterior algebra should retain the classical dimension. To establish this, however, it remains to close the set of quadratic relations under the right action of $\mathcal{O}_{q}(\mathrm{Sp}_{4})$ and to verify that no overlap ambiguities occur among degree-$3$ monomials in the chosen basis of the cotangent space. 
 
\begin{thm}
Let $\Lambda^{(0,\bullet)}_{4}$ be the exterior algebra of the maximal prolongation of the calculus $\underline{\Omega}_{4}^{(0,1)}(\mathrm{Sp}_{4})$ over $\mathcal{O}_{q}(\mathrm{Sp}_{4})$. Then $\Lambda^{(0,\bullet)}_{4}$ is generated in degree one by $e_{1},e_{2},e_{2+1},e_{2+1+1}$, subject to the same relations as in Lemma \ref{lem:relations_from_G2_C2_4d}.  Moreover, the ordered monomials
\[
\begin{gathered}
1,\\
e_{1},\quad e_{2},\quad e_{2+1},\quad e_{2+1+1},\\
e_{1}\wedge e_{2},\quad
e_{1}\wedge e_{2+1},\quad
e_{1}\wedge e_{2+1+1},\\
e_{2}\wedge e_{2+1},\quad
e_{2}\wedge e_{2+1+1},\quad
e_{2+1}\wedge e_{2+1+1},\\
e_{1}\wedge e_{2}\wedge e_{2+1},\quad
e_{1}\wedge e_{2}\wedge e_{2+1+1},\\
e_{1}\wedge e_{2+1}\wedge e_{2+1+1},\quad
e_{2}\wedge e_{2+1}\wedge e_{2+1+1},\\
e_{1}\wedge e_{2}\wedge e_{2+1}\wedge e_{2+1+1}
\end{gathered}
\]

form a $\mathbb{C}$-basis of $\Lambda^{(0,\bullet)}_{4}$. In particular, the exterior algebra has classical dimension:

\[
\dim_{\mathbb{C}}\Lambda^{(0,k)}_{4} =\binom{4}{k}, \quad 0\leq k\leq 4,\qquad  \Lambda^{(0,k)}_{4}=0, \quad k>4.
\]
\begin{proof}
First, we prove that the span of the ten quadratic relation tensors is stable under the right $\mathcal{O}_{q}(\mathrm{Sp}_{4})$-action. We verify this explicitly using the $\mathcal{O}_{q}(\mathrm{Sp}_{4})$-module structure of $\Lambda^{(0,1)}_{4}$ and the diagonal action on its tensor square. To begin, consider
\[
(e_1\otimes e_1)\triangleleft_{\otimes}u_{ij}=\sum_{a=1}^{4}(e_1\triangleleft u_{ia})\otimes(e_1\triangleleft u_{aj}).
\]
By Figure \ref{figure:module_structure_full_4d_C2}, the only off-diagonal matrix generators acting nontrivially on $e_1$ are $u_{32},u_{31}$ and $u_{42}$. Computing the tensor-product action gives
\begin{align*}
(e_1\otimes e_1)\triangleleft_{\otimes}u_{32} &=(q^{-2}-q^2)\bigl(e_1\otimes e_{2+1}+q^2e_{2+1}\otimes e_1\bigr),\\
(e_1\otimes e_1)\triangleleft_{\otimes}u_{31}&=(q^{-2}-q^{-4})\bigl(e_1\otimes e_{2+1+1} +e_{2+1+1}\otimes e_1\bigr),\\
(e_1\otimes e_1)\triangleleft_{\otimes}u_{42}&=(q^4-q^2)\bigl(e_1\otimes e_{2+1+1}+e_{2+1+1}\otimes e_1\bigr).
\end{align*}
All other off-diagonal matrix generators annihilate $e_1\otimes e_1$, while each diagonal generator acts by a scalar. Hence the action of every matrix generator on $e_1\otimes e_1$ belongs to the span of the ten quadratic relation tensors.

We repeat the same procedure for the other relations, listing only the nonzero off-diagonal actions. For $e_{2+1}\otimes e_{2+1}$, we obtain
\begin{align*}
(e_{2+1}\otimes e_{2+1})\triangleleft_{\otimes}u_{21} &=(q^{-2}-q^{-4})\bigl(e_{2+1}\otimes e_{2+1+1}+e_{2+1+1}\otimes e_{2+1}\bigr),\\ 
(e_{2+1}\otimes e_{2+1})\triangleleft_{\otimes}u_{43}&=(1-q^2)\bigl(e_{2+1}\otimes e_{2+1+1} +e_{2+1+1}\otimes e_{2+1}\bigr).
\end{align*}

Every off-diagonal matrix generator annihilates
$e_{2+1+1}\otimes e_{2+1+1}$. For the mixed relation tensors involving $e_1,e_{2+1}$ and $e_{2+1+1}$, we obtain
\begin{align*}
(e_1\otimes e_{2+1}+q^2e_{2+1}\otimes e_1)\triangleleft_{\otimes}u_{21}&=(1-q^{-2})\bigl(e_1\otimes e_{2+1+1} +e_{2+1+1}\otimes e_1\bigr),\\
(e_1\otimes e_{2+1}+q^2e_{2+1}\otimes e_1)\triangleleft_{\otimes}u_{31}&=(q^{-2}-q^{-4})\bigl(e_{2+1}\otimes e_{2+1+1}+e_{2+1+1}\otimes e_{2+1}\bigr),\\
(e_1\otimes e_{2+1}+q^2e_{2+1}\otimes e_1)\triangleleft_{\otimes}u_{32}&=(q^{-2}-q^2)(1+q^4)e_{2+1}\otimes e_{2+1},\\
(e_1\otimes e_{2+1}+q^2e_{2+1}\otimes e_1)\triangleleft_{\otimes}u_{42}&=(q^6-q^4)\bigl(e_{2+1}\otimes e_{2+1+1}+e_{2+1+1}\otimes e_{2+1}\bigr),\\
(e_1\otimes e_{2+1}+q^2e_{2+1}\otimes e_1)\triangleleft_{\otimes}u_{43}&=(1-q^2)\bigl(e_1\otimes e_{2+1+1}+e_{2+1+1}\otimes e_1\bigr)\\ 
(e_1\otimes e_{2+1+1}+e_{2+1+1}\otimes e_1)\triangleleft_{\otimes}u_{31}&=(q^{-1}-q^{-5})e_{2+1+1}\otimes e_{2+1+1},\\
(e_1\otimes e_{2+1+1}+e_{2+1+1}\otimes e_1)\triangleleft_{\otimes}u_{32}&=(q^{-1}-q^3)\bigl(e_{2+1}\otimes e_{2+1+1}+e_{2+1+1}\otimes e_{2+1}\bigr),\\
(e_1\otimes e_{2+1+1}+e_{2+1+1}\otimes e_1)\triangleleft_{\otimes}u_{42}&=(q^5-q)e_{2+1+1}\otimes e_{2+1+1},\\
(e_{2+1}\otimes e_{2+1+1}+e_{2+1+1}\otimes e_{2+1})\triangleleft_{\otimes}u_{21}&=(q^{-1}-q^{-5})
e_{2+1+1}\otimes e_{2+1+1},\\
(e_{2+1}\otimes e_{2+1+1}+e_{2+1+1}\otimes e_{2+1})\triangleleft_{\otimes}u_{43} &=(q^{-1}-q^3)
e_{2+1+1}\otimes e_{2+1+1}.
\end{align*}

We now consider the four additional relation tensors involving $e_2$. Since both $e_2$ and $e_{2+1+1}$ are annihilated by every off-diagonal matrix generator, we have

\[
(e_2\otimes e_2)\triangleleft_{\otimes}u_{ij}=0, \qquad (e_2\otimes e_{2+1+1}+e_{2+1+1}\otimes e_2)\triangleleft_{\otimes}u_{ij}=0,\qquad i\neq j.
\]

For the remaining two relation tensors, the nonzero
off-diagonal actions are
\begin{align*}
(e_1\otimes e_2+q^{-2}e_2\otimes e_1)\triangleleft_{\otimes}u_{31}&=(q^{-1}-q^{-3})
\bigl(e_2\otimes e_{2+1+1}+e_{2+1+1}\otimes e_2\bigr),\\
(e_1\otimes e_2+q^{-2}e_2\otimes e_1)\triangleleft_{\otimes}u_{32}&=(q^{-1}-q^3)
\bigl(e_2\otimes e_{2+1}+q^{-2}e_{2+1}\otimes e_2\bigr),\\
(e_1\otimes e_2+q^{-2}e_2\otimes e_1)\triangleleft_{\otimes}u_{42}&=(q-q^{-1})\bigl(e_2\otimes e_{2+1+1}+e_{2+1+1}\otimes e_2\bigr),\\
(e_2\otimes e_{2+1}+q^{-2}e_{2+1}\otimes e_2)\triangleleft_{\otimes}u_{21}&=(q^{-3}-q^{-5})\bigl(e_2\otimes e_{2+1+1}+e_{2+1+1}\otimes e_2\bigr),\\
(e_2\otimes e_{2+1}+q^{-2}e_{2+1}\otimes e_2)\triangleleft_{\otimes}u_{43}&=(q^{-1}-q)\bigl(e_2\otimes e_{2+1+1} +e_{2+1+1}\otimes e_2\bigr).
\end{align*}

All off-diagonal actions not listed above vanish. Moreover, each diagonal matrix generator sends every relation tensor to a scalar multiple of itself. Consequently, the span of the ten quadratic relation tensors is stable under all matrix generators and hence is a right $\mathcal{O}_{q}(\mathrm{Sp}_{4})$-submodule.

The computations of $\varpi^{\otimes_{2}}(\mathcal{G})$ show that these images span precisely this subspace. Since the quadratic relation space of the maximal prolongation is the right $\mathcal{O}_{q}(\mathrm{Sp}_{4})$-submodule generated by these images, it coincides with their linear span. Consequently, the ten displayed relations give a complete presentation of $\Lambda^{(0,\bullet)}_{4}$.

It remains to prove that the ordered monomials in the statement form a basis. For this, we apply the Diamond
Lemma \cite[Theorem~1.2]{Ber78}.

We use the ordering
\[
e_{1} < e_{2} < e_{2+1} < e_{2+1+1} 
\]
on the generators of $\Lambda^{(0,1)}_{4}$, and orient the relations as 
\[
\begin{aligned}
e_{2}e_{1}&\longrightarrow -q^{2}e_{1}e_{2}, & e_{2+1}e_{1}&\longrightarrow -q^{-2}e_{1}e_{2+1}, & e_{2+1+1}e_{1}&\longrightarrow -e_{1}e_{2+1+1},\\
e_{2+1}e_{2} & \longrightarrow -q^{2}e_{2}e_{2+1}, & e_{2+1+1}e_{2} & \longmapsto -e_{2}e_{2+1+1}, & e_{2+1+1}e_{2+1}&\longrightarrow -e_{2+1}e_{2+1+1}.
\end{aligned}
\]
and 
\[ 
e_{1}^{2}, e_{2}^{2}, e_{2+1}^{2}, e_{2+1+1}^{2}\longrightarrow 0,
\]
These reductions strictly decrease the degree-lexicographic order induced by the chosen ordering of the generators. Since all leading words have length two, there are no inclusion ambiguities. The weakly decreasing length-three words feature four cubes, twelve words with two distinct generators, and four words with three distinct generators, giving twenty overlaps in total. For the four cubes, it is immediate to see that reducing adjacent pairs gives zero. Next, consider the twelve overlap ambiguities involving exactly two distinct generators. In each case, reducing the adjacent square gives zero immediately, and the alternative reduction paths give zero as well. For example:
\[
e_{2+1}e_{2+1}e_{1} \longrightarrow -q^{-2}e_{2+1}e_{1}e_{2+1}\longrightarrow q^{-4}e_{1}e_{2+1}e_{2+1}\longrightarrow 0,
\]
and the other calculations are performed similarly. Finally, we analyse the remaining combination, corresponding to the products of three different generators of $\Lambda^{(0,1)}_{4}$. We find that the four overlap ambiguities involving distinct generators are indeed resolved. We give a pictorial description of the result:
\[
\begin{tikzcd}[
    column sep=2em,
    row sep=3.5em
]
e_{2+1+1}e_{2+1}e_1
    \arrow[d, bend right=25, swap, "\mathrm{L}"]
    \arrow[d, bend left=25, "\mathrm{R}"]
&
e_{2+1}e_2e_1
    \arrow[d, bend right=25, swap, "\mathrm{L}"]
    \arrow[d, bend left=25, "\mathrm{R}"]
&
e_{2+1+1}e_{2+1}e_2
    \arrow[d, bend right=25, swap, "\mathrm{L}"]
    \arrow[d, bend left=25, "\mathrm{R}"]
&
e_{2+1+1}e_2e_1
    \arrow[d, bend right=25, swap, "\mathrm{L}"]
    \arrow[d, bend left=25, "\mathrm{R}"]
\\
-q^{-2}e_1e_{2+1}e_{2+1+1}
&
-q^2e_1e_2e_{2+1}
&
-q^2e_2e_{2+1}e_{2+1+1}
&
-q^2e_1e_2e_{2+1+1}
\end{tikzcd}
\]
where the arrows labelled $\mathrm{L}$ and $\mathrm{R}$ denote reduction sequences beginning with the leftmost and rightmost adjacent pair, respectively.

Eventually, all overlap ambiguities are resolvable. By the Diamond Lemma, the irreducible words form a vector space basis of $\Lambda^{(0,\bullet)}_{4}$. These are precisely the strictly increasing monomials listed in the statement. Counting them by degree gives 
\[
\dim_{\mathbb{C}}\Lambda^{(0,k)}_{4}=\binom{4}{k}
\]
for $0\leq k\leq 4$, and there are no irreducible words of higher degree.
\end{proof}
\end{thm}

Thus the framing calculus on $\mathcal O_q(\mathrm{Sp}_4)$ has the graded exterior dimensions of the antiholomorphic cotangent space of the classical full flag manifold of type $C_2$. 

Finally, let $V^{(0,\bullet)}_{4}$ be the exterior algebra of the maximal prolongation of the projected calculus $\Omega^{(0,1)}_{4}$ over the corresponding quantised full flag manifold. We show that, under the same identification provided in Theorem \ref{thm:maximal_prolongation_for_the_contact_flag_c2_3d}, the quadratic relations defining $V^{(0,\bullet)}_{4}$ coincide with those defining $\Lambda^{(0,\bullet)}_{4}$.

\begin{thm}
We have an isomorphism of graded algebras
\[
j^{(\bullet)}\colon V^{(0,\bullet)}_{4}\to \Lambda^{(0,\bullet)}_{4}.
\]
Equivalently, the quadratic relation spaces of the two maximal prolongations coincide under $j\otimes j$, that is, $(j\otimes j)\bigl(\mathcal{I}^{(2)}_{4}\bigr)=\underline{\mathcal{I}}^{(2)}_{4}$, and in particular,
\[
\dim_{\mathbb{C}}V^{(0,k)}_{4} = \binom{4}{k}, \quad 0\leq k\leq 4,\qquad V^{(0,k)}_{4}=0,\quad k>4
\]
\begin{proof}
	The proof follows closely that of Theorem \ref{thm:maximal_prolongation_for_the_contact_flag_c2_3d}, with the main difference arising from the module structure of the cotangent space. In particular, the generators $z_{ij}$ employed there are no longer sufficient for the full flag manifold, since one must additionally account for the cotangent direction dual to the root vector $E_{2}$. A suitable representative for this additional direction is given by $u_{32}S(u_{22})$, which is readily verified to belong to the coordinate algebra of the full quantised flag manifold. Moreover, the remaining $z_{ij}$-generators considered in Theorem \ref{thm:maximal_prolongation_for_the_contact_flag_c2_3d} are already invariant under the torus action. Their corresponding cosets, together with that of $u_{32}S(u_{22})$, therefore provide the required four-dimensional quantum cotangent space of the full flag manifold.

		It remains to show that the quadratic relations defining the maximal prolongation of the upstairs calculus are recovered from the corresponding relations at the level of the full flag manifold. More precisely, we must identify suitable elements of the classifying ideal $\mathcal{I}_{4}$ and evaluate their images under $\varpi^{\otimes 2}$, verifying that the resulting tensors generate the quadratic relation space $\underline{\mathcal{I}}^{(2)}_{4}$ of the upstairs calculus. As in the proof of Theorem \ref{thm:maximal_prolongation_for_the_contact_flag_c2_3d}, the required elements can be selected systematically by exploiting the root-space decomposition and the associated weight considerations. In particular, the image of $\varpi^{\otimes_{2}}$ on the elements already established in Theorem \ref{thm:maximal_prolongation_for_the_contact_flag_c2_3d} coincides. We now carry out the remaining calculations. Again, motivated by root-system arguments, we consider the elements 
		
		\[
		u_{32}S(u_{22})u_{32}S(u_{22}), \qquad u_{32}S(u_{22})\,z_{43},\qquad u_{32}S(u_{22})\,z_{42}, \qquad u_{32}S(u_{22})\,z_{41}, 
		\]
		
		which, as may readily be checked, are annihilated by elements of $\underline{T}^{(0,1)}_{4}$. Finally, we find 
		
		\begin{align*}
			\varpi^{\otimes_{2}}(u_{32}S(u_{22})u_{32}S(u_{22})) & = q^{6}(q^{2}+q^{-2})e_{2}\otimes e_{2},\\ 
			\varpi^{\otimes_{2}}(u_{32}S(u_{22})\,z_{43}) & = q^{3}(e_{1}\otimes e_{2}+ q^{-2}e_{2}\otimes e_{1}),\\ 
			\varpi^{\otimes_{2}}(u_{32}S(u_{22})\,z_{42}) & = q^{9}(e_{2}\otimes e_{2+1}+ q^{-2}e_{2+1}\otimes e_{2}),\\ 
			\varpi^{\otimes_{2}}(u_{32}S(u_{22})\,z_{41}) & = -q^{6}(e_{2}\otimes e_{2+1+1}+e_{2+1+1}\otimes e_{2}),
		\end{align*}
		completing the proof.
	\end{proof}
\end{thm}

\begin{rmk}
The FODC $(\underline{\Omega}^{(0,1)}_{4}(\mathrm{Sp}_{4}),\overline{\partial})$ constructed above restricts to the anti-holomorphic Heckenberger–Kolb calculus on the cominuscule quantised flag manifold of type $\mathrm{C}_{2}$. However, the identification at first order does not identify the standard tensor presentations used to describe the higher-degree forms. In fact, in the presentation obtained by the framing calculus $(\underline{\Omega}^{(0,\bullet)}_{4}(\mathrm{Sp}_{4}),\overline{\partial})$, the quadratic relations are those of a quantum symmetric algebra, while the presentation of Heckenberger–Kolb features the mixed term
\[
e_{2+1}\wedge e_{2+1} =-(1-q^{2})e_{2+1+1}\wedge e_{2}.
\]

To compare these presentations, we must account for the different monoidal identifications. Generally speaking, consider a quantum homogeneous space $B\subseteq A$, a left $A$-covariant FODC $\Omega^{1}(B)=A\square_{\pi_{B}}V^{1}\in {}^{A}_{B}\mathrm{Mod}_{0}$ and a framing calculus $\Omega^{1}(A)=A\otimes \Lambda^{1}$. For such $\Omega^{1}(B)$, the equality $\Omega B^+=B^+\Omega$ allows the usual monoidal identification following in light of Theorem \ref{thm:takeuchi_equivalence_0}, while on the other hand the framing uses the monoidal identification supplied by Theorem \ref{th:fundamental_theorem_hopf}:
\[
\mu_{\mathrm{HK}} \colon [\omega\otimes_B\eta]\mapsto [\omega]_V\otimes[\eta]_V, \qquad \mu_{\mathrm{fr}}\colon [\omega\otimes_A\eta]\mapsto [\omega\eta_{(-1)}]_\Lambda\otimes[\eta_{(0)}]_\Lambda,
\]
Letting $\iota\colon \Omega^{1}(B)\hookrightarrow \Omega^{1}(A)$ be the map induced by the inclusion of the restricted calculus into its framing and by passage from $\otimes_B$ to $\otimes_A$, given $\Phi_{-}$ as in Theorem \ref{thm:takeuchi_equivalence}, and given that the image of such FODCi under $\Phi$ precisely consists of the corresponding cotangent spaces, we have

\[
\begin{tikzcd}[column sep=5.5em,row sep=5.5em]
  \begin{gathered}
    \Phi_B(\Omega^{1}(B)\otimes_B\Omega^{1}(B))\\[-1pt]
    {\scriptstyle [\omega\otimes_B\eta]}
  \end{gathered}
  \arrow[r,"\mu_{\mathrm{0}}"]
  \arrow[dr,"\mu_{\mathrm{fr}}\circ \iota"']
  &
  \begin{gathered}
    \Phi_B(\Omega^{1}(B))\otimes
    \Phi_B(\Omega^{1}(B))
    =V^{1}\otimes V^{1}\\[-1pt]
    {\scriptstyle [\omega]_V\otimes[\eta]_V}
  \end{gathered}
  \arrow[d,shift left=1.2ex,"\chi"]
  \\
  &
  \begin{gathered}
    \Lambda^{1}\otimes\Lambda^{1}\\[-1pt]
    {\scriptstyle
    [\omega\eta_{(-1)}]_\Lambda
    \otimes[\eta_{(0)}]_\Lambda}
  \end{gathered}
  \arrow[u,shift left=1.2ex,"\chi^{-1}"]
\end{tikzcd}
\]

Where the resulting comparison maps are
\[
\chi([\omega]_{V}\otimes [\eta]_{V})=[\omega\eta_{(-1)}]_\Lambda\otimes[\eta_{(0)}]_\Lambda, \qquad \chi^{-1}([\omega]_{\Lambda}\otimes [\eta]_{\Lambda})=[\omega S(\eta_{(-1)})]_V \otimes[\eta_{(0)}]_V
\]
 
   We now specialise the construction to the case at hand. To give a meaningful comparison, we need a representative of $e_{2+1}$ belonging to the quantised flag manifold $\mathcal{O}_{q}(\mathrm{Sp}_{4}/\mathrm{L}_{\alpha_{1}})$, where now the quantised Levi subalgebra features the $\alpha_{1}$-direction alongside the full Cartan part. Considering as plausible generators for the flag the quantum minors $p_{ij}:=u_{i1}u_{j2}-q\,u_{i2}u_{j1}$, we find  
   
\[
K_{1}^{\pm1}\triangleright p_{ij}= \varepsilon(K_{1}^{\pm1})p_{ij}, \qquad E_{1}\triangleright p_{ij}  = 0 = F_{1}\triangleright p_{ij},  \qquad K_{2}^{\pm1}\triangleright p_{ij} = q^{\mp 2}\varepsilon(K_{2}^{\pm1})p_{ij},
\]
showing invariance only under $\mathrm{U}_{q}(\ell_{1}^{s})$. To obtain full invariance, we consider the elements $\tilde{p}_{ij}:=p_{ij}S(p_{12})$, whose coproducts read
\[
\begin{aligned}
\Delta\bigl(p_{ij}S(p_{12})\bigr) &= \bigl(u_{ia}u_{jb}\otimes p_{ab}\bigr) \Delta\bigl(S(p_{12})\bigr)\\
&= \bigl(u_{ia}u_{jb}\otimes p_{ab}\bigr) (S\otimes S) \bigl(p_{cd}\otimes u_{1c}u_{2d}\bigr)\\
&= u_{ia}u_{jb}S(p_{cd}) \otimes p_{ab}S(u_{2d})S(u_{1c}).
\end{aligned}
\]
We now find 
\[
\begin{aligned}
K_{k}^{\pm1}\triangleright \bigl(p_{ij}S(p_{12})\bigr) &= u_{ia}u_{jb}S(p_{cd}) \left\langle K_{k}^{\pm1}, p_{ab}S(u_{2d})S(u_{1c})\right\rangle\\
&= u_{ia}u_{jb}S(p_{12}) \left\langle K_{k}^{\pm1},\bigl(u_{a1}u_{b2}-q\,u_{a2}u_{b1}\bigr) S(u_{22})S(u_{11})\right\rangle\\
&=u_{i1}u_{j2}S(p_{12})\left\langle K_{k}^{\pm1},u_{11}u_{22}S(u_{22})S(u_{11})\right\rangle\\
&\quad-q\,u_{i2}u_{j1}S(p_{12})\left\langle K_{k}^{\pm1},u_{22}u_{11}S(u_{22})S(u_{11})\right\rangle\\
&=\varepsilon(K_{k}^{\pm1})p_{ij}S(p_{12}).
\end{aligned}
\]
Similarly,
\[
\begin{aligned}
E_{1}\triangleright
\bigl(p_{ij}S(p_{12})\bigr)
&=
u_{ia}u_{jb}S(p_{cd})
\left\langle E_{1},
\bigl(u_{a1}u_{b2}-q\,u_{a2}u_{b1}\bigr)
S(u_{2d})S(u_{1c})\right\rangle\\
&=
u_{ia}u_{jb}S(p_{cd})
\left\langle E_{1},
u_{a1}u_{b2}S(u_{2d})S(u_{1c})\right\rangle\\
&\quad-
q\,u_{ia}u_{jb}S(p_{cd})
\left\langle E_{1},
u_{a2}u_{b1}S(u_{2d})S(u_{1c})\right\rangle\\
&=
u_{i2}u_{j2}S(p_{12})
\left\langle E_{1},
u_{21}u_{22}S(u_{22})S(u_{11})\right\rangle\\
&\quad-
q\,u_{i2}u_{j2}S(p_{12})
\left\langle E_{1},
u_{22}u_{21}S(u_{22})S(u_{11})\right\rangle\\
&=
\varepsilon(E_{1})p_{ij}S(p_{12}).
\end{aligned}
\]
The $F_1$-invariance is verified analogously.  Let $V^{(0,1)}_{\mathrm{HK}}$ denote the Heckenberger–Kolb cotangent space and let $\Lambda^{(0,1)}_{\mathrm{HK}}$ be the cotangent space of the corresponding framing. We identify the corresponding generators by dually pairing the $\mathcal{O}_{q}(\mathrm{Sp}_{4}/\mathrm{L}_{1})^{\circ}$-restriction of the root vectors in $\underline{T}^{(0,1)}_{\mathrm{HK}}$ with the elements $\tilde{p}_{ij}$. We find
\[
e_2 = \left[q^{-2}\tilde{p}_{13})\right]_{V^{(0,1)}_{\mathrm{HK}}},\qquad e_{2+1}=\left[-q^{-2}\tilde{p}_{14}\right]_{V^{(0,1)}_{\mathrm{HK}}}.
\]

As for the framed cotangent space the quadratic tensors are easier, we rather apply the inverse of the comparison map $\chi$ to that element giving the relation of interest. Set $b:=-q^{-2}\widetilde p_{14}$. We find 

\[
\begin{aligned}
\chi^{-1}(e_{2+1}\otimes e_{2+1}) &= [bS(b_{(1)})]_{V^{(0,1)}_{\mathrm{HK}}} \otimes[b_{(2)}^{+}]_{V^{(0,1)}_{\mathrm{HK}}}\\
&=q^{-4}\left[p_{14}S(p_{12})S(u_{1a}u_{4c}S(p_{rs}))\right]_{V^{(0,1)}_{\mathrm{HK}}}\otimes\left[\left(p_{ac}S(u_{2s})S(u_{1r})\right)^{+}\right]_{V^{(0,1)}_{\mathrm{HK}}}\\
&=q^{-4}\left[p_{14}S(p_{12})S(u_{11}u_{44}S(p_{12}))\right]_{V^{(0,1)}_{\mathrm{HK}}} \otimes\left[\left(p_{14}S(u_{22})S(u_{11})\right)^{+}\right]_{V^{(0,1)}_{\mathrm{HK}}}\\
&\quad+q^{-4}\left[p_{14}S(p_{12})S(u_{11}u_{43}S(p_{12}))\right]_{V^{(0,1)}_{\mathrm{HK}}}\otimes\left[\left(p_{13}S(u_{22})S(u_{11})\right)^{+}\right]_{V^{(0,1)}_{\mathrm{HK}}}\\
& = q^{-2}(e_{2+1}\otimes e_{2+1} - q^{2}(1-q^{-2})e_{2+1+1}\otimes e_{2})\\ 
& = q^{-2}(e_{2+1}\otimes e_{2+1}+ (1-q^{2})e_{2+1+1}\otimes e_{2}). 
\end{aligned}
\]
where only the index choices $(a,c,r,s)=(1,4,1,2),(1,3,1,2)$ give non-vanishing contributions.  

The calculation performed above shows that, up to a normalisation prefactor, we get back the same relation featured in the Heckenberger--Kolb prolongation.
\end{rmk}

\section{Differential calculi on $\mathcal{O}_{q}(\mathrm{SO}_{8})$ and its contact quantised flag manifold}
\label{section:lusztig_for_so(8)}
Building on the results of the previous section, we now turn to the special orthogonal Lie algebra $\mathfrak{so}_{8}$. As the first non-trivial case in the $\mathrm{D}_{n}$ series, the quantum group $\mathcal{O}_{q}(\mathrm{SO}_{8})$ already exhibits considerable complexity, while on the other hand remaining feasible for explicit computations with the aid of its triality symmetry. 

In Section~\ref{subsection:basics_for_so8}, we recall the Lie algebra $\mathfrak{so}_{2\ell}$, its quantised enveloping algebra, and the associated quantum group. In Section~\ref{section:calculus_D4_flag}, we construct an eight-dimensional FODC associated with the contact directions on the corresponding contact quantised flag manifold, supply a framing calculus at the level of the quantum group, and compute the corresponding maximal prolongation. Finally, in Section~\ref{section:tangent_so8}, we extend this calculus to a nine-dimensional one by adjoining the Reeb direction and show that the resulting FODC cannot be framed by any Lusztig FODC on $\mathcal{O}_{q}(\mathrm{SO}_{8})$. For ease of reference, the dual pairings used throughout this section are collected in a table displayed in Appendix~\ref{appendix}.
  
	\subsection{The Lie algebra $\mathfrak{so}_{2\ell}$, its quantised envelope and quantum group}\label{subsection:basics_for_so8}
	Consider the special orthogonal Lie algebra $\mathfrak{g}:=\mathfrak{so}_{2\ell}$. As in Section \ref{section:basics_C2}, we briefly recall the construction of the corresponding root system and provide an explicit basis. We then write down the Cartan matrix and Dynkin diagram and present the associated Weyl group. Finally, we discuss the structure of the quantised enveloping algebra $\mathrm{U}_{q}(\mathfrak{so}_{2\ell})$ and quantised coordinate algebra $\mathcal{O}_{q}(\mathrm{SO}_{8})$. Again, the material presented in this subsection is standard and can be found in many textbooks on Lie algebras and representation theory, see \emph{e.g.} \cite{Hump78} and \cite{Kli-Sch}.

Let $\mathfrak{h}\subset \mathfrak{g}$ be the set of diagonal matrices. We write the diagonal entries of an element in $H$ labelling from $0$ to $2\ell$, and thus the diagonal entries of an element $h\in \mathfrak{h}$ may be written as $a_{1},\dots, a_{\ell},-a_{1},\dots,-a_{\ell}$. Equivalently, $h=\sum_{i=1}^{\ell}a_{i}(e_{i,i}-e_{i+\ell,i+\ell})$. A basis for $\mathfrak{g}$ consisting of simultaneous eigenvectors for the adjoint $\mathfrak{h}$-action $\mathrm{Ad}_{\mathfrak{h}}\colon \mathfrak{h} \otimes \mathfrak{g}\to \mathfrak{g}$ on $\mathfrak{g}$ is given by the elements

\begin{equation}
	\begin{aligned}
		m_{ij}& =e_{i,j}-e_{j+\ell,i+\ell}, \quad \text{for} \ 1\leq i\neq j \leq \ell,  \\ 
		p_{ij}& =e_{i,\ell +j} - e_{j,\ell +i},\quad  \text{for} \ 1\leq i < j \leq \ell,  \\
		q_{ji} =p^{t}_{ij}& =e_{\ell +j,i} - e_{\ell +i,j}, \quad \text{for} \ 1\leq i < j \leq \ell, 
	\end{aligned}
\end{equation}
together with the Cartan elements $h_i=e_{ii}-e_{i+\ell,i+\ell}$, where  $1\leq i\leq\ell$, with eigenvalues
\begin{equation} 
\begin{aligned}
   \mr{Ad}_{\mathfrak{h}}(m_{ij})& =(a_{i}-a_{j})m_{ij}, \qquad \mr{Ad}_{\mathfrak{h}}(p_{ij})&=(a_{i}+a_{j})p_{ij}, \qquad  \mr{Ad}_{\mathfrak{h}}(q_{ji})&=-(a_{i}+a_{j})q_{ji}. 
\end{aligned}
\end{equation}

\noindent For $i=1,\dots,\ell$ consider maps $\epsilon_{i}\colon \mathfrak{h}\to \mathbb{C}$, sending $h \mapsto a_{i}$, i.e. sending each element $h\in \mathfrak{h}$ to the corresponding $i-$th diagonal entry $a_{i}$. We summarise the roots corresponding to the common eigenvectors for the action of the Cartan subalgebra on $\mathfrak{g}$.
\begin{equation*}
\begin{tabular}{c|c|c|c}
 
  root & $\epsilon_{i}-\epsilon_{j}$ & $\epsilon_{i}+\epsilon_{j}$ & $-(\epsilon_{i}+\epsilon_{j})$ \\
  \hline
  eigenvector & $m_{ij}\ (i\neq j)$ & $p_{ij} \ (i<j)$ & $q_{ji} \ (i<j)$ \\

\end{tabular}.
\end{equation*}
The \emph{standard basis} for the corresponding root system is given by 
\begin{equation*}
	\mathcal{B}=\big\{\alpha_{i}:=\epsilon_{i}-\epsilon_{i+1},  \ 1\leq i < \ell\big\}\cup \big\{\beta_{n}=\epsilon_{n-1}+\epsilon_{n}\big\}, 
\end{equation*}

\noindent and the associated set of fundamental weights is 
\begin{align*}
\varpi_{\alpha_i} & =\epsilon_{1}+\cdots + \epsilon_{i},  \ \text{for} \ i\leq n-2, \\
\varpi_{\alpha_{n-1}} & = \frac{1}{2}\left(\epsilon_{1}+\cdots + \epsilon_{n-1}-\epsilon_{n}\right), \\
\varpi_{\beta_n} & = \frac{1}{2}\left(\epsilon_{1}+\cdots + \epsilon_{n-1}+\epsilon_{n}\right). 
\end{align*}

\noindent The Cartan matrix reads 

\[
\mathrm{A}_{\mathrm{D}_{n}}=
\begin{pmatrix}
2 & -1 & 0 & 0 & \cdots & 0 & 0 & 0 \\
-1 & 2 & -1 & 0 & \cdots & 0 & 0 & 0 \\
0 & -1 & 2 & -1 & \cdots & 0 & 0 & 0 \\
0 & 0 & -1 & 2 & \cdots & 0 & 0 & 0 \\
\vdots & \vdots & \vdots & \vdots & \ddots & -1 & 0 & 0 \\
0 & 0 & 0 & 0 & -1 & 2 & -1 & -1 \\
0 & 0 & 0 & 0 & 0 & -1 & 2 & 0 \\
0 & 0 & 0 & 0 & 0 & -1 & 0 & 2
\end{pmatrix},
\]
and the corresponding Dynkin diagram is as in Figure \ref{figure:dynk_D}. Finally, the Weyl group reads
\begin{equation*}
 \begin{aligned}
 	 	\mathcal{W}=\langle s_{\alpha_{1}},s_{\alpha_{2}},\dots,s_{\alpha_{n-1}},s_{\beta} & \mid
(s_i)^2=1,\;
(s_i s_j)^{m_{ij}}=1\rangle ,  
\end{aligned}
\end{equation*}
where  $m_{ii}=1$,  $m_{ij}=3$ if $i,j$ are connected and $m_{ij}=2$ otherwise. 

\begin{figure}[!htbp]
\centering	
\begin{tikzpicture}[scale=1]
\tikzstyle{smallcircle} = [circle, draw=black, fill=lightgray, minimum size=2.5mm, inner sep=0pt]

\node[smallcircle, label=above:$\alpha_1$] (n1) at (0,0) {};
\node[smallcircle, label=above:$\alpha_2$] (n2) at (2,0) {};
\node (n) at (4,0) {$\cdots$};
\node[smallcircle, label=above:$\alpha_{n-2}$] (n3) at (6,0) {};

\node[smallcircle, label=right:$\alpha_{n}$] (n4) at (8,1.2) {};
\node[smallcircle, label=right:$\alpha_{n-1}$]   (n5) at (8,-1.2) {};

\draw (n1) -- (n2);
\draw (n2) -- (n);
\draw (n) -- (n3);
\draw (n3) -- (n4);
\draw (n3) -- (n5);

\end{tikzpicture}.
\caption{The Dynking diagram for the root system $\mathrm{D}_{n}$.}
\label{figure:dynk_D}
\end{figure}
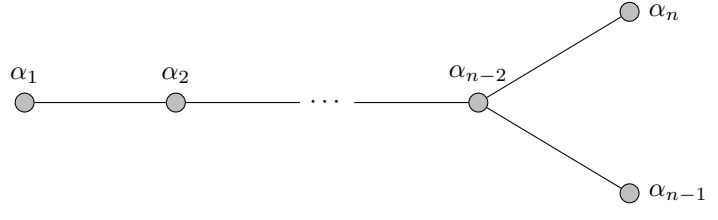

Let us fix $q_{i}=q^{\frac{(\alpha_{i},\alpha_{i})}{2}}$, with $q_{i}^{2}\neq 1$ for each $i=1,\dots, n$.  With the above datum we construct the quantised enveloping algebra $\mathrm{U}_{q}(\mathfrak{so}_{2\ell})$ precisely as in Section \ref{subsection:DJ_quantum_groups_and_flags}, and as the Lie algebra $\mathfrak{g}$ is simply laced, we have $q_{i}= q$ for all $i$'s and the $q$-Serre relations simplify as
\begin{equation}\nonumber
\begin{aligned}
	  E_{i}^{2}E_{j}-(q+q^{-1})E_{i}E_{j}E_{i}+E_{j}E_{i}^{2}& =0,  \\ 
	 F_{i}^{2}F_{j}-(q+q^{-1})F_{i}F_{j}F_{i}+F_{j}F_{i}^{2}& =0. 
\end{aligned}	
\end{equation}
Let us now fix a weight basis $\{e_{i}\}_{i=1}^{2\ell}$ of $V_{\varpi_{1}}$. The quantised coordinate algebra $\mathcal{O}_{q}(\mr{SO}_{2\ell})$ is generated by the elements $u_{ij}:=c^{\varpi_{1}}_{e^{i},e_{j}}$. The unique Hopf algebra structure on $\mathcal{O}_{q}(\mr{SO}_{2\ell})$  is given by the coproduct and counit

\[
\Delta(u_{ij}):=\sum_{k=1}^{2\ell}u_{ik}\otimes u_{kj},\quad \varepsilon(u_{ij}):=\delta_{ij}, 
\]  

\noindent and an antipode that we omit. Moreover, denoting by 
\[
\langle - , - \rangle \colon \mr{U}_{q}(\mathfrak{so}_{2\ell})\times \mathcal{O}_{q}(\mr{SO}_{2\ell})\to \mathbb{C}
\] 
the non-degenerate dual pairing between the two dual Hopf algebras, letting $1\leq i \leq 2\ell-1$ and denoting $i'=2\ell+1-i$, we have \cite{Kli-Sch}

\begin{equation}\label{eq:dual_pairing_D4}
\begin{aligned}
\langle E_{i},u_{i+1,i}\rangle = -\langle E_{i},u_{i',(i+1)'}\rangle = \langle F_{i},u_{i,i+1}\rangle = -\langle F_{i},u_{(i+1)',i'}\rangle & = 1, \\ 
-\langle E_{n},u_{n+2,n}\rangle = \langle E_{n},u_{n+1,n-1}\rangle = -\langle F_{n},u_{n,n+2}\rangle = \langle F_{n},u_{n-1,n+1}\rangle & = 1, \\ 
 \langle K_{i}^{-1},u_{i,i}\rangle = \langle K_{i},u_{i+1,i+1}\rangle =\langle K_{i}^{-1},u_{(i+1)',(i+1)'}\rangle = \langle K_{i},u_{i',i'}\rangle & = q, \\
 \langle K_{n}^{-1},u_{n-1,n-1}\rangle = \langle K_{n}^{-1},u_{n,n}\rangle = \langle K_{n},u_{n+1,n+1}\rangle = \langle K_{n},u_{n+2,n+2}\rangle & = q, \\
 \text{for} \quad  j\neq i,i+1,i',(i+1)',  \quad \langle K_{i}^{-1},u_{j,j}\rangle & = 1, \\ 
  \text{for} \quad  j\neq n-1,n,n+1,n+2,  \quad \langle K_{n},u_{j,j}\rangle & = 1.
\end{aligned}
\end{equation}

\subsection{Contact differential calculi on $\mathcal{O}_{q}(\mathrm{SO}_{8})$ and $\mathcal{O}_{q}(\mathrm{SO}_{8}/\mathrm{L}_{S})$}
\label{section:calculus_D4_flag}
 From now until the end of the Section we fix $\ell=4$ and consider the quantised enveloping and coordinate algebras associated to the Lie algebra $\mathfrak{so}_{8}$. The aim of this section is the study of differential calculi for the only non-irreducible quantised flag manifold corresponding to the choice of a single root for the quantised Levi subalgebra

\[
\mathrm{U}_{q}(\ell_{S}):=\left\langle \, K_{i}^{\pm 1},E_{j},F_{j}\mid i=1,\dots,n,\ j\in S \right\rangle\subseteq \mathrm{U}_{q}(\mathfrak{so}_{8}), 
\]

\noindent namely $S:=\Pi \backslash \{\alpha_{2}\}$. This is precisely the choice of simple roots corresponding to a contact grading of the lowest rank case of the Lie type $\mathrm{D}_{n}$. This choice may be represented graphically by a cross on the Dynkin diagram corresponding to the simple root we exclude from $\mathrm{U}_{q}(\ell_{S})$, namely
\medskip

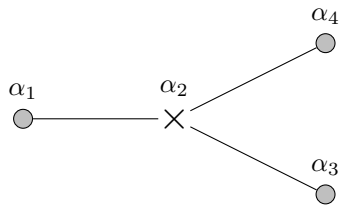
\[
\begin{tikzpicture}
  \tikzstyle{smallcircle} = [circle, draw=black, fill=lightgray, minimum size=2.5mm, inner sep=0pt]
  \tikzstyle{cross} = [minimum size=3mm, inner sep=0pt]
  \node[smallcircle, label=above:$\alpha_1$] (n1) at (0,0) {};
  \node[cross, label=above:$\alpha_2$] (n2) at (2,0) {\scalebox{1.5}{$\times$}};
  \node[smallcircle, label=above:$\alpha_3$] (n3) at (4,-1) {};
  \node[smallcircle, label=above:$\alpha_{4}$]    (n4) at (4,1) {};
  \draw (n1) -- (n2);
  \draw (n2) -- (n3);
  \draw (n2) -- (n4);
\end{tikzpicture}.
\] 

\medskip

\noindent As can already be deduced from the general picture described in Section \ref{subsection:DJ_quantum_groups_and_flags}, the quantised flag manifold $\mathcal{O}_{q}(\mathrm{SO}_{8}/\mathrm{L}_{S})$ is defined as the coideal subalgebra of invariant elements in $\mathcal{O}_{q}(\mathrm{SO}_{8})$ under the action of $\mathrm{U}_{q}(\ell_{S})$, namely 
\[
\mathcal{O}_{q}(\mathrm{SO}_{8}/\mathrm{L}_{S}) := {}^{\mathrm{U}_{q}(\ell_{S})}\mathcal{O}_{q}(\mathrm{SO}_{8})=\left\{ b\in \mathcal{O}_{q}(\mathrm{SO}_{8}) \mid X\triangleright b = b_{(1)}\langle X,b_{(2)}\rangle =\varepsilon(X)b, \ X\in \mathrm{U}_{q}(\ell_{S})\right\}. 
\]

Proceeding along the same lines as in Section \ref{section:contact_calculus_C2}, we now explicitly construct a left $\mathcal{O}_{q}(\mathrm{SO}_{8})$-covariant FODC for $\mathcal{O}_{q}(\mathrm{SO}_{8}/\mathrm{L}_{S})$. In particular, we generate a quantum tangent space $\underline{T}^{(0,1)}_{S}$ for $\mathcal{O}_{q}(\mathrm{SO}_{8})$ via the right adjoint action of $\mathrm{U}_{q}(\ell_{S})$ and consider its restriction to $\mathcal{O}_{q}(\mathrm{SO}_{8}/\mathrm{L}_{S})$, a procedure originally exploited by Heckenberger, I. and Kolb, S. in \cite{hek-kolb-2006} to construct a de Rham complex for the quantised flag manifolds in the irreducible setting.  A similar procedure is also employed in \cite{ale_re_jun} to construct covariant FODCi on the A-series quantised flags. 

\begin{prop}
A tangent space for $\mathcal{O}_{q}(\mathrm{SO}_{8})$ is given by 
\begin{equation}
\label{eq:tangent_space_for_D4_flag_upstairs}
\underline{T}^{(0,1)}_{S}:=E_{2}\triangleleft_{\mathrm{Ad}} \mathrm{U}_{q}(\ell_{S})=\Span_{\mathbb{C}}\Big\{E_{2}, [E_{2},E_{i}]_{q}, [[E_{2},E_{i}]_{q},E_{j}]_{q}, [[[E_{2},E_{i}]_{q},E_{j}]_{q},E_{k}]_{q}\Big\}. 
\end{equation}
Moreover, the latter restricts to the tangent space for $\mathcal{O}_{q}(\mathrm{SO}_{8}/\mathrm{L}_{S})$ given by  
 \begin{equation}
T^{(0,1)}_{S}:=\mathrm{U}_{q}(\ell_{S})E_{2}=\Span_{\mathbb{C}}\Big\{E_{2}, E_{i}E_{2},  E_{j}E_{i}E_{2}, E_{k}E_{j}E_{i}E_{2}\Big\}. 
\end{equation}
 where $i,j,k$ are pairwise distinct indices ranging in $S=\Pi/\{\alpha_{2}\}$. \begin{proof}
	The right adjoint action of the $K_{i}$-generators in $\mathrm{U}_{q}(\ell_{S})$ consists of a rescaling by an integer power of $q$, namely 
	
	\[
	E_{\alpha}\triangleleft K_{\beta} = K_{\beta}^{-1}E_{\alpha}K_{\beta} = q^{-\langle \alpha,\beta\rangle} E_{\alpha},  
	\]
	
	\noindent where with $E_{\alpha}$ and $K_{\beta}$ we denoted the generators of the quantised enveloping algebra corresponding to the positive $\alpha ,\beta$. Next, considering the action of the $E_{i}$-generators we find: 
	
 \begin{align*}
     E_{2}\triangleleft E_{i} & =-qE_{i}E_{2}+E_{2}E_{i} = [E_{2},E_{i}]_{q},  \\
		[E_{2},E_{i}]_{q}\triangleleft E_{j} & = [[E_{2},E_{i}]_{q},E_{j}]_{q}, \\
		[[E_{2},E_{i}]_{q},E_{j}]_{q} \triangleleft E_{k} & = [[[E_{2},E_{i}]_{q},E_{j}]_{q},E_{k}]_{q},  \\
		[[[E_{2},E_{i}]_{q},E_{j}]_{q},E_{k}]_{q}\triangleleft E_{\ell} & = 0, 	
	\end{align*}

\noindent for pairwise distinct $i,j,k\in\{1,3,4\}$ and arbitrary $\ell\in\{1,3,4\}$. It is immediate to deduce that other, higher contributions vanish owing to the $q$-Serre relations. Finally, consider the action of the $F_{i}$-generators:

\begin{align*}
	     E_{2}\triangleleft F_{i} & = -K_{i}F_{i}E_{2}+K_{i}E_{2}F_{i}=0,  \\ 
		[E_{2},E_{i}]_{q}\triangleleft F_{j} & =  \delta_{ij}E_{2}, \\
		[[E_{2},E_{i}]_{q},E_{j}]_{q}\triangleleft F_{k} & = (\delta_{kj}\delta_{hi}+\delta_{ki}\delta_{hj})[E_{2},E_{h}]_{q},  \\
		[[[E_{2},E_{i}]_{q},E_{j}]_{q},E_{k}]_{q}\triangleleft F_{\ell} & = (\delta_{\ell k}\delta_{hi}\delta_{sj}+\delta_{\ell j}\delta_{hi}\delta_{sk}+\delta_{\ell i}\delta_{hj}\delta_{sk})[[E_{2},E_{h}]_{q},E_{s}]_{q}. 
	\end{align*}
	
	Thus, the vector space in \ref{eq:tangent_space_for_D4_flag_upstairs} is a right $\mathrm{U}_{q}(\ell_{S})$-module.  We check that it realises a tangent space for $\mathcal{O}_{q}(\mathrm{SO}_{8})$ by evaluating the coproducts for its generators, for which we find

{\small \begin{align*}
\Delta([E_{2},E_{i}]_{q}) & = [E_{2},E_{i}]_{q}\otimes K_{2}K_{i}+(q^{-1}-q)E_{2}\otimes E_{i}K_{2} + 1\otimes [E_{2},E_{i}]_{q}, \\ \\
\Delta([[E_{2},E_{i}]_{q},E_{j}]_{q}) & = [[E_{2},E_{i}]_{q},E_{j}]_{q}\otimes K_{i}K_{j}K_{2} +(q^{-1}-q)[E_{2},E_{i}]_{q}\otimes E_{j}K_{2}K_{i} \\ & \quad + (q^{-1}-q)[E_{2},E_{j}]_{q}\otimes E_{i}K_{2}K_{j}  + (q^{-1}-q)^{2}E_{2}\otimes E_{i}E_{j}K_{2}\\ 
& \quad + 1\otimes [[E_{2},E_{i}]_{q},E_{j}]_{q}, \\ \\
\Delta([[[E_{2},E_{i}]_{q},E_{j}]_{q},E_{k}]_{q}) & = [[[E_{2},E_{i}]_{q},E_{j}]_{q},E_{k}]_{q}\otimes K_{2}K_{i}K_{j}K_{k} + 1\otimes [[[E_{2},E_{i}]_{q},E_{j}]_{q},E_{k}]_{q} \\ 
& \quad + (q^{-1}-q)[[E_{2},E_{i}]_{q},E_{j}]_{q}\otimes E_{k}K_{2}K_{i}K_{j} + (q^{-1}-q)[[E_{2},E_{i}]_{q},E_{k}]_{q} \otimes E_{j}K_{2}K_{i}K_{k}\\ 
& \quad + (q^{-1}-q)^{2}[E_{2},E_{i}]_{q}\otimes E_{k}E_{j}K_{2}K_{i} + (q^{-1}-q)[[E_{2},E_{j}]_{q},E_{k}]_{q}\otimes E_{i}K_{2}K_{j}K_{k} \\ 
& \quad +(q^{-1}-q)^{2}[E_{2},E_{j}]_{q}\otimes E_{k}E_{i}K_{2}K_{j} + (q^{-1}-q)^{2}[E_{2},E_{k}]_{q}\otimes E_{i}E_{j}K_{2}K_{k} \\ 
& \quad + (q^{-1}-q)^{3}E_{2}\otimes E_{i}E_{j}E_{k}K_{2}. 
\end{align*}}

 Finally, let $Y\in\mathcal{O}_{q}(\mathrm{SO}_{8}/\mathrm{L}_{S})^{\circ}$. As $X\triangleright b=\varepsilon(X)b$ for any $b\in\mathcal{O}_{q}(\mathrm{SO}_{8}/ \mathrm{L}_{S})$ and $X\in\mathrm{U}_{q}(\ell_{S})$, we have 

\[
\langle Y\varepsilon(X),b\rangle  = \langle Y,\varepsilon(X)b\rangle = \langle Y, X\triangleright b\rangle = \langle Y, b_{(1)}\langle X,b_{(2)}\rangle \rangle = \langle Y,b_{(1)}\rangle \langle X,b_{(2)}	\rangle = \langle YX,b\rangle, 
\]

\noindent and thus the tangent space $T^{(0,1)}_{S}$ restricts to a tangent space for $\mathcal{O}_{q}(\mathrm{SO}_{8}/ \mathrm{L}_{S})$, where each $q$-commutator $[E_{k},[E_{j},\ldots,[E_{i},E_{2}]_{q}]_{q}]_{q}$ is reduced to a scalar multiple of $E_{k}E_{j}\dots E_{i}E_{2}$. 
\end{proof}
\end{prop} 
\noindent We may represent the tangent space $\underline{T}^{(0,1)}_{S}$ graphically as follows, with the corresponding restriction to the $\mathcal{O}_{q}(\mathrm{SO}_{8}/\mathrm{L}_{S})$-tangent space $T^{(0,1)}_{S}$ similarly understood:
\begin{equation}\label{eq:tangent_space_project_flag}
\begin{tikzcd}[row sep=1.0cm, column sep=0.5cm,every arrow/.append style={-stealth}]
0 \arrow[r] & E_{2} \arrow[rr] \arrow[rd] \arrow[dd]            &                                            & {[E_{2},E_{3}]_{q}} \arrow[rd] \arrow[dd, dotted]  &                                                       &        \\
            &                                                   & {[E_{2},E_{4}]_{q}} \arrow[rr] \arrow[dd]  &                                                    & {[[E_{2},E_{4}]_{q},E_{3}]_{q}} \arrow[dd]            &        \\
            & {[E_{2},E_{1}]_{q}} \arrow[rr, dotted] \arrow[rd] &                                            & {[[E_{2},E_{3}]_{q},E_{1}]_{q}} \arrow[rd, dotted] &                                                       &        \\
            &                                                   & {[[E_{2},E_{4}]_{q},E_{1}]_{q}} \arrow[rr] &                                                    & {[[[E_{2},E_{4}]_{q},E_{3}]_{q},E_{1}]_{q}} \arrow[r] & {0. }
\end{tikzcd}
\end{equation}
From the last proposition and Theorems \ref{th:hermisson_class}--\ref{th:hk_tangent}, we have a left $\mathcal{O}_{q}(\mathrm{SO}_{8})$-covariant FODC on $\mathcal{O}_{q}(\mathrm{SO}_{8}/\mathrm{L}_{S})$ corresponding to the tangent space $T^{(0,1)}_{S}$, namely  
\begin{equation}
\label{eq:restriction_Lusztig_flag}
\left(\Omega^{(0,1)}(\mathrm{SO}_{8}/\mathrm{L}_{S}):=\mathcal{O}_{q}(\mathrm{SO}_{8})\square_{\pi_{B}} V_S^{(0,1)},\overline{\partial}\right), 
\end{equation}
where $V^{(0,1)}_S$ is the corresponding quantum cotangent space. Similarly, we have a left $\mathcal{O}_{q}(\mathrm{SO}_{8})$-covariant FODC at the level of the quantum group defined by

\begin{equation}
	\label{eq:framing_calculus}
\left(\underline{\Omega}^{(0,1)}(\mathrm{SO}_{8}):=\mathcal{O}_{q}(\mathrm{SO}_{8})\otimes \Lambda_S^{(0,1)},\overline{\partial}\right),
\end{equation}
with $\Lambda^{(0,1)}$ the corresponding quantum cotangent space. By dimensional arguments, there is a vector space isomorphism $V^{(0,1)}_{S} \stackrel{\cong}{\to} \Lambda^{(0,1)}_{S}$.  Equivalently, writing $\underline{\mathcal{I}}_{S}$ and $\mathcal{I}_S$ for the corresponding defining ideals, we have $\mathcal{I}=\underline{\mathcal{I}}_{S}\cap \mathcal{O}_{q}(\mathrm{SO}_{8}/\mathrm{L}_{S})^{+}$, and the canonical map
\begin{align*}
j:V_{S}^{(0,1)}& \to  \Lambda_{S}^{(0,1)},\\ 
 [b]_{V} &\mapsto [b]_{\Lambda},
\end{align*}
is an isomorphism. Thus, the map
\[
\left(\operatorname{id}_{A}\otimes j )\right|{\mathcal{O}_{q}(\mathrm{SO}_{8})\square_{\pi_B}V_{S}^{(0,1)}}\colon \Omega^{(0,1)}(\mathrm{SO}_{8}/\mathrm{L}_{S}) \to  \underline{\Omega}^{(0,1)}(\mathrm{SO}_{8})
\]
identifies the calculus on the flag with the restriction of the calculus on the quantum group. Moreover, by surjectivity, its image is automatically a right $\mathcal{O}_{q}(\mathrm{SO}_{8})$-submodule. Hence $\underline{\Omega}^{(0,1)}(\mathrm{SO}_{8})$ is a framing calculus for $\Omega^{(0,1)}(\mathrm{SO}_{8}/\mathrm{L}_{S})$, in the sense of \cite{ale_re_jun}. We now describe the quantum cotangent spaces corresponding to the above FODCi, giving an explicit description both in terms of the quantum group generators and the quantised flag generators.

\begin{lemma}
The generators of $\Lambda^{(0,1)}_{S}$ are given as in Table \eqref{table:root_vectors_frame_dual_pairing}.
\begin{proof}
	\noindent We write down the dual pairing between elements of the tangent space of Diagram \eqref{eq:tangent_space_project_flag} and the elements $u_{ij}\in\mathcal{O}_{q}(\mathrm{SO}_{8})$: exploiting the dual pairing of Hopf algebras again in Equation \eqref{eq:dual_pairing_D4}, we get the claim.
\end{proof}
\end{lemma}

 The general form for the generators of $\mathcal{O}_{q}(\mathrm{SO}_{8}/\mathrm{L}_{S})$ is discussed in Theorem \ref{th:generators_of_a_flag}. However, for such a low-rank example as $\mathrm{D}_{4}$, we find it more practical to use a more direct approach to provide a basis in terms of quantum minors. 

\begin{prop} 
The elements 
\begin{equation}
\label{eq:generators_of_so(8)_flag}
		\overline{z}_{ij}:=z_{ij}z_{12}^{\ast}:=\left(u_{i1}u_{j2}-q\,u_{i2}u_{j1}\right)\left(u_{77}u_{88}-q\, u_{78}u_{87}\right), \qquad i,j=1,\dots,8, 
	\end{equation}
	lie in $\mathcal{O}_{q}(\mathrm{SO}_{8}/\mathrm{L}_{S})$.
	
\begin{proof}
	After deducing the coproduct $
			\Delta(\overline{z}_{ij}) = u_{ia}u_{jb}u_{7c}u_{8d}\otimes z_{ab}\left(u_{c7}u_{d8} -q\,u_{c8}u_{d7}\right)$, the proof consists of a direct calculation using the $\mathrm{U}_{q}(\ell_{S})$-action on the elements $\overline{z}_{ij}$. We illustrate the procedure in a particular case, as the general calculation follows the same pattern. We find

	\[
	\begin{aligned}
	E_{1}\triangleright \overline{z}_{ij} & = ({z}_{ij})_{(1)}\langle E_{1},	({z}_{ij})_{(2)}\rangle \\ 
	& = u_{ia}u_{jb}u_{7c}u_{8d} \, \langle E_{1}, z_{ab}(u_{c7}u_{d8} -q\,u_{c8}u_{d7})\rangle \\ 
	& = u_{ia}u_{jb}u_{7c}u_{8d} \left( \langle E_{1}, z_{ab}\rangle \langle K_{1},(u_{c7}u_{d8} -q\,u_{c8}u_{d7})\rangle + \langle 1, z_{ab}\rangle \langle E_{1},(u_{c7}u_{d8} -q\,u_{c8}u_{d7})\rangle\right) \\ 
	& = u_{ia}u_{jb}u_{7c}u_{8d} \left( \langle E_{1}, u_{a1}u_{b2}-q\, u_{a2}u_{b1}\rangle \langle K_{1},(u_{c7}u_{d8} -q\,u_{c8}u_{d7})\rangle\right)\\ 
	& =0, 
	\end{aligned}
	\]
	
	\noindent and thus $E_{1}\triangleright \overline{z}_{ij}=\varepsilon(E_{1})\,\overline{z}_{ij}$. The first inclusion now follows after repeating the procedure with the remaining generators of $\mathrm{U}_{q}(\ell_{S})$.
\end{proof}
\end{prop}

	\noindent In the next proposition we provide bases for the cotangent spaces $\Lambda^{(0,1)}_{S}$ and $V^{(0,1)}_{S}$  considering respectively the dual pairing of Hopf algebras $\langle -,-\rangle\colon\mathcal{O}_{q}(\mathrm{SO}_{8})\times \mathrm{U}_{q}(\mathfrak{so}_{8}) \to \mathbb{C}$ and the corresponding restriction to a dual pairing  $\langle -,-\rangle\colon\mathcal{O}_{q}(\mathrm{SO}_{8}/\mathrm{L}_{S})\times \mathcal{O}_{q}(\mathrm{SO}_{8}/\mathrm{L}_{S})^{\circ} \to \mathbb{C}$ at the level of the quantised flag manifold and its dual coalgebra $\mathcal{O}_{q}(\mathrm{SO}_{8}/\mathrm{L}_{S})^{\circ}\subset \mathrm{U}_{q}(\mathfrak{so}_{8})$. 
	
	\begin{prop}\label{prop:cotangent_space_flag}
	Let $\Lambda^{(0,1)}_{S}$ be the dual space to $\underline{T}_{S}^{(0,1)}$ with respect to the non-degenerate  dual pairing $\langle -,-\rangle\colon \mathrm{U}_{q}(\mathfrak{so}_{8})\times  \mathcal{O}_{q}(\mathrm{SO}_{8}) \to \mathbb{C}$, and let $V^{(0,1)}_{S}$ be the dual space to $T_{S}^{(0,1)}$ with respect to the non-degenerate  dual pairing $\langle -,-\rangle\colon \mathcal{O}_{q}(\mathrm{SO}_{8}/\mathrm{L}_{S})\times \mathcal{O}_{q}(\mathrm{SO}_{8}/\mathrm{L}_{S})^{\circ} \to \mathbb{C}$. Table \ref{table:root_vectors_dual_pairing_flag_frame} provides two bases for the latter.

\begin{proof}
The proof is a calculation using the dual pairing of Hopf algebras and the corresponding restriction to $\left(\mathcal{O}_{q}(\mathrm{SO}_{8}/\mathrm{L}_{S})\right)^{\circ}\times \mathcal{O}_{q}(\mathrm{SO}_{8}/\mathrm{L}_{S})\to \mathbb{C}$. For example, we have 
		 \begin{align*}
	 	\langle E_{1}E_{2},\tilde{z}_{ij}\rangle & = \langle E_{1}E_{2}, u_{i1}u_{j2}u_{77}u_{88} -q \, u_{i1}u_{j2}u_{78}u_{87} -q\, u_{i2}u_{j1}u_{77}u_{88} +q^{2}u_{i2}u_{j1}u_{78}u_{87}\rangle \\ 
	 	& = \langle E_{1}, u_{ia}u_{jb}u_{7c}u_{8d}\rangle \langle E_{2}, u_{a1}u_{b2}u_{c7}u_{d8} -q \, u_{a1}u_{b2}u_{c8}u_{d7} -q\, u_{a2}u_{b1}u_{c7}u_{d8} +q^{2}u_{a2}u_{b1}u_{c8}u_{d7}\rangle \\ 
	 	&= q\,\langle E_{1}, u_{i1}u_{j3}u_{77}u_{88}\rangle-q^{2}\langle E_{1}, u_{i1}u_{j3}u_{78}u_{87}\rangle-q^{2}\langle E_{1}, u_{i3}u_{j1}u_{77}u_{88}\rangle+q^{3}\langle E_{1}, u_{i3}u_{j1}u_{78}u_{87}\rangle\\ 
	 	&= q\,\delta_{i2}\delta_{j3} -q^{2}\delta_{i3}\delta_{j2}. 
	 \end{align*}
Now, as $\langle E_{2}E_{1},\tilde{z}_{23}\rangle =0$, we find $\langle [E_{2},E_{1}]_{q},q^{-3}\tilde{z}_{32}\rangle=1$, and in the same way $\langle [E_{2},E_{1}]_{q},-q^{-2}\tilde{z}_{23}\rangle =1.$ With analogous calculations we deduce the remaining pairings as in Table \ref{table:root_vectors_dual_pairing_flag_frame}.
\end{proof}
\end{prop}
\subsection{The maximal prolongation of $\Omega^{(0,1)}\left(\mathrm{SO}_{8}/\,\mathrm{L}_{S}\right)$}
\label{section:exterior_D4_flag}
In this Section, we explicitly describe the relations for the maximal prolongation of the FODC

\[
\left(\Omega^{(0,1)}\left(\mathrm{SO}_{8}/\,\mathrm{L}_{S}\right):=\mathcal{O}_{q}(\mathrm{SO}_{8})\square_{\pi_{B}} V_{S}^{(0,1)},\overline{\partial}\right)
\]

\noindent on the contact $\mathcal{O}_{q}(\mathrm{SO}_{8})$ quantised flag manifold. The relations are obtained by applying the general reconstruction principle for maximal prolongations of FODCs over quantum homogeneous spaces, expressed in terms of the bimodule map of a torsion-free bimodule connection, as already presented in Section \ref{section:contact_calculus_C2}.

In the case of $\mathcal{O}_{q}(\mathrm{SO}_{8})$, we again use the existence of a distinguished covariant connection on each irreducible component of the Heckenberger–Kolb type calculus, which yields a torsion-free bimodule connection on the relevant subcalculi. The non-irreducibility of the corresponding quantised flag manifold manifests itself through the decomposition of the antiholomorphic tangent space into two irreducible $\mathrm{U}_{q}(\ell_{S})$-submodules.

Our analysis will focus first on the irreducible component generated by the action of $\mathrm{U}_{q}(\ell_{S})$ on the root vector associated to a fixed simple root omitted from the quantised Levi subalgebra, in direct analogy with the classical construction. In addition, we isolate the remaining irreducible submodule $T^{(0,1)}_{\theta}$  associated with the highest root of $\mathfrak{so}_{8}$. 

We start by specialising the result of Theorem \ref{theorem:relations_max_prolongation_flag} to the case under study. In particular, we write down an explicit formula for the bimodule map in this setting.
\begin{prop}\label{prop:bimodule_map_cotangent_space_flag}
	 Let $e_{\alpha}, \ \alpha\in\Delta_{2}^{+}$ denote the generators for the anti-holomorphic cotangent space $V_{S}^{(0,1)}$ of $\mathcal{O}_{q}(\mathrm{SO}_{8}/\mathrm{L}_{S})$ as in Proposition \ref{prop:cotangent_space_flag}. The bimodule map on generators reads 
	 
	 \begin{equation}\label{eq:sigma_on_cotangent_space}
	\begin{aligned}
		\sigma(e_{\alpha}\otimes e_{\beta})& = \sum_{a,b=1}^{8}q^{a+b-i-j-1} [e_{\beta}]\triangleleft u_{b'j'}u_{a'i'}\otimes [e_{\gamma_{ab}}], \qquad \sigma^{-1}(e_{\alpha}\otimes e_{\beta}) & =q\sum_{a,b=1}^{8}[{e}_{ab}]\otimes [e_{\alpha}]\triangleleft (u_{ia}u_{jb}). 
	\end{aligned}
	\end{equation}
	
\begin{proof} We prove the statement by direct calculation, during which summation over repeated indices is understood, finding

\begin{align*}
\sigma(e_{\alpha}\otimes e_{\beta}) & = q^{-2}\sigma([\tilde{z}_{ij}]\otimes [\tilde{z}_{k\ell}]) \\ 
& = q^{-1} [e_{\beta}] \triangleleft S((\tilde{z}_{ij})_{(1)}) \otimes [(\tilde{z}_{ij})_{(2)}] \\
& =q^{-1} [e_{\beta}] \triangleleft S(u_{ia}u_{jb}u_{7c}u_{8d})\otimes [(u_{a1}u_{b2}-q\, u_{a2}u_{b1})(u_{c7}u_{d8}-q\, u_{c8}u_{d7})]\\  
& = q^{-1}[e_{\beta}]\triangleleft S(u_{8d})S(u_{7c})S(u_{jb})S(u_{ia}) \otimes [(u_{a1}u_{b2}-q\, u_{a2}u_{b1})(u_{c7}u_{d8}-q\, u_{c8}u_{d7})]\\ 
& = q^{-1} [e_{\beta}]\triangleleft S(u_{8d})S(u_{7c})S(u_{jb})S(u_{ia}) \otimes [(u_{a1}u_{b2}-q\, u_{a2}u_{b1})(\delta_{c7}\delta_{d8}z^{*}_{12}+\delta_{c8}\delta_{d7}z^{*}_{21})]\\
& = q^{-1} [e_{\beta}]\triangleleft S(u_{8d})S(u_{7c})S(u_{jb})S(u_{ia}) \otimes [\delta_{c7}\delta_{d8}\tilde{z}_{ab}+\delta_{c8}\delta_{d7}\bar{z}_{ab}]\\
& = q^{-1}[e_{\beta}]\triangleleft \left(S(u_{88})S(u_{77})S(u_{jb})S(u_{ia}) \otimes [\tilde{z}_{ab}] + S(u_{87})S(u_{78})S(u_{jb})S(u_{ia}) \otimes [\bar{z}_{ab}]\right)\\
& = [e_{\beta}]\triangleleft S(u_{88})S(u_{77})S(u_{jb})S(u_{ia}) \otimes [e_{\gamma_{ab}}]\\
& = q^{a+b-i-j}[e_{\beta}]\triangleleft u_{11}\,u_{22}\,u_{b'j'}\,u_{a'i'}\otimes [e_{\gamma_{ab}}]\\  
& = q^{a+b-i-j}\langle K_{\beta},u_{11}\,u_{22}\rangle [e_{\beta}]\triangleleft u_{b'j'}\,u_{a'i'}\otimes [e_{\gamma_{ab}}]\\
& = q^{a+b-i-j-1} [e_{\beta}]\triangleleft u_{b'j'}\,u_{a'i'}\otimes [e_{\gamma_{ab}}]. 
\end{align*} 

In particular, we used the fact that a non-vanishing right action $e_\alpha \triangleleft u_{ij}$ requires $i>j$, and that $\langle K_{\beta},u_{11}u_{22}\rangle=q^{-1}$ for any $\beta$ in the positive root span of $\alpha_{2}$. 

We proceed similarly for the inverse $\sigma^{-1}$, finding
\begin{align*}
\sigma^{-1}(e_{\alpha}\otimes e_{\beta}) & = q^{-1}\sigma^{-1}([e_{\alpha}]\otimes [\tilde{z}_{ij}]) \\  
& = q^{-1} [(\tilde{z}_{ij})_{(2)}]\otimes [e_{\alpha}]\triangleleft (\tilde{z}_{ij})_{(1)} \\
& = q^{-1} [(u_{a1}u_{b2}-q\, u_{a2}u_{b1})(u_{c7}u_{d8}-q\, u_{c8}u_{d7})]\otimes [e_{\alpha}]\triangleleft (u_{ia}u_{jb}u_{c7}u_{8d})\\ 
& = q\,[e_{ab}]\otimes [e_{\alpha}]\triangleleft (u_{ia}u_{jb}). 
\end{align*}
\end{proof}
\end{prop}

In view of the last proposition, the calculation for the relations of the maximal prolongation of $\Omega^{(0,1)}(\mathrm{SO}_{8}/\mathrm{L}_{S})$ is essentially reduced to the evaluation of Equation \eqref{eq:sigma_on_cotangent_space}.

An alternative and somewhat more convenient approach is again the one discussed in Section \ref{section:C2}, where instead of calculating the relations for the exterior algebra at the level of the cotangent space, we do the same at the level of the quantum tangent space. Using Lemma \ref{lemma:bimodule_map_tangent_space_sigma}, we calculate the relations for $\Omega^{(0,1)}(\mathrm{SO}_{8}/\mathrm{L}_{S})$ by the same method used in Proposition \ref{prop:relations_C2_flag_S} for $\Omega^{(0,1)}(\mathrm{Sp}_{4}/\mathrm{L}_{S})$. 

\noindent Recall that  $\widecheck\sigma$ denotes the dual version of the bimodule map $\sigma$ as in Equation \eqref{eq:sigma_check}. 

\begin{lemma}\label{lemma:sigma_magic_formula}
Consider $T_{S}^{(0,1)}:=\mathrm{U}_{q}(\ell_{S})E_{2}$. For any $Y\in\mathrm{U}_{q}(\ell_{S})$ and $E_{\alpha},E_{\gamma},E_{\delta}$ the elements corresponding to root-lattice weights $\alpha,\gamma,\delta$ such that $E_\alpha E_2,E_\gamma E_2,E_\delta E_2$ are root vectors in $T_{S}^{(0,1)}$, we have 
	\begin{equation}
	\widecheck{\sigma}(E_{\alpha}E_{2}\otimes YE_{2})  = \sum_{\gamma+\delta=\alpha} (q^{-1}-q)^{h(\gamma)}E_{\gamma}K_{2}^{-1}K_{\alpha}^{-1}YE_{2}\otimes E_{\delta}E_{2}, 
\end{equation}
 where $h(\gamma)$ is the height of $\gamma$. 
\begin{proof}
As we are dealing with a finite-dimensional tangent space, the proof is easily carried out by a direct calculation. The case $E_{\alpha}=1$ is clear. For the remaining combinations, we exploit that, generally speaking, given any $X,Y,Z\in\mathrm{U}_{q}(\ell_{S})$ it holds that 

\[
\widecheck{\sigma}(XYE_{2}\otimes ZE_{2})  = X_{(1)}\triangleright\left(Y_{(1)}K_{2}^{-1}S(Y_{(3)})S(X_{(2)})ZE_{2}\otimes Y_{(2)}E_{2}\right). 
\]

\noindent We have

\begin{align*}
\widecheck{\sigma}(E_{i}E_{2}\otimes YE_{2}) & = (q^{-1}-q)(E_{i}K_{2}^{-1}K_{i}^{-1}YE_{2}\otimes E_{2})+K_{2}^{-1}K_{i}^{-1}YE_{2}\otimes E_{i}E_{2}, \\ \\
\widecheck{\sigma}(E_{i}E_{j}E_{2}\otimes YE_{2}) & = (q^{-1}-q)(E_{j}K_{2}^{-1}K_{j}^{-1}K_{i}^{-1}YE_{2}\otimes E_{i}E_{2}+E_{i}K_{2}^{-1}K_{j}^{-1}K_{i}^{-1}YE_{2}\otimes E_{j}E_{2} \\  
&\quad + (q^{-1}-q)E_{i}E_{j}K_{2}^{-1}K_{j}^{-1}K_{i}^{-1}YE_{2}\otimes E_{2}) + K_{2}^{-1}K_{j}^{-1}K_{i}^{-1}YE_{2}\otimes E_{i}E_{j}E_{2}, \\ \\
\widecheck{\sigma}(E_{i}E_{j}E_{k}E_{2}\otimes YE_{2}) & =(q^{-1}-q)^{3}E_{i}E_{k}E_{j}K_{2}^{-1}K_{k}^{-1}K_{j}^{-1}K_{i}^{-1}YE_{2}\otimes E_{2}\\ 
		&\quad +(q^{-1}-q)^{2}E_{k}E_{j}K_{2}^{-1}K_{k}^{-1}K_{j}^{-1}K_{i}^{-1}YE_{2}\otimes E_{i}E_{2}\\ 
		&\quad +(q^{-1}-q)^{2}E_{i}E_{k}K_{2}^{-1}K_{k}^{-1}K_{j}^{-1}K_{i}^{-1}YE_{2}\otimes E_{j}E_{2}\\
		&\quad +(q^{-1}-q)^{2}E_{i}E_{j}K_{2}^{-1}K_{k}^{-1}K_{j}^{-1}K_{i}^{-1}YE_{2}\otimes E_{k}E_{2}\\  
		&\quad +(q^{-1}-q)E_{j}K_{2}^{-1}K_{k}^{-1}K_{j}^{-1}K_{i}^{-1}YE_{2}\otimes E_{i}E_{k}E_{2}\\ 
		&\quad +(q^{-1}-q)E_{k}K_{2}^{-1}K_{k}^{-1}K_{j}^{-1}K_{i}^{-1}YE_{2}\otimes E_{i}E_{j}E_{2}\\  
		&\quad +(q^{-1}-q)E_{i}K_{2}^{-1}K_{k}^{-1}K_{j}^{-1}K_{i}^{-1}YE_{2}\otimes E_{j}E_{k}E_{2}\\ 
		&\quad +K_{2}^{-1}K_{k}^{-1}K_{j}^{-1}K_{i}^{-1}YE_{2}\otimes E_{i}E_{j}E_{k}E_{2}, \\ 
		\end{align*}
		giving the claim.
\end{proof}
\end{lemma}

The quantised exterior algebra for this example closely resembles the exterior algebra of the antiholomorphic de Rham complex of the smooth manifold $\mathrm{SO}_{8}/\mathrm{L}_{S}$, but includes corrections involving the deformation parameter $q$ that arise as a purely quantum phenomenon. In the next proposition, we calculate the corresponding full set of relations. The triality symmetry of the $\mathrm{D}_{4}$ Dynkin diagram considerably simplifies the evaluation of such relations, which exhibit symmetries among the outer nodes. This behaviour is of course very specific to the contact quantised flag manifold studied here and, for example, the irreducible quantised flag manifold of the same Lie type yields relations featuring less symmetry.
 
\begin{prop} \label{prop:relations_for_D4}
Let $\alpha,\beta\in\Delta_{2}^{+}$ and let $e_{\alpha},e_{\beta}$ be the corresponding basis elements in $V^{(0,1)}_{S}$. The relations for the maximal prolongation of $\Lambda^{(0,1)}_{S}$ corresponding to the generating sets $G_{i}$, $i=1,2$ of Theorem \ref{theorem:relations_max_prolongation_flag}, are of the following form: 
\begin{equation}
	\begin{aligned}
		(1+q^{-2}) \, e_{\alpha}\wedge e_{\alpha} &=0,\\ 
		e_{\alpha}\wedge e_{\beta}+q^{-\langle \alpha,\beta\rangle}e_{\beta}\wedge e_{\alpha}&=0,\\ 
	e_{\alpha}\wedge e_{\beta}+q^{\langle \alpha,\beta\rangle}e_{\beta}\wedge e_{\alpha} + \sum_{\gamma+\delta=\alpha+\beta}\nu_{\gamma\delta}\, e_{\gamma}\wedge e_{\delta}&=0. 
\end{aligned}
\end{equation}
\begin{proof} 
We calculate the relations exploiting the results of Theorem \ref{theorem:relations_max_prolongation_flag} and Proposition \ref{prop:bimodule_map_cotangent_space_flag}. The first identity is clear according to the structure of $\Delta_{2}^{+}$. To deduce the other relations, we use Lemma \ref{lemma:sigma_magic_formula}. To simplify notation, we write $e_{\alpha_{i}}:=e_{i}$ and take $i,j,k\in\{1,3,4\}$, with $i,j,k$ pairwise distinct. The relations are divided into three categories:
\begin{enumerate}
	\item The first set of relations consists of anticommutation relationships between the generators of $V^{(0,1)}_{S}$, where a prefactor corresponding to the inner product of the underlying roots appears

\begin{align*}
e_{2}\wedge e_{2+i}+q^{-1}e_{i+2}\wedge e_{2}&=0,  \\ 
e_{2}\wedge e_{2+i+j} + e_{2+i+j}\wedge e_{2}&=0,   \\ 
e_{2}\wedge  e_{2+i+j+k}+ q\,e_{2+i+j+k}\wedge  e_{2}&=0,  \\ 
e_{2+i}\wedge  e_{2+i+j}+ q^{-1}e_{2+i+j}\wedge  e_{2+i}&=0,  \\ 
e_{2+i}\wedge e_{2+i+j+k}+e_{2+i+j+k}\wedge  e_{2+i}&=0, 	\\ 
e_{2+i+j}\wedge  e_{2+i+j+k}+ q^{-1}e_{2+i+j+k}\wedge  e_{2+i+j}&=0.  
\end{align*}

\item The second set of relations features an extra factor $\nu:=q^{-1}-q$, and again a prefactor corresponding to the inner product of the underlying roots, namely

\begin{align*}
e_{2+j}\wedge  e_{2+i} +e_{2+i}\wedge  e_{2+j}+\nu \, e_{2+i+j}\wedge  e_{2} &=0, \\  
e_{2+i+j}\wedge  e_{2+k}+q^{-1}e_{2+k}\wedge  e_{2+i+j}+ \nu\,  e_{2+i+j+k}\wedge e_{2} &=0,  \\ 
e_{2+i+j}\wedge  e_{2+i+k}+e_{2+i+k}\wedge  e_{2+i+j}+ \nu\,  e_{2+i+j+k}\wedge  e_{2+i}&=0. \\   
\end{align*}

\item  Finally, the third set consists of more involved relations owing to the multiple possibilities with which the positive root corresponding to both sides of the tensor factor can be decomposed. We obtain 
\begin{align*}
e_{2+i+j}\wedge e_{2+k} +q\,e_{2+k}\wedge e_{2+i+j} +q\nu\bigl( e_{2+i+k}\wedge e_{2+j}+e_{2+j+k}\wedge e_{2+i} +\nu\,e_{2+i+j+k}\wedge e_{2}\bigr)=0.\\ 
e_{2}\wedge e_{2+i+j+k}+ q^{-1}e_{2+i+j+k}\wedge e_{2}+\nu(e_{2+i}\wedge e_{2+j+k} + e_{2+j}\wedge e_{2+i+k} + e_{2+k}\wedge e_{2+i+j})\\
+ \nu^{2}(e_{2+j+k}\wedge e_{2+i} + e_{2+i+k}\wedge e_{2+j} + e_{2+i+j}\wedge e_{2+k}) \\ 
+ \nu^{3}e_{2+i+j+k}\wedge e_{2}& = 0. 
\end{align*}
Notice how the more involved relations are the only ones featuring the root content corresponding to a positive root in $\Delta_{\alpha_{2}}^{+}$, which in particular happens to correspond to the highest root of the underlying Lie algebra. 
\end{enumerate}

To determine the remaining relations arising from $G_2$, recall that $\bigl(\Omega^{(0,1)}(\mathrm{SO}_8/\mathrm L_S),\overline{\partial}\bigr)$ is the restriction to $\mathcal O_q(\mathrm{SO}_8/\mathrm L_S)$ of the FODC $\bigl(\underline{\Omega}^{(0,1)}(\mathrm{SO}_8),\overline{\partial}\bigr)$. In the present case, the restriction map on the corresponding quantum tangent spaces is surjective; in particular it is an isomorphism of vector spaces at the level of the quantum cotangent spaces $j\colon V_S^{(0,1)}\rightarrow\Lambda_S^{(0,1)}.$  Tensoring this inclusion yields an isomorphism $j^{\otimes 2}\colon V_S^{(0,1)}\otimes V_S^{(0,1)}\rightarrow\Lambda_S^{(0,1)}\otimes\Lambda_S^{(0,1)}$.  Using this identification, we apply the argument of Section~\ref{section:contact_calculus_C2} to compute the remaining quadratic relations. Specifically, to obtain the relations associated with $G_2$ of Theorem \ref{theorem:relations_max_prolongation_flag}, it is enough to evaluate the map $\varpi^{\otimes_{2}}$ on a given set of generators of $\mathcal{O}_{q}(\mathrm{SO}_{8}/\mathrm{L}_{S})$ as an associative unital algebra (see the discussion in Section \ref{section:contact_calculus_C2} or more in detail \cite[Lemma 3.5]{Re-Jun-Ale-io}). Here we limit our attention to the elements $\tilde{z}_{ij}$, and specialise the map $\varpi^{\otimes_{2}}$ to the case of interest as
\begin{equation}
	\varpi^{\otimes_{2}}(z_{ij})=[u_{ia}^{+}u_{jb}]\otimes [\tilde{z}_{ab}^{+}]+[u_{jb}^{+}]\otimes [\tilde{z}_{ib}^{+}], 
\end{equation} 

\noindent where the square brackets denote cosets in the framing cotangent space $\Lambda^{(0,1)}_{S}$. Evaluating on all the elements $\tilde{z}_{ij}$ we find the extra relations

\begin{align*}
 e_2\wedge e_{2+i+j}+q^{-2}e_{2+i+j}\wedge e_2 -q^{-1}\bigl(e_{2+i}\wedge e_{2+j}+e_{2+j}\wedge e_{2+i}\bigr)&=0,\\
e_{2+i}\wedge e_{2+i+j+k}+q^{-2}e_{2+i+j+k}\wedge e_{2+i}-q^{-1}\bigl(e_{2+i+j}\wedge e_{2+i+k}+e_{2+i+k}\wedge e_{2+i+j}\bigr) &=0,\\
e_2\wedge e_{2+i+j+k}+q^{-2}e_{2+i+j}\wedge e_{2+k} -q^{-1}\bigl(e_{2+i}\wedge e_{2+j+k}+e_{2+j}\wedge e_{2+i+k}\bigr)&=0,\\
 e_{2+k}\wedge e_{2+i+j}+q^{-2}e_{2+i+j+k}\wedge e_2 -q^{-1}\bigl(e_{2+i+k}\wedge e_{2+j}+q^{-2}e_{2+j+k}\wedge e_{2+i}\bigr)&=0.
\end{align*}

Together, the relations presented above span a $46$-dimensional space $\mathcal I^{(2)}$. Thus, already in degree two, the resulting graded algebra
\[
V_S^{(0,\bullet)}:=\mathcal T\!\bigl(V_S^{(0,1)}\bigr)/\langle\mathcal I^{(2)}\rangle
\]
has smaller dimension than the classical exterior algebra, whose second homogeneous component has dimension $\binom{8}{2}=28$. In particular, the quadratic relations yield a degree-two space of dimension $18$ rather than $28$.
 \end{proof}
\end{prop}

Again, we observe the same phenomenon as in Section~\ref{section:contact_calculus_C2}: for calculi constructed solely from the contact distribution of the underlying classical manifold, the associated exterior algebra exhibits a drop in dimension relative to its classical counterpart. 

\subsection{A Reeb extended FODC for $\mathcal{O}_{q}(\mathrm{SO}_{8}/\mathrm{L}_{S})$}
\label{section:tangent_so8}

Following the approach of Section \ref{section:reeb_calculus_C2}, we extend the eight-dimensional tangent space constructed in the previous section by adjoining the direction associated classically with the Reeb vector field. This additional direction corresponds to the highest root $\theta$ of $\mathfrak{so}_{8}$ and spans a one-dimensional irreducible $\mathrm{U}_{q}(\ell_{S})$-module, denoted by $T^{(0,1)}_{\theta}$. The resulting nine-dimensional space $T^{(0,1)}:=T^{(0,1)}_{S}\oplus T^{(0,1)}_{\theta}$ is the quantum analogue of the full anti-holomorphic tangent space of the contact flag $\mathcal{O}_{q}(\mathrm{SO}_{8}/\mathrm{L}_{S})$. The requirement that $T^{(0,1)}_{\theta}$ be one-dimensional reflects the classical contact grading, in which the corresponding component $\mathfrak{g}_{2}$ is one-dimensional. We therefore seek a single generator whose transformation under the right adjoint action of $\mathrm{U}_{q}(\ell_{S})$ realises the required module structure. The next proposition gives the general form of such a generator.
	
\begin{prop} 
There exists a unique nine-dimensional quantum tangent space for $\mathcal O_q(\mathrm{SO}_8/\mathrm{L}_S)$ which decomposes, as a $\mathrm{U}_q(\mathfrak \ell_S)$-module, into the eight-dimensional irreducible module $T_S^{(0,1)}$ and a one-dimensional irreducible module given by the complex span of
	\begin{align*}
	 E_{\theta}&= E_{2}E_{4}E_{3}E_{1}E_{2}  +2(q+q^{-1})^{-3}E_{4}E_{3}E_{1}E_{2}^{2}\\ 
	 & \quad -(q+q^{-1})^{-1}\left(E_{4}E_{2}E_{3}E_{1}E_{2} + E_{1}E_{2}E_{4}E_{3}E_{2} + E_{3}E_{2}E_{4}E_{1}E_{2}\right).
	 	\end{align*}
 We call the latter $T^{(0,1)}$.
 \begin{proof}
The verification that $T^{(0,1)}$ is a $9$-dimensional quantum tangent space follows by the same argument as in Proposition \ref{prop:tangent_space_with_theta_C2}. To prove uniqueness within the class under consideration, let $T$ be another such tangent space. The Levi-invariance and coideal conditions force $T$ to contain $T_S^{(0,1)}$ and force its remaining one-dimensional summand to be $\mathbb C E_\theta$. Hence $T=T^{(0,1)}$.
\end{proof}
\end{prop}

 Letting $\mathcal{I}_{\theta}$ be the right $\mathcal{O}_{q}(\mathrm{SO}_{8}/\mathrm{L}_{S})$-ideal in $\mathcal{O}_{q}(\mathrm{SO}_{8}/\mathrm{L}_{S})^{+}$ corresponding to the tangent space $T^{(0,1)}$ according to Theorem \ref{th:hermisson_class}, we have a left $\mathcal{O}_{q}(\mathrm{SO}_{8})$-covariant FODC on $\mathcal{O}_{q}(\mathrm{SO}_{8}/\mathrm{L}_{S})$ as 

\begin{equation}
\left({\Omega}^{(0,1)}_{\theta}(\mathrm{SO}_{8}/\mathrm{L}_{S}):=\mathcal{O}_{q}(\mathrm{SO}_{8})\square_{\pi_{B}} V^{(0,1)},\overline{\partial}\right), 
\end{equation}

\noindent where $V^{(0,1)}:=\mathcal{O}_{q}(\mathrm{SO}_{8}/\mathrm{L}_{S})^{+}/\,\mathcal{I}_{\theta}$ is the associated cotangent space.

The differential calculus constructed above does not belong to the category ${}^{\mathcal{O}_{q}(\mathrm{SO}_{8})}_{\mathcal{O}_{q}(\mathrm{SO}_{8}/\mathrm{L}_{S})}\mathrm{Mod}_{0}$. Consequently, determining the relations of its maximal prolongation requires us to follow the approach outlined in Section \ref{section:contact_adapted_calculus_C2}, which relies on the choice of a suitable framing calculus. As in the $\mathrm{C}_{2}$ case, a central step in this procedure is the construction of an explicit lift of the root vector $E_{\theta}$. While this was computationally feasible in that case, carrying out an analogous computation in the present setting has so far proved elusive. This difficulty points to the need for a more general strategy for constructing such lifts, whose development we leave to future work.

Nevertheless, the problem of describing a geometrically coherent exterior algebra for the flag manifold under consideration remains. With this motivation, we explore an alternative approach, following the construction of \cite{Re-Sombrg} based on Lusztig root vectors. The authors construct covariant first order differential calculi for quantised flag manifolds of type $\mathrm{A}$ whose maximal prolongations have classical dimension. Their results suggest that Lusztig root vectors may also provide a useful starting point in the present setting, provided that their choice is compatible with the root vector $E_{\theta}$ appearing in our construction.

Since Lusztig root vectors depend on the choice of a reduced expression for the longest element of the Weyl group, the compatibility required above leads to the following question:
\begin{quote}
Does there exist a reduced expression for the longest element of the Weyl group of type $\mathrm{D}_{4}$ such that the corresponding root vectors span a quantum tangent space and the root vector associated to the highest root restricts to the prescribed $E_{\theta}$?
\end{quote}
We investigate this question in what follows, beginning with a brief review of the construction of \cite{Re-Sombrg} and the role played by the choice of reduced expression. 

 Let $\mathcal{B}$ denote the braid group associated to the root system $\mathrm{D}_{4}$, that is, the group generated by elements $s_{\alpha_{1}},\dots,s_{\alpha_{4}}$ subject to the relations 
 \begin{equation}
 	\begin{aligned}
 		s_{\alpha_{i}}s_{\alpha_{2}}s_{\alpha_{i}} & = s_{\alpha_{2}}s_{\alpha_{i}}s_{\alpha_{2}}, \\
 		s_{\alpha_{i}}s_{\alpha_{j}} & = s_{\alpha_{j}}s_{\alpha_{i}}, \\
 	\end{aligned}
 \end{equation} 
 where $\alpha_{i}\in\Pi:=\{\alpha_{1},\alpha_{3},\alpha_{4}\}$. It is a well-known result of Lusztig \cite{Lusz88} that there is a map
 
 \begin{equation}
 \begin{aligned}
\mathcal{B}_{\mathfrak{so}_{8}} & \to \mr{End}_{\mr{Alg}}\left(\mr{U}_{q}(\mathfrak{so}_{8})\right), \\
 s_{\alpha_i} & \mapsto \mathcal{T}_{\alpha_i}=:\mathcal{T}_{i}, 
 \end{aligned}
 \end{equation}
 
 \noindent of the braid group into the group of algebra automorphisms of the quantised enveloping algebra, defining an action $\mathcal{T}\colon\mathcal{B}\otimes \mathrm{U}_{q}(\mathfrak{so}_{8}) \to \mathrm{U}_{q}(\mathfrak{so}_{8})$. In particular, given $k\in\mathbb{C}$ and letting  $[X,Y]_{k}:=XY-k\, YX$ denote the $k$-commutator of elements in $\mr{U}_{q}(\mathfrak{so}_{8})$, the action on the $K$-generators is  $\mathcal{T}_{i}(K_{j}) = K_{j}K_{i}^{-a_{ij}}$, that on the $E,F$-generators is
 
 \begin{align*}
 		\mathcal{T}_{i}(E_{i}) &=-F_{i}K_{i}, & \mathcal{T}_{i}(E_{j})& = - [ E_{i},E_{j}]_{q^{-1}},  & & \text{if} \ a_{ij}=-1, & \mathcal{T}_{i}(E_{j})& =E_{j}, & &  \text{if} \ a_{ij}=0,  \notag \\	  
 		\mathcal{T}_{i}(F_{i})&=-K_{i}^{-1}E_{i}, & \mathcal{T}_{i}(F_{j})& = - [F_{j},F_{i}]_{q} ,  & &  \text{if} \ a_{ij}=-1, & \mathcal{T}_{i}(F_{j})& =F_{j}, & &  \text{if} \ a_{ij}=0. 
        \notag  
 \end{align*}
 
 We define root elements for $\mathrm{U}_{q}(\mathfrak{so}_{8})$ given arbitrary roots of $\mathfrak{so}_{8}$ using the action of the braid group. Consider a fixed reduced decomposition $w_{0}=w_{i_{1}}\cdots w_{i_{n}}$ of the longest element of the Weyl group $\mathcal{W}$. As all positive roots of the Lie algebra are exhausted by reflections with respect to the simple roots, we define the \emph{Lusztig root vectors}, or just root vectors, of $\mathrm{U}_{q}(\mathfrak{so}_{8})$ corresponding to the positive roots $\beta_{r}$ and $-\beta_{r}$ respectively as
\begin{equation}
\label{eq:definition_of_root_vectors}
	E_{\beta_{r}}:=\mathcal{T}_{i_{1}}\cdots \mathcal{T}_{i_{r-1}}(E_{i_{r}}), \qquad F_{\beta_{r}} := \mathcal{T}_{i_{1}}\cdots \mathcal{T}_{i_{r-1}}(F_{i_{r}}).
\end{equation}

To exemplify the procedure, consider the \emph{maximal orthogonal decomposition} (see \cite{STEKOLSHCHIK23,Uli}) for the longest element of $\mathcal{W}$
\begin{equation}
	w_{\mathrm{max}}:=s_{\alpha_4}s_{\alpha_3}s_{\alpha_2}s_{\alpha_4}s_{\alpha_3}s_{\alpha_2}s_{\alpha_1}s_{\alpha_2}s_{\alpha_4}s_{\alpha_3}s_{\alpha_2}s_{\alpha_1}. 
\end{equation} 
and the corresponding positive root vectors. A crucial property (see \cite{Lus}, page 262) that we make repeated use of in the next proposition is that, for any pair of adjacent nodes in the Dynkin diagram the following relation holds
\[
\mathcal{T}_{i}\,\mathcal{T}_{j}(E_{i})=E_{j}. 
\]  
Furthermore, if the property $w(\alpha_{i})=\alpha_{j}$ holds at the level of the Lie algebra, for any element $w\in \mathcal{W}$ and any simple roots $\alpha_{i},\alpha_{j}$, then the corresponding action of the braid group on the quantised enveloping algebra reads $\mathcal{T}_{w}(E_{i})=E_{j}$ (see \cite{Kli-Sch} page 176, Proposition 23).

\begin{prop}\label{proposition:root_vectors_of_D4}
The reduced decomposition $w_{\mathrm{max}}$ of the longest element in $\mathcal{W}$ provides the Lusztig root vectors 
\[
\begin{aligned}
E_{4}, &  & E_{3}, &  & E_{2}, &  &E_{1},\\ 
 -[E_{2},E_{1}]_{q^{-1}},& &-[E_{3},E_{2}]_{q^{-1}}, & & -[E_{4},E_{2}]_{q^{-1}}, & & [E_{4},[E_{3},E_{2}]_{q^{-1}}]_{q^{-1}},\\ 
 [[E_{4},E_{2}]_{q^{-1}},E_{1}]_{q^{-1}},& & [[E_{3},E_{2}]_{q^{-1}},E_{1}]_{q^{-1}}, & &  -[[E_{4},[E_{3},E_{2}]_{q^{-1}}]_{q^{-1}},E_{1}]_{q^{-1}}, & & [[[E_{2},E_{4}]_{q^{-1}},E_{3}]_{q^{-1}},[E_{2},E_{1}]_{q^{-1}}]_{q^{-1}}.
\end{aligned}
\] 

	\begin{proof}
	The proof consists of a direct calculation using the definition of root vectors as in Equation \eqref{eq:definition_of_root_vectors} and exploiting the properties of the action of $\mathcal{B}$.
	\end{proof}	
\end{prop}

It turns out that the root vectors in the last proposition supply a quantum tangent space, which we will call the \emph{Lusztig tangent space}, for $\mathcal{O}_{q}(\mathrm{SO}_8)$. The latter is identified, under a suitable limit procedure, with the anti-holomorphic component of the complexified tangent bundle of the complex manifold underlying the full flag manifold of type $\mathrm{D}_{4}$. 
\begin{prop}
\label{proposition:tangent_space_D4}
The complex span of the root vectors in Proposition \ref{proposition:root_vectors_of_D4} provides a quantum tangent space $\underline{T}^{(0,1)}_{\mathcal{L}}$ for the quantised coordinate algebra $\mathcal{O}_{q}(\mathrm{SO}_8)$. 
\begin{proof} 
As the counit is an algebra morphism vanishing on the $E_{i}$-generators of $\mathrm{U}_{q}(\mathfrak{so}_{8})$ we are only left to check that $T^{(0,1)}_{\mathcal{L}}\oplus \mathbb{C}1$ is a right coideal of $\mathrm{U}_{q}(\mathfrak{so}_{8})$. We proceed through a direct check of the latter. The condition is clearly verified for the positive root vectors corresponding to simple roots. Next, consider elements featuring at least a $q^{-1}$-commutator. Owing to the triality symmetry of the $\mathrm{D}_{4}$ Dynkin diagram, it is enough to consider a single representative in terms of the number of $q^{-1}$-commutators. We find

		\begin{align*}
		       \Delta[E_{2},E_{i}]_{q^{-1}} & = [E_{2},E_{i}]_{q^{-1}}\otimes K_{i}K_{2} + \nu E_{i}\otimes K_{i}E_{2} +  1\otimes [E_{2},E_{i}]_{q^{-1}},
       \\ 
       \Delta[E_{i},E_{2}]_{q^{-1}} & = [E_{i},E_{2}]_{q^{-1}}\otimes K_{i}K_{2} + \nu E_{2}\otimes K_{2}E_{i} +  1\otimes [E_{i},E_{2}]_{q^{-1}}, \\ \\ 
		\Delta[E_{i},[E_{j},E_{2}]_{q^{-1}}]_{q^{-1}} & =  [E_{i},[E_{j},E_{2}]_{q^{-1}}]_{q^{-1}}\otimes K_{2}K_{i}K_{j} + \nu [E_{j},E_{2}]_{q^{-1}}\otimes K_{j}K_{2}E_{i} \\
		&\quad + 1\otimes [E_{i},[E_{j},E_{2}]_{q^{-1}}]_{q^{-1}} + \nu [E_{i},E_{2}]_{q^{-1}}\otimes K_{2}K_{i}E_{j},\\  
       		\Delta[[E_{i},E_{2}]_{q^{-1}},E_{j}]_{q^{-1}} & = [[E_{i},E_{2}]_{q^{-1}},E_{j}]_{q^{-1}}\otimes K_{2}K_{i}K_{j} + \nu [E_{2},E_{j}]_{q^{-1}}\otimes K_{2}K_{j}E_{i} \\ 
		&\quad + \nu E_{j}\otimes K_{j}[E_{i},E_{2}]_{q^{-1}}+  1\otimes [[E_{i},E_{2}]_{q^{-1}},E_{j}]_{q^{-1}}, \\ \\
		\Delta[[E_{i},[E_{j},E_{2}]_{q^{-1}}]_{q^{-1}},E_{k}]_{q^{-1}} & = [[E_{i},[E_{j},E_{2}]_{q^{-1}}]_{q^{-1}},E_{k}]_{q^{-1}}\otimes K_{i}K_{j}K_{k}K_{2} \\
		&\quad + \nu [[E_{j},E_{2}]_{q^{-1}},E_{k}]_{q^{-1}}\otimes K_{j}K_{k}K_{2}E_{i} \\ 
		&\quad + \nu E_{k}\otimes K_{k}[E_{i},[E_{j},E_{2}]_{q^{-1}}]_{q^{-1}} \\
		&\quad + 1\otimes [[E_{i},[E_{j},E_{2}]_{q^{-1}}]_{q^{-1}},E_{k}]_{q^{-1}} \\ 
		&\quad + \nu [[E_{i},E_{2}]_{q^{-1}},E_{k}]_{q^{-1}}\otimes K_{i}K_{k}K_{2}E_{j}, \\ \\
		\Delta[[[E_{2},E_{i}]_{q^{-1}},E_{j}]_{q^{-1}},[E_{2},E_{k}]_{q^{-1}}]_{q^{-1}} & = [[[E_{2},E_{i}]_{q^{-1}},E_{j}]_{q^{-1}},[E_{2},E_{k}]_{q^{-1}}]_{q^{-1}} \otimes K_{2}^{2}K_{k}K_{j}K_{i} \\ 
		& \quad + \nu [E_{j},[E_{2},E_{k}]_{q^{-1}}]_{q^{-1}}\otimes K_{k}K_{2}K_{j}[E_{2},E_{i}]_{q^{-1}} \\
		& \quad + \nu E_{j}\otimes K_{j}[[E_{2},E_{i}]_{q^{-1}},[E_{2},E_{k}]_{q^{-1}}]\\
		& \quad + 1\otimes [[[E_{2},E_{i}]_{q^{-1}},E_{j}]_{q^{-1}},[E_{2},E_{k}]_{q^{-1}}]_{q^{-1}} \\ 
		&\quad + \nu E_{k}\otimes K_{k}(q[[E_{2},E_{i}]_{q^{-1}},E_{j}]_{q^{-1}}E_{2}-q^{-1}E_{2}[[E_{2},E_{i}]_{q^{-1}},E_{j}]_{q^{-1}}) \\
				& \quad + \nu [E_{i},[E_{2},E_{k}]_{q^{-1}}]_{q^{-1}}\otimes K_{i}K_{2}K_{k}[E_{2},E_{j}]_{q^{-1}}\\ 
				&\quad +\nu [[E_{i},[E_{j},E_{2}]_{q^{-1}}]_{q^{-1}},E_{k}]_{q^{-1}} \otimes K_{i}K_{j}K_{2}K_{k}E_{2}, 
	\end{align*}
	
\noindent where $\nu=q-q^{-1}$ and where $i,j,k$ are distinct indices in $\{1,3,4\}$. 
	
\end{proof}
\end{prop}

\begin{rmk}
It turns out that the tangent space $\underline{T}^{(0,1)}_{\mathcal L}$ does not restrict to $T^{(0,1)}$. More in detail, it is easy to see  that the part corresponding to the contact distribution, namely the one generated by the adjoint $\mathrm{U}_{q}(\ell_{S})$-action on the crossed node, restricts well. However, expanding the root vector corresponding to the highest root of the underlying Lie algebra, projecting to the flag, and applying the $q$-Serre relations to carry $E_{2}$ to the rightmost part of the single summands gives 
\[
\begin{aligned}
E'_{\theta} & = -q^{-2}(q+q^{-1})E_2E_iE_jE_kE_2+q^{-2}E_iE_2E_jE_kE_2 +q^{-2}E_jE_2E_iE_kE_2\\
&\quad +q^{-4}E_kE_2E_iE_jE_2 -\frac{q^{-3}}{q+q^{-1}}E_iE_jE_kE_2^2,
\end{aligned}
\]
which is not a scalar multiple of $E_\theta$. Moreover, as it may be checked by direct calculation, $\mathrm{U}_{q}(\ell_{S})E'_{\theta}$ does not supply a 1-dimensional module. 
 \end{rmk}

\begin{prop}
No Lusztig tangent space of $\mathcal{O}_q(\mathrm{SO}_8)$ restricts to the quantum tangent space $T^{(0,1)}$. 
\begin{proof}

 By the characterization of Carotenuto, Hohlweg, and Papi \cite{CHP}, precisely $15$ commutation classes of reduced decompositions of the longest Weyl group element in type $\mathrm D_4$ give rise to quantum tangent spaces. Following the lines of the previous remark, a direct inspection of the associated Lusztig root vectors shows that none of these tangent spaces restricts to $T^{(0,1)}$.
	\end{proof}
\end{prop}

The preceding proposition shows that the Lusztig construction does not recover the nine-dimensional tangent space $T^{(0,1)}$ by restriction to the contact flag manifold, paralleling the obstruction found in type $\mathrm C_2$ in Section \ref{section:reeb_calculus_C2}. Although the contact directions restrict correctly, the highest-root vector fails to generate the required one-dimensional $\mathrm U_q(\ell_S)$-submodule. A framing calculus for this extension must therefore be obtained analysing possibilities that go beyond the class of Lusztig tangent spaces.

 \section*{Acknowledgements}
The author would like to thank Reamonn \'{O} Buachalla and Thomas Weber for their guidance throughout the development of this paper. The author is also grateful to Julius Benner, Arnab Bhattacharjee and Emanuele Latini for many fruitful discussions. Special thanks go to Alessandro Carotenuto and Giovanni Gava for the insightful comments and suggestions. This publication is partially supported by HORIZON-MSCA-2022-SE-01-01 CaLIGOLA and UNCE/24/SCI/022. The author acknowledges the use of \href{https://tikzcd.yichuanshen.de/}{https://tikzcd.yichuanshen.de/} and ChatGPT in preparing some of the diagrams and tables.

\appendix

\clearpage

\section{Quantum tangent spaces and dual pairings}
\label{appendix}

We collect the tables of quantum vectors and dual pairings used throughout the paper. The tables bring together the data for the flag calculi and their framing calculi, facilitating comparison between the constructions.

\begingroup
\newcommand{\pairingzero}{\textcolor{black!45}{$0$}}
\newcommand{\pairingdash}{\textcolor{black!45}{---}}

\begin{table}[h]
\centering
\small
\renewcommand{\arraystretch}{1.5}
\setlength{\tabcolsep}{3pt}
\arrayrulecolor{flagblue!60}

\begin{tabular*}{\textwidth}{
  @{\extracolsep{\fill}}lcccccc@{}
}
\toprule

\rowcolor{flagblue!8}
\multicolumn{7}{l}{
  \strut\textcolor{flagblue}{
    \textbf{(a) Pairings with matrix generators}
  }
}\\
\addlinespace[3pt]

\textbf{Tangent vector}
& $u_{21}$ & $u_{43}$ & $u_{31}$
& $u_{42}$ & $u_{41}$ & $u_{32}$ \\
\cmidrule(l){2-7}

$E_1$
& $1$ & $-1$
& \pairingzero & \pairingzero
& \pairingzero & \pairingzero \\

$E_2$
& \pairingzero & \pairingzero
& \pairingzero & \pairingzero
& \pairingzero & $1$ \\

$[E_1,E_2]_{q^2}$
& \pairingzero & \pairingzero
& $-q^2$ & $-1$
& \pairingzero & \pairingzero \\

$[[E_2,E_1]_{q^2},E_1]_{q^{-2}}$
& \pairingzero & \pairingzero
& \pairingzero & \pairingzero
& $q^2+q^{-2}$ & \pairingzero \\

\addlinespace[5pt]
\midrule

\rowcolor{contactpurple!8}
\multicolumn{7}{l}{
  \strut\textcolor{contactpurple}{
    \textbf{(b) Pairings with flag elements}
  }
}\\
\addlinespace[3pt]

\textbf{}
& $z_{43}$ & $z_{42}$
& \multicolumn{2}{c}{$z_{41}$}
& \multicolumn{2}{c}{$u_{32}S(u_{22})$} \\
\cmidrule(l){2-7}

\multicolumn{7}{l}{
  \textcolor{contactpurple}{
    \itshape Contact-flag calculus
  }
}\\

$E_1$
& $q$ & \pairingzero
& \multicolumn{2}{c}{\pairingzero}
& \multicolumn{2}{c}{\pairingdash} \\

$E_2E_1$
& \pairingzero & $-q^3$
& \multicolumn{2}{c}{\pairingzero}
& \multicolumn{2}{c}{\pairingdash} \\

$\begin{gathered}
  E_2E_1^2-(q^2+q^{-2})E_1E_2E_1
\end{gathered}$
& \pairingzero & \pairingzero
& \multicolumn{2}{c}{$-q^4(q^2+q^{-2})$}
& \multicolumn{2}{c}{\pairingdash} \\

\addlinespace[5pt]

\multicolumn{7}{l}{
  \textcolor{flagblue}{
    \itshape Framing and full-flag calculi
  }
}\\

$E_1$
& $q$ & \pairingzero
& \multicolumn{2}{c}{\pairingzero}
& \multicolumn{2}{c}{\pairingzero} \\

$E_2$
& \pairingzero & \pairingzero
& \multicolumn{2}{c}{\pairingzero}
& \multicolumn{2}{c}{$q^2$} \\

$[E_1,E_2]_{q^2}$
& \pairingzero & $q^5$
& \multicolumn{2}{c}{\pairingzero}
& \multicolumn{2}{c}{\pairingzero} \\

$[[E_2,E_1]_{q^2},E_1]_{q^{-2}}$
& \pairingzero & \pairingzero
& \multicolumn{2}{c}{$-q^4(q^2+q^{-2})$}
& \multicolumn{2}{c}{\pairingzero} \\

\bottomrule
\end{tabular*}

\caption{Pairings for the type $\mathrm C_2$ calculi. Panel~(a) lists the nonzero pairings with matrix generators. Panel~(b) compares the contact-flag tangent vectors with their framing counterparts and includes the additional full-flag direction detected by $u_{32}S(u_{22})$. Dashes indicate that this element does not belong to the contact flag.}

\label{table:C2_pairings}

\label{table:root_vectors_dual_pairing_C2_up}
\label{table:root_vectors_dual_pairing_C2}
\label{table:root_vectors_dual_pairing_C2_full}
\label{table:root_vectors_dual_pairing_total_C2}

\end{table}
\endgroup

 \begingroup
\small
\renewcommand{\arraystretch}{1.4}
\setlength{\tabcolsep}{2.5pt}
\setlength{\LTleft}{\fill}
\setlength{\LTright}{\fill}
\setlength{\LTcapwidth}{\textwidth}
\arrayrulecolor{flagblue!60}

\newcommand{\Dfourpairs}[4]{%
  $\begin{gathered}
    #1\mapsto #2\\[-2pt]
    #3\mapsto #4
  \end{gathered}$%
}
\newcommand{\Dfoursection}[2]{%
  \rowcolor{#1!8}
  \multicolumn{5}{l}{%
    \strut\textcolor{#1}{\itshape #2}%
  }\\*[4pt]
}

\begin{longtable}{@{}
  >{\raggedright\arraybackslash}m{0.14\textwidth}
  >{\raggedright\arraybackslash}m{0.28\textwidth}
  >{\centering\arraybackslash}m{0.16\textwidth}
  >{\centering\arraybackslash}m{0.17\textwidth}
  >{\centering\arraybackslash}m{0.17\textwidth}
@{}}
\toprule
\endfirsthead

\toprule
\endhead

\midrule
\multicolumn{5}{r}{\footnotesize\textcolor{flagblue}{Continued on the next page}}\\
\endfoot

\bottomrule
\caption{Root vectors and dual pairings in type $\mathrm D_4$. An entry $a\mapsto c$ means that the pairing of the indicated vector with $a$ equals $c$. Panel~(a) lists the Lusztig root vectors associated with $w_{\mathrm{max}}$ and their matrix pairings. Panel~(b) compares the contact-flag and framing vectors; matrix pairings in that panel refer to the framing vector.}
\label{table:D4_pairings}
\label{table:root_vectors_of_D4}
\label{table:root_vectors_frame_dual_pairing}
\label{table:root_vectors_dual_pairing_flag_frame}
\label{table:root_vectors_dual_pairing_total}\\
\endlastfoot

\Dfoursection{flagblue}{(a) Lusztig root vectors and matrix pairings}
\multicolumn{3}{l}{\textbf{Root vector}}
& \multicolumn{2}{c}{\textbf{Matrix pairings}}\\*
\cmidrule(lr){1-3}\cmidrule(l){4-5}

\multicolumn{3}{>{\raggedright\arraybackslash}m{\dimexpr0.58\textwidth+10pt\relax}}{
  $E_4$
}
& $u_{53}\mapsto 1$ & $u_{64}\mapsto -1$ \\

\multicolumn{3}{>{\raggedright\arraybackslash}m{\dimexpr0.58\textwidth+10pt\relax}}{
  $E_3$
}
& $u_{43}\mapsto 1$ & $u_{65}\mapsto -1$ \\

\multicolumn{3}{>{\raggedright\arraybackslash}m{\dimexpr0.58\textwidth+10pt\relax}}{
  $E_2$
}
& $u_{32}\mapsto 1$ & $u_{76}\mapsto -1$ \\

\multicolumn{3}{>{\raggedright\arraybackslash}m{\dimexpr0.58\textwidth+10pt\relax}}{
  $E_1$
}
& $u_{21}\mapsto 1$ & $u_{87}\mapsto -1$ \\

\addlinespace[2pt]

\multicolumn{3}{>{\raggedright\arraybackslash}m{\dimexpr0.58\textwidth+10pt\relax}}{
  $-[E_2,E_1]_{q^{-1}}$
}
& $u_{31}\mapsto -1$ & $u_{86}\mapsto q^{-1}$ \\

\multicolumn{3}{>{\raggedright\arraybackslash}m{\dimexpr0.58\textwidth+10pt\relax}}{
  $-[E_3,E_2]_{q^{-1}}$
}
& $u_{42}\mapsto -1$ & $u_{75}\mapsto q^{-1}$ \\

\multicolumn{3}{>{\raggedright\arraybackslash}m{\dimexpr0.58\textwidth+10pt\relax}}{
  $-[E_4,E_2]_{q^{-1}}$
}
& $u_{52}\mapsto -1$ & $u_{74}\mapsto q^{-1}$ \\

\addlinespace[2pt]

\multicolumn{3}{>{\raggedright\arraybackslash}m{\dimexpr0.58\textwidth+10pt\relax}}{
  $[E_4,[E_3,E_2]_{q^{-1}}]_{q^{-1}}$
}
& $u_{62}\mapsto -1$ & $u_{73}\mapsto q^{-2}$ \\

\multicolumn{3}{>{\raggedright\arraybackslash}m{\dimexpr0.58\textwidth+10pt\relax}}{
  $[[E_4,E_2]_{q^{-1}},E_1]_{q^{-1}}$
}
& $u_{51}\mapsto 1$ & $u_{84}\mapsto -q^{-2}$ \\

\multicolumn{3}{>{\raggedright\arraybackslash}m{\dimexpr0.58\textwidth+10pt\relax}}{
  $[[E_3,E_2]_{q^{-1}},E_1]_{q^{-1}}$
}
& $u_{41}\mapsto 1$ & $u_{85}\mapsto -q^{-2}$ \\

\addlinespace[2pt]

\multicolumn{3}{>{\raggedright\arraybackslash}m{\dimexpr0.58\textwidth+10pt\relax}}{
  $\begin{aligned}
    &-[[E_4,[E_3,E_2]_{q^{-1}}]_{q^{-1}}, E_1]_{q^{-1}}
  \end{aligned}$
}
& $u_{61}\mapsto -1$ & $u_{83}\mapsto q^{-3}$ \\

\multicolumn{3}{>{\raggedright\arraybackslash}m{\dimexpr0.58\textwidth+10pt\relax}}{
  $\begin{aligned}
    &[[[E_2,E_4]_{q^{-1}},E_3]_{q^{-1}},[E_2,E_1]_{q^{-1}}]_{q^{-1}}
  \end{aligned}$
}
& $u_{71}\mapsto 1$ & $u_{82}\mapsto -q^{-4}$ \\
\Dfoursection{contactpurple}{(b) Contact-flag and framing calculi}
\multicolumn{2}{c}{\textbf{Tangent vectors}}
& \multicolumn{3}{c}{\textbf{Pairings}}\\*
\cmidrule(lr){1-2}\cmidrule(l){3-5}
{\footnotesize\itshape Contact flag}
& {\footnotesize\itshape Framing}
& {\footnotesize\itshape Matrix elements\newline(framing)}
& {\footnotesize\itshape Flag elements\newline(contact flag)}
& {\footnotesize\itshape Flag elements\newline(framing)}\\*
\midrule
$E_2$ & $E_2$
& \Dfourpairs{u_{32}}{1}{u_{76}}{-1}
& \Dfourpairs{\tilde z_{13}}{q}{\tilde z_{31}}{-q^2}
& \Dfourpairs{\tilde z_{13}}{q}{\tilde z_{31}}{-q^2}
\\[3pt]
$E_1E_2$ & $[E_2,E_1]_q$
& \Dfourpairs{u_{31}}{1}{u_{86}}{-q}
& \Dfourpairs{\tilde z_{23}}{q}{\tilde z_{32}}{-q^2}
& \Dfourpairs{z_{32}}{q^3}{z_{23}}{-q^2}
\\[3pt]
$E_3E_2$ & $[E_2,E_3]_q$
& \Dfourpairs{u_{75}}{1}{u_{42}}{-q}
& \Dfourpairs{\tilde z_{14}}{q}{\tilde z_{41}}{-q^2}
& \Dfourpairs{z_{41}}{q^3}{z_{14}}{-q^2}
\\[3pt]
$E_4E_2$ & $[E_2,E_4]_q$
& \Dfourpairs{u_{74}}{1}{u_{52}}{-q}
& \Dfourpairs{\tilde z_{15}}{q}{\tilde z_{51}}{-q^2}
& \Dfourpairs{z_{51}}{q^3}{z_{15}}{-q^2}
\\[3pt]
$E_3E_4E_2$ & $[[E_2,E_4]_q,E_3]_q$
& \Dfourpairs{u_{73}}{1}{u_{62}}{-q^2}
& \Dfourpairs{\tilde z_{61}}{q^2}{\tilde z_{16}}{-q}
& \Dfourpairs{\tilde z_{61}}{q^4}{\tilde z_{16}}{-q^3}
\\[3pt]
$E_4E_1E_2$ & $[[E_2,E_4]_q,E_1]_q$
& \Dfourpairs{u_{84}}{q}{u_{51}}{-q}
& \Dfourpairs{\tilde z_{25}}{q}{\tilde z_{52}}{-q^2}
& \Dfourpairs{\tilde z_{25}}{q^3}{\tilde z_{52}}{-q^4}
\\[3pt]
$E_3E_1E_2$ & $[[E_2,E_3]_q,E_1]_q$
& \Dfourpairs{u_{85}}{q}{u_{41}}{-q}
& \Dfourpairs{\tilde z_{24}}{q}{\tilde z_{42}}{-q^2}
& \Dfourpairs{\tilde z_{24}}{q^3}{\tilde z_{42}}{-q^4}
\\[3pt]
$E_3E_4E_1E_2$ & $\begin{aligned}&[[[E_2,E_4]_q,E_3]_q,E_1]_q\end{aligned}$
& \Dfourpairs{u_{83}}{q}{u_{61}}{-q^2}
& \Dfourpairs{\tilde z_{62}}{q^2}{\tilde z_{26}}{-q}
& \Dfourpairs{\tilde z_{26}}{q^4}{\tilde z_{62}}{-q^5}
\\[3pt]
\end{longtable}
\endgroup

\end{document}